\documentclass[11pt, reqno]{amsart}

\usepackage{textcmds}  % amsrefs needs this, but it has to be loaded early to avoid re-defs.
\usepackage{amsmath, amssymb, amsfonts, amstext, verbatim, amsthm, mathrsfs}
\usepackage[mathcal]{eucal}
\usepackage{microtype}
\usepackage[all]{xy}
\usepackage[modulo]{lineno}
\usepackage[usenames]{color}
\usepackage{aliascnt}
\usepackage{enumitem}
\usepackage{xspace}
\usepackage{amsfonts}
\usepackage{amssymb}
\usepackage[centertags]{amsmath}
\usepackage{amsthm}
\usepackage[margin=3cm]{geometry}
\usepackage{dsfont}
\usepackage{bm}
\usepackage{color}
\usepackage{subfigure}
\usepackage{amsmath}
\usepackage{array}
\usepackage[all]{xy}

\usepackage{graphics,graphicx}  %% comment for "draft" mode w/o pictures
\let\oldtocsection=\tocsection

\let\oldtocsubsection=\tocsubsection

\renewcommand{\tocsection}[2]{\hspace{0em}\oldtocsection{#1}{#2}}
\renewcommand{\tocsubsection}[2]{\hspace{1em}\oldtocsubsection{#1}{#2}}

\usepackage[backref,colorlinks=true,linkcolor=blue,citecolor=blue,urlcolor=blue,citebordercolor={0 0 1},urlbordercolor={0 0 1},linkbordercolor={0 0 1}]{hyperref} %needs to be loaded after most things

\newtheorem{thm}{Theorem}[section]

\newtheorem{theorem}[thm]{Theorem}

\newtheorem{proposition}[thm]{Proposition}

\newtheorem{lemma}[thm]{Lemma}

\newtheorem{corollary}[thm]{Corollary}

\newtheorem{definition}[thm]{Definition}

\newtheorem{remark}[thm]{Remark}

\newtheorem*{thm*}{Theorem}
\newtheorem*{cor*}{Corollary}
\newtheorem*{prop*}{Proposition}

\newcommand{\e}{\mathbf{e}}
\newcommand{\x}{\mathbf{x}}
\newcommand{\y}{\mathbf{y}}
\renewcommand{\aa}{\mathbf{a}}
\newcommand{\bb}{\mathbf{b}}
\newcommand{\cc}{\mathbf{c}}

\newcommand{\T}{\mathcal{T}}
\newcommand{\Z}{\mathbb{Z}}
\newcommand{\A}{\mathcal{A}}
\newcommand{\B}{\mathcal{B}}

\newcommand{\R}{\mathbb{R}}

\newcommand{\C}{\mathbb{C}}

\newcommand{\D}{\mathcal{D}}
\newcommand{\Fuk}{\mathcal{F}uk}
\newcommand{\Aoo}{A_\infty}
\newcommand{\F}{\mathcal{F}}
\newcommand{\E}{\mathcal{E}}
\newcommand{\tC}{\text{C}}

\renewcommand{\P}{\mathbb{P}}
\renewcommand{\O}{\mathcal{O}}

\newcommand{\cP}{\mathcal{P}}
\newcommand{\cQ}{\mathcal{Q}}
\newcommand{\cR}{\mathcal{R}}
\newcommand{\Coh}{\text{Coh}}
\newcommand{\Ext}{\text{Ext}}
\newcommand{\Perf}{\text{Perf}}
\newcommand{\perf}{\text{perf}}
\newcommand{\twvect}{\text{tw vect}}
\newcommand{\ve}{\text{vect}}
\renewcommand{\H}{\mathbb{H}}

\newcommand{\vect}{\text{vect}}

\newcommand{\bdm}{\begin{displaymath}}
\newcommand{\edm}{\end{displaymath}}

\newcommand{\bq}{\begin{equation}}
\newcommand{\eq}{\end{equation}}

\numberwithin{equation}{section}

\title{Homological Mirror Symmetry for hypersurface cusp singularities}
\author{Ailsa Keating}
\thanks{Partially supported by NSF grant DMS-1505798 and by a Junior Fellow award from the Simons Foundation.}

\begin{document}

\begin{abstract}
We study versions of homological mirror symmetry for hypersurface cusp singularities and the three hypersurface simple elliptic singularities. 
We show that the Milnor fibres of each of these carries a distinguished Lefschetz fibration; its derived directed Fukaya category is equivalent to the derived category of coherent sheaves on a smooth rational surface $Y_{p,q,r}$. 
By using localization techniques on both sides, we get an isomorphism between the derived wrapped Fukaya category of the Milnor fibre and the derived category of coherent sheaves on a quasi-projective surface given by deleting an anti-canonical divisor $D$ from $Y_{p,q,r}$. 
In the cusp case, the pair $(Y_{p,q,r}, D)$ is naturally associated to the dual cusp singularity, tying into Gross, Hacking and Keel's proof of Looijenga's conjecture.
\end{abstract}

\maketitle

\tableofcontents

\section{Introduction}

A landmark application of the field of mirror symmetry is the recent proof by Gross, Hacking and Keel of Looijenga's conjecture, about pairs of cusp singularities \cite{GHK_logCY}. Cusp singularities come in naturally dual pairs;  in \cite{GHK_logCY}, this duality gets strengthened to a mirror symmetry statement, of the flavour developed by Gross--Siebert (e.g.~\cite{GS1, GS2, GS3}) and Kontsevich--Soibelman (e.g.~\cite{KS1, KS2}). In particular, all of the invariants involved in \cite{GHK_logCY} belong to the world of algebraic geometry. In this paper, we prove versions of Kontsevich's Homological Mirror Symmetry Conjecture \cite{Kontsevich_ICM} for spaces appearing in Gross, Hacking and Keel's work. 

We will consider Floer-theoretic invariants associated to the following singularities:\bq
T_{p,q,r}(x,y,z) = x^p + y^q +z^r + axyz
\eq
where  $(p, q, r)$ is a triple of positive integers with 
\bq \frac{1}{p} + \frac{1}{q} + \frac{1}{r} \leq 1.
\eq
Here $a$ is a constant which may take all but finitely many complex values, depending on $(p,q,r)$.
For each $a$, view $T_{p,q,r}$ as the germ of a holomorphic function near the origin: it is an isolated hypersurface singularity. 
We assume without loss of generality that $p \geq q \geq r$. 
Let $\mathcal{T}_{p,q,r}$ denote the Milnor fibre of $T_{p,q,r}$. This is a Liouville domain, which, as shown in \cite[Section 2]{Keating14}, is independent of choices, including the choice of representative for a germ and the constant $a$. 

In the classification of isolated hypersurface singularities by Arnol'd and collaborators, these are all but finitely many of the \emph{modality one} singularities \cite[Section I.2.3 and II.2.5]{Arnold_VI}; missing are the fourteen so-called `exceptional' singularities (known as the object of strange duality).  
In particular, from the perspective of this classification, these singularities are the next most sophisticated after the simple singularities, which are the modality zero ones: $A_n$, $D_n$, $E_6$, $E_7$ and $E_8$. In contrast, homological mirror symmetry for these is comparatively well understood -- see for instance \cite{Chan-Ueda, Abouzaid-Auroux-Katzarkov, Etgu-Lekili}.

\subsection{Motivation: candidate mirror spaces following Gross--Hacking--Keel}

\subsubsection*{Cusp singularities: $1/p + 1/q + 1/r < 1$}  Let us first recall the set-up of Looijenga's conjecture. 
A \emph{cusp singularity} $(X,x)$ is the germ of an isolated normal surface singularity such that the exceptional divisor $D = \pi^{-1}(x)$ of the minimal resolution $\pi: \widetilde{X} \to X$ of the singularity is a cycle of smooth rational curves meeting transversally. Cusp singularities naturally arise in so-called `dual' pairs; for such a pair $(X, x)$, $(X', x')$, the associated exceptional divisors, say $D$ and $D'$, are called dual cycles. Given the sequence of self-intersections of the components of $D$, one can algorithmically obtain those of $D'$, and vice versa. 
Looijenga  \cite{Looijenga} showed that if a cusp with cycle $D'$ is smoothable, then there exists a pair $(Y, D)$ such that $Y$ is a smooth rational surface, and $D \in |-K_Y|$ is the dual cycle to $D'$; he conjectured that the converse also holds. 
Key to his work is the construction for each dual pair $(D, D')$  of a \emph{Hirzebruch--Inoue surface},
a smooth complex surface 
 whose only cycles are the components of $D$ and $D'$.  

For a triple $(p,q,r)$ with $1/p+ 1/q+1/r <1$,  and any $a \neq 0$, the singularity $T_{p,q,r}$ is a cusp singularity -- in fact, these are precisely those cusp singularities which are also hypersurface singularities.  In the work of Arnol'd and collaborators, they are called \emph{hyperbolic singularities}. 
  (For $a=0$, one gets a  different isolated singularity, often know as a Brieskorn--Pham one.)

The dual cusp singularity to $T_{p,q,r}$ is usually known as a \emph{triangle singularity}. It has a resolution with  exceptional divisor $D$, a cycle of three rational curves meeting transversally; the three components of $D$ have self intersections $1-p, 1-q$ and $1-r$ respectively. In this case, as $T_{p,q,r}$ is a hypersurface singularity, it is immediate to give a smoothing of it: its Milnor fibre $\mathcal{T}_{p,q,r}$. 
On the other hand, one can construct by hand pairs $(Y, D)$.
To do so, one possibility is as follows: pick collections of, respectively, $p, q$ and $r$ points on the interiors of each of the components  of the toric anti-canonical divisor $D_{\P^2}$ on $\P^2$ (possibly with repeats). Let $Y_{\bar{p}, \bar{q}, \bar{r}}$ be the smooth rational variety obtained by blowing up all $p+q+r$ points. 
If a point $x$ is repeated in the collection, say twice, our convention is to first blow up $x$, then  blow-up the intersection of the exceptional divisor $\pi^{-1}(x)$ with the strict transform of $D_{\P^2}$. 
Let $D \subset Y_{\bar{p}, \bar{q}, \bar{r}}$ be the strict transform of $D_{\P^2}$; by construction, the pair $(Y_{\bar{p}, \bar{q}, \bar{r}}, D)$ is as desired. 

Loosely, from \cite{GHK_logCY, CPS} we expect that under homological mirror symmetry, pairs $(Y_{\bar{p}, \bar{q}, \bar{r}}, D)$ will correspond to smoothings of $T_{p,q,r}$. 
As a special example, consider the case where we simply take $p$ copies of one point on the first component of $D_{\P^2}$, $q$ copies of a point on the second, and $r$ copies of a point on the third. Moreover, assume that these points are collinear -- for instance, pick $[1:-1:0]$, $[0:1:-1]$ and $[-1:0:1]$. 
Denote the resulting blow-up by $Y_{p,q,r}$. We expect algebraic invariants of $(Y_{p,q,r}, D)$ to correspond to symplectic invariants of $\mathcal{T}_{p,q,r}$. (Deforming the choices of $p$, $q$ and $r$ points should correspond to equipping $\mathcal{T}_{p,q,r}$ with non-exact symplectic forms; we will not consider this here.)

\subsubsection*{Simple elliptic singularities: $1/p + 1/q + 1/r = 1$}
There are precisely three triples of positive integers the sum of whose reciprocals is equal to one: $(3,3,3)$, $(4,4,2)$ and $(6,3,2)$. 
For the constant $a$ we exclude values such that  $x^p+y^q+z^r+ axyz$ has a non-isolated singularity at zero: $a^3 \neq -27$ for $(3,3,3)$, $a^2 \neq 9$ for $(4,4,2)$, and  $a^6 \neq 432$ for $(6,3,2)$. 
In works of Arnol'd and collaborators, these are known as \emph{parabolic singularities}; to use different terminology, they are also precisely those \emph{simple elliptic singularities} which are hypersurface singularities \cite{Saito}.

The Milnor fibres $\mathcal{T}_{3,3,3}$, $\mathcal{T}_{4,4,2}$ and $\mathcal{T}_{6,3,2}$ are given by deleting a smooth anticanonical divisor (an elliptic curve) from a del Pezzo surface of degree 3, 2 and 1, respectively \cite[Proposition 5.19]{Keating14}. 
As such, these also fit into the framework of Gross, Hacking and Keel \cite[p.~6]{GHKv1}: they conjectured that if one applies their constructions to a pair $(Y, D)$, where $Y$ is a rational elliptic surface and $D$ is an $I_d$ fibre of the surface (i.e.~a cycle of $d$ rational curves meeting transversally), then the mirror family that one obtains contains a del Pezzo surface of degree $d$ with a smooth anti-canonical divisor deleted. 
As motivation for this conjecture, note that there are analogues of Hirzebruch--Inoue surfaces in this setting, called \emph{parabolic Inoue surfaces}. These are smooth compact complex surfaces whose only curves are an elliptic curve of self-intersection $-n$, and the components of a cycle of $n$ rational curves of self-intersection $-2$.

Now notice that $Y_{3,3,3}$, by construction, is a rational elliptic surface, with $D$ an $I_3$ fibre. In the other two cases, by blowing down either one rational $-1$ curve in $D$ (for $Y_{4,4,2}$), or, sequentially, two rational $-1$ curves in $D$ (for $Y_{6,3,2}$), one gets rational elliptic surfaces with, respectively, an $I_2$ fibre and an $I_1$ fibre. In particular, Gross--Hacking--Keel's conjecture gives candidate mirror spaces to these. 

%%%%%%%%%%%%%%%%%%%%

\subsection{Statement of results}
There is a distinguished Lefschetz fibration $\Xi: \mathcal{T}_{p,q,r} \to \C$, with smooth fibre $M$, such that we have the following collection of equivalences of categories:
\begin{eqnarray}
D^b \Fuk^{\to} (\Xi) & \cong & D^b \Coh(Y_{p,q,r}) \label{eq:iso1} \\
D^{\pi} \Fuk(M) & \cong & \Perf(D) \label{eq:iso2} \\
D^b \mathcal{W}(\mathcal{T}_{p,q,r}) & \cong & D^b \Coh(Y_{p,q,r} \backslash D)  \label{eq:iso3}
\end{eqnarray}
where
\begin{itemize}
\item $\Fuk^{\to} (\Xi)$ denotes the directed Fukaya category of $\Xi$, sometimes also known as its Fukaya--Seidel category. This is associated to the fibration $\Xi$ together with the choice of a distinguished collection of vanishing cycles for it. (`Distinguished' means that the corresponding vanishing cycles only intersect at one end point.) We use the set-up of \cite[Section 18]{Seidel_book}.
\item $\Fuk(M)$ denotes the Fukaya category of the fibre $M$, a three-punctured elliptic curve, as set up in \cite[Section 12]{Seidel_book}. Its objects are complexes built from compact Lagrangians in $M$.
$D^{\pi} \Fuk(M)$ denotes its derived split-closure. 
 $\Perf(D)$ is the category of perfect complex of algebraic vector bundles on the singular variety $D$.
\item $\mathcal{W}(\T_{p,q,r})$ denotes the wrapped Fukaya category of $\mathcal{T}_{p,q,r}$, as set-up in \cite{Abouzaid-Seidel_wrapped}. Its objects are complexes built from Lagrangians which either are compact, or whose shape near the boundary of $\mathcal{T}_{p,q,r}$ is prescribed. 
\end{itemize}

\begin{remark}
One could also  consider other Lefschetz fibrations on $\mathcal{T}_{p,q,r}$. However, one can't in general expect the corresponding directed Fukaya categories to be mirror to coherent sheaves on any nice compactifications of $Y_{p,q,r} \backslash D$. As such, from the perspective of mirror symmetry, $\Xi$ is a `preferred' Lefschetz fibration on $\mathcal{T}_{p,q,r}$. 
\end{remark}

\subsubsection*{Extension for $1/p+1/q+1/r > 1$} While the singularities $T_{p,q,r}$ exist for $1/p+1/q+1/r \leq 1$,  the space $Y_{p,q,r}$ makes sense for any triple of positive integers $(p,q,r)$. Similarly, one can define a Liouville domain $\mathcal{T}_{p,q,r}$, together with a Lefschetz fibration $\Xi: \mathcal{T}_{p,q,r} \to \C$:
simply take the description of Proposition \ref{thm:LF} as a definition. (It is however no longer the Milnor fibre of a hypersurface singularity.) Thus interpreted,  the equivalences \eqref{eq:iso1}, \eqref{eq:iso2} and \eqref{eq:iso3} above also hold for these $(p,q,r)$. 

The possible triples of integers here are $(n,2,2)$ for any $n \geq 2$, and $(k,3,2)$, $k=3,4,5$. The manifolds $\mathcal{T}_{p,q,r}$ share many properties of Milnor fibres: they are homotopy equivalent to a wedge of spheres, and a basis for $H_2$ is given by a collection of Lagrangian spheres. The associated intersection form is given by the Dynkin diagram of Figure \ref{fig:Dynkin}. These are `augmented' versions of the intersection forms of $D_{n+2}$ and $E_{k+3}$, given by taking their direct sum with a one-dimensional vector space equipped with the trivial intersection form.
\begin{figure}[htb]
\begin{center}
\includegraphics[scale=0.38]{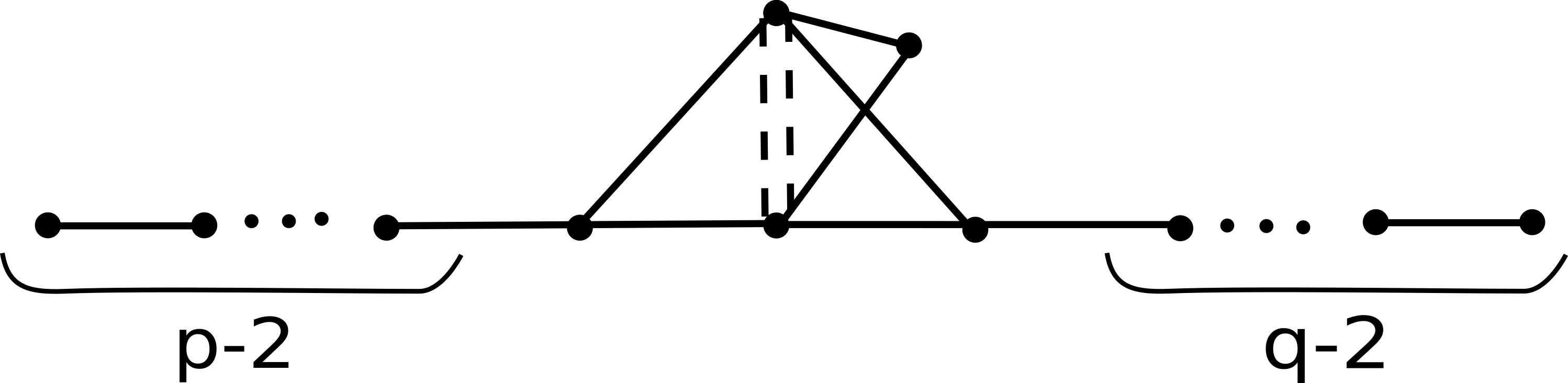}
% or [scale=0.85]
\caption{Intersection form for $\mathcal{T}_{p,q,2}$: Dynkin diagram.}
\label{fig:Dynkin}
\end{center}
\end{figure}

On the other hand, the spaces $Y_{n,2,2}$ and $Y_{k,3,2}$ contain triangular configurations of $-2$ curves.  These can be blown down to obtain singular varieties $\bar{Y}_{n,2,2}$, and $\bar{Y}_{k,3,2}$, with, respectively, singularities of types $D_{n+2}$ and $E_{k+3}$. 
We note that the spaces $\bar{Y}_{n,2,2} \backslash D$ and $\bar{Y}_{k,3,2} \backslash D$ are affine. They have Euler number two\footnote{The author thanks Daniel Litt for this observation.}; in contrast,  the affine cones on the $D_{n+2}$ and $E_{k+3}$ singularities (e.g.~$x^3+y^3=0$ for $D_4$) have Euler number one.

\subsubsection*{Relation to the directed Fukaya category of $T_{p,q,r}$} It was already known \cite[Theorem 7.1]{Keating14} that
\bq
D^b \Fuk^{\to}(T_{p,q,r}) \cong D^b \Coh(\P^1_{p,q,r})
\eq
where $\Fuk^{\to}(T_{p,q,r})$ is the directed Fukaya category of the singularity $T_{p,q,r}$, i.e.~the directed Fukaya category associated to the Lefschetz fibration on $\C^3$ given by any Morsification of $T_{p,q,r}$, and $\P^1_{p,q,r}$ denotes an orbifold $\P^1$, with orbifold points of isotropies $p,q$ and $r$. 
This can be tied back into the picture presented here as follows. Consider the cotangent bundle of $\P^1_{p,q,r}$. Under a resolution of the three orbifold points, the zero-section $\P^1_{p,q,r}$ pulls back to a triangular chain of $-2$ curves, with one central $-2$ curve, and chains of lengths $p-1$, $q-1$ and $r-1$ attached to it. On the other hand, one sees precisely such a triangular configuration in the interior of $Y_{p,q,r} \backslash D$.

\subsection{Outline of proof}

The Lefschetz fibration $\Xi$ is described in Section \ref{sec:fibration}.
It has $p+q+r+3$ critical points, and $\Fuk^{\to}(\Xi)$ has the same number of objects, one for each of the vanishing cycles. We match up the distinguished collection of Lagrangian branes associated with the vanishing cycles with a full exceptional sequence of objects for $D^b \Coh(Y_{p,q,r})$, and show that the cohomology level products on the two agree (Proposition \ref{prop:isomorphic_sequences}). The latter exceptional collection is obtained using Bondal--Orlov's theorem \cite{Bondal-Orlov}. See Section \ref{sec:iso1}. 

	To prove equivalence \eqref{eq:iso2} (Proposition \ref{pro:iso2}), we use work of Lekili and Perutz \cite{Lekili-Perutz}, which relates the Fukaya category of a once-punctured torus with the category of perfect complexes on a nodal elliptic curve, together with $3:1$ covering arguments. See Section \ref{sec:iso2}.

Our principal aim is to prove equivalences \eqref{eq:iso1} and \eqref{eq:iso3} (Corollary \ref{thm:iso1} and Theorem \ref{thm:iso3}). These follow from a sufficiently fine understanding of Propositions  \ref{prop:isomorphic_sequences}  and Proposition \ref{pro:iso2}, along with relations between them. We consider two short exact sequences of  the form
\bq
0 \to \A \to \B \to \A^\vee[-1] \to 0
\eq
where, for the $\Aoo$--category $\A$, $\B$ is an $\Aoo$--bimodule over $\A$, and $\A$ and  $\A^\vee$ are viewed as the diagonal, resp.~dual diagonal, bimodules. 
On the symplectic side, such a sequence arises in Seidel's program on Fukaya categories associated to Lefschetz fibrations \cite{Seidel_FLI, Seidel_FL2}.  
Following \cite{Seidel_FLI}, one takes $\A_\F = \Fuk^{\to} (\Xi)$ for our choice of distinguished collection of vanishing cycles. $\B_\F \subset \Fuk(M)$ is the full subcategory with objects the same Lagrangians as in $\A_\F$. For suitable models there is a natural inclusion of categories $\A_\F \subset \B_\F$, which equips $\B_\F$ with the structure of a bimodule over $\A_\F$.

On the algebraic side, there exists such a sequence with categories $\A_\tC$ and $\B_\tC$ as follows. We fix some dg enhancement $\twvect(Y_{p,q,r})$ (resp.~$\twvect(D)$) of $D^b \Coh(Y_{p,q,r})$ (resp.~$\Perf(D)$). $\A_\tC \subset \twvect(Y_{p,q,r})$ is the full subcategory on the objects of the exceptional sequence,  and $\B_\tC \subset \twvect(D)$ a full subcategory on a collection of split-generators.
For suitable models there is an inclusion of categories $\A_\tC \subset \B_\tC$, which one can think of as a refinement of the pull-back $i^\ast: D^b \Coh(Y_{p,q,r}) \to \Perf(D)$. 
 The equivalence \eqref{eq:iso1} can be strengthened  to a quasi-isomorphism $\B_\tC \cong \B_\F$, which, together with Proposition \ref{prop:isomorphic_sequences}, is sufficient to prove Corollary \ref{thm:iso1}. Along the way, we prove formality of $\A_\F$ and $\A_C$. See Section \ref{sec:restrictions}.

To prove equivalence \eqref{eq:iso3}, we apply the the localization procedure described in \cite{Seidel_subalgebras} to the two pairs $(\A_\F, \B_\F)$ and $(\A_\tC, \B_\tC)$. On the algebraic side,  localization  yields $D^b \Coh(Y_{p,q,r} \backslash D)$; on the Fukaya side, by work-in-progress of Abouzaid and Seidel \cite{Abouzaid-Seidel}, it gives 
 $D^b \mathcal{W}(\T_{p,q,r})$. 
Here is some geometric intuition: on the symplectic side, we obtain the localization of $D^b \Fuk^{\to} (\Xi)$ with respect to a natural transformation $T: \mu \to \text{Id}$, where $\mu$  is the `monodromy at infinity' acting on the Lefschetz fibration, as in \cite{Seidel_FL2}. On the algebraic side, we get the localization of $D^b \Coh(Y_{p,q,r})$ with respect to a natural transformation $T:- \otimes \O_{Y_{p,q,r}}(-D) \to \text{Id}$ given by multiplying by a section of $\O_{Y_{p,q,r}}(D)$. See Section \ref{sec:localization}.

\subsection*{Acknowledgements} This project began towards the end of my time as a graduate student at MIT; it is a pleasure to thank my advisor, Paul Seidel, for several useful conversations. I would also like to thank Mark Gross and Paul Hacking for their patient explanations, and Mohammed Abouzaid, Denis Auroux, Robert Friedman, John Lesieutre, Timothy Perutz, Nicholas Sheridan and Ivan Smith for helpful conversations and correspondence.

%%%%%%%%%%%%%%%%%%%%%%%%%%%%%%%%%%%%%%%%%%
%%%%%%%%%%%%%%%%%%%%%%%%%%%%%%%%%%%%%%%%%%
%%%%%%%%%%%%%%%%%%%%%%%%%%%%%%%%%%%%%%%%%%

\section{A distinguished Lefschetz fibration on the Milnor fibre $\T_{p,q,r}$}\label{sec:fibration}

\subsection{Known Lefschetz fibration on $\T_{p,q,r}$}

Let $\mathcal{T}_{p,q,r}$ denote  the Milnor fibre of $T_{p,q,r}$; pick any ordering of the subscripts so that $p \geq 3$, $q \geq 3$ and $r \geq 2$.

\begin{proposition}\cite[Proposition 4.3]{Keating14}\label{thm:oldLF}
 The space $\mathcal{T}_{p,q,r}$ can be described as a smoothing of the corners of the total space of a Lefschetz fibration $\pi$ with $2p+2q+r$ critical points, and smooth fibre the Milnor fibre of the two-variable singularity $x^p+y^q+x^2y^2$ (after adding $z^2$, this is the singularity $T_{p,q,2}$). This has the following topology
\begin{itemize}
\item if  $p$ and $q$ are odd, two punctures and genus $(p+q)/2$;
\item if $p$ is odd and $q$ is even, or vice versa, three punctures and genus $(p+q-1)/2$;
\item if $p$ and $q$ are even, four punctures and genus $(p+q-2)/2$. 
\end{itemize}
%(Recall that its stabilization, $x^p+y^q+x^2y^2+z^2$ is a representative germ of $T_{p,q,2}$.) 
Of these critical points, $2(p+q+1)$ of them pair off to give $p+q+1$ matching paths (and vanishing cycles in $\T_{p,q,r}$), which all intersect in a single smooth point. In the fibre above this point, the matching cycles restrict to a configuration of vanishing cycles for the two-variable singularity $x^p+y^q+x^2y^2$ . This consists of a `core' of five cycles, labelled $A$, $B$, $P_1$, $Q_1$ and $R_1$ in Figure \ref{fig:T345fibreandbaseOLD}, together with two $A_n$--type chains (on the Riemann surface): one, of total length $p-1$, starting at $P_1$, and another, of total length $q-1$, starting at $Q_1$. The case of $p= 3$ and $q= 4$ is given in Figure \ref{fig:T345fibreandbaseOLD}.  

   There is an $A_{r-2}$--type chain of matching paths between  the remaining $r-2$ critical points; these give the remaining vanishing cycles for $\mathcal{T}_{p,q,r}$. They are always attached following the configuration of Figure \ref{fig:T345fibreandbaseOLD} (case $r=5$). 

Moreover, in all cases, there is a Hamiltonian isotopy of the total space such that the image of $R_2$ does not intersect $A$, and its intersections with other cycles are unchanged.

\begin{figure}[htb]
\begin{center}
\includegraphics[scale=0.85]{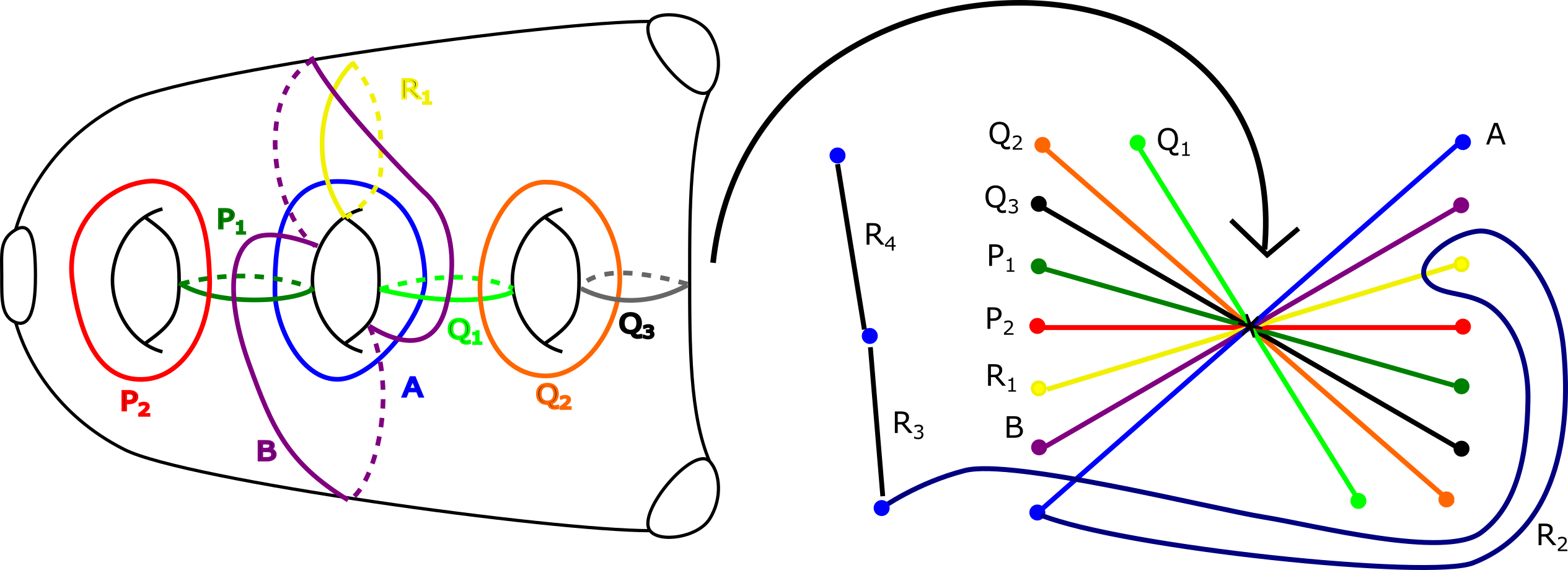}
% or [scale=0.85]
\caption{ Matching paths, and cycles in a fibre, giving a distinguished configuration of vanishing cycles for $T_{3,4,5}$
}
\label{fig:T345fibreandbaseOLD}
\end{center}
\end{figure}
 
\end{proposition}

The intersections between vanishing cycles are encoded in the Dynkin diagram of Figure \ref{fig:Dynkin_cycles}.
\begin{figure}[htb]
\begin{center}
\includegraphics[scale=0.45]{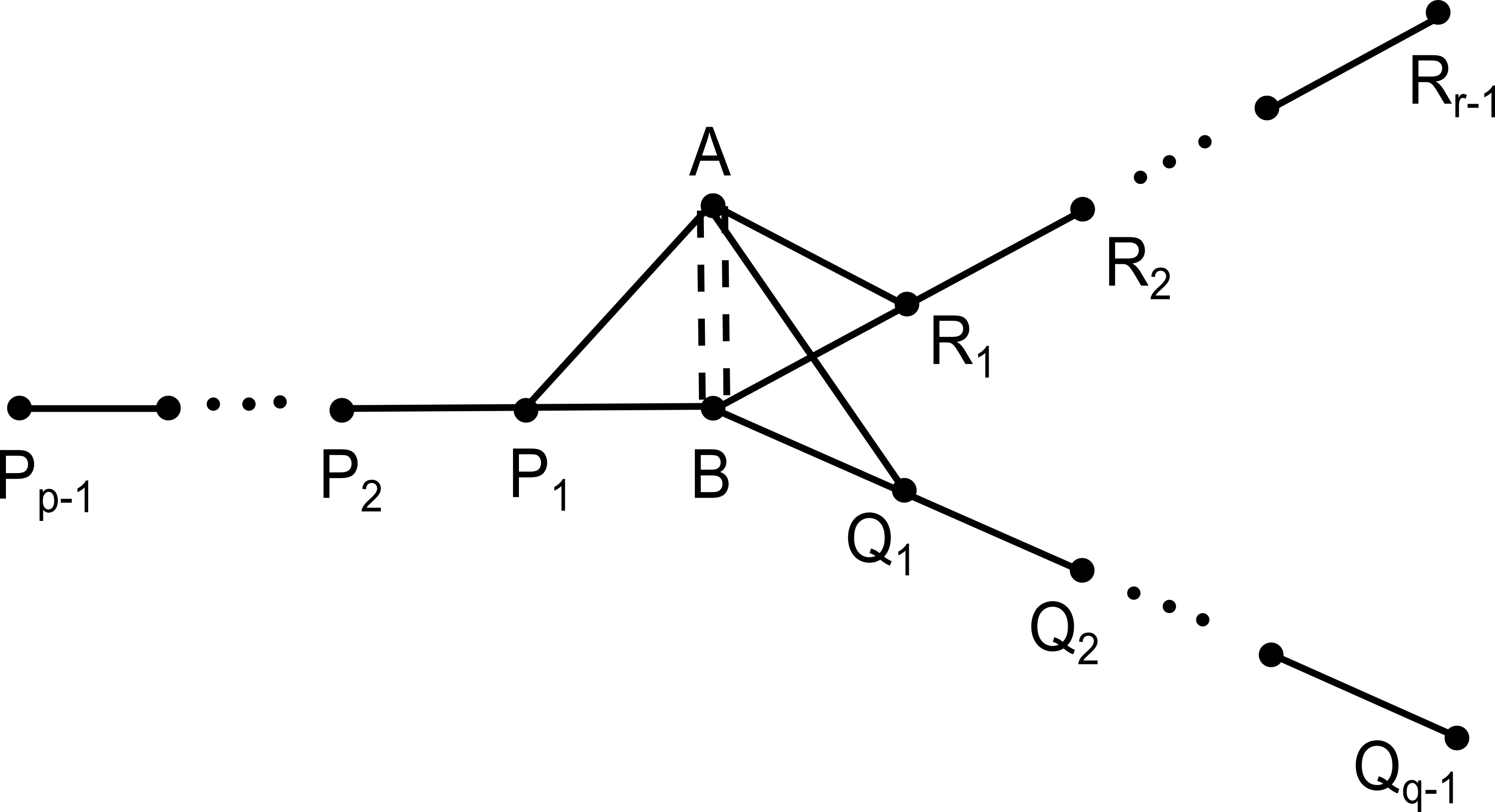}
% or [scale=0.85]
\caption{Intersections for the vanishing cycles of $\mathcal{T}_{p,q,r}$.
}
\label{fig:Dynkin_cycles}
\end{center}
\end{figure}

\subsection{Some operations on Lefschetz fibrations}

We will want to use a different Lefschetz fibration on $\T_{p,q,r}$, given below in Proposition \ref{thm:LF}. It will be related to the  description of Proposition  \ref{thm:oldLF} by a sequence of moves called stabilizations and mutations, which we briefly recall.

\subsubsection{Stabilizations}

A stabilization is the following operation. Start with a Lefschetz fibration on a four-dimensional Liouville domain  $(M, \omega = d \theta)$, with smooth fibre $\Sigma$ and $n$ critical points. Pick a  distinguished collection of vanishing paths; call the (cyclically ordered) associated  vanishing cycles $L_1, L_2, \ldots, L_n \subset \Sigma$. 
Given an embedded interval $\gamma \subset \Sigma$, with $\partial \gamma \subset \partial \Sigma$ and $0 = [\theta] \in H^1 (\gamma, \partial \gamma)$, one can construct a new Lefschetz fibration as follows:
\begin{itemize}
\item Replace $\Sigma$ with the surface $\Sigma'$ given by attaching a Weinstein handle to $\Sigma$ along $\partial \gamma$.
\item Let $L' \subset \Sigma'$ be the Lagrangian $S^1$ given by gluing $\gamma$ with the core of the handle. Add a critical point to the base of the Lefschetz fibration with corresponding vanishing cycle $L'$. More precisely, with respect to some distinguished collection of vanishing paths extending our previous choices, we get the cyclically ordered collection of vanishing cycles $L', L_1, L_2, \ldots, L_n$. 
\end{itemize}
For further details and generalizations, see e.g.~in \cite[Section 1.2]{Giroux-Pardon} and \cite[Section 3]{Akbulut-Arikan}. 
The total space of the new  Lefschetz fibration is deformation--equivalent to the total space of the original one (they can be connected by a one-parameter family of Liouville domains). 

\subsubsection{Mutations}

A mutation is a modification of the data given to describe a fixed Lefschetz fibration.
Given a distinguished collection of vanishing paths $\gamma_1, \ldots, \gamma_n$, with associated vanishing cycles $L_1, \ldots, L_n$, it is a change of the following form (or its inverse):
\begin{itemize}
\item Take the collection of vanishing paths $\gamma_1, \ldots, \gamma_{i-1}, \gamma'_{i+1}, \gamma_i, \gamma_{i+2}, \ldots, \gamma_n$, where $\gamma'_{i+1}$ is given by post-composing $\gamma_{i+1}$ with a positive loop about $\gamma_i$. See Figure \ref{fig:mutation}.
\item The associated collection of vanishing cycles is $L_1, \ldots, L_{i-1}, \tau_{L_i}L_{i+1}, L_i, \ldots, L_n$.
\end{itemize}
For further details, see e.g.~the exposition in \cite{Seidel_vcm}. 

\begin{figure}[htb]
\begin{center}
\includegraphics[scale=0.80]{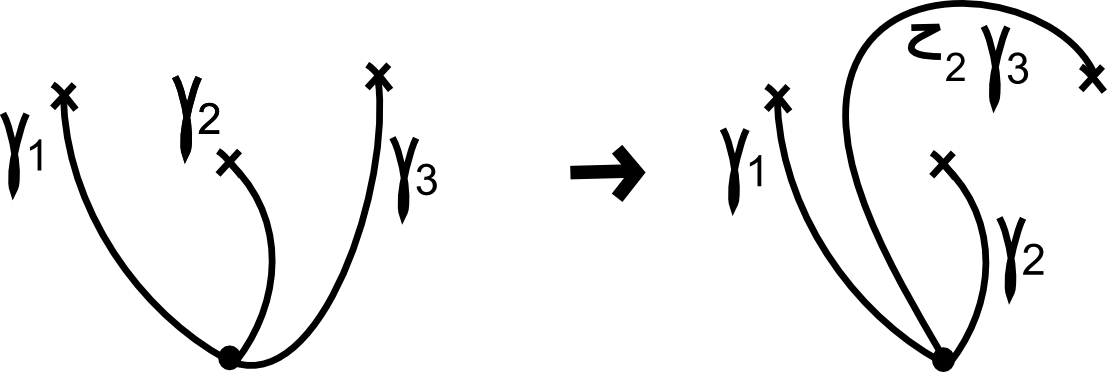}
% or [scale=0.85]
\caption{Mutation: change to the vanishing paths.}
\label{fig:mutation}
\end{center}
\end{figure}

\subsection{A symmetric Lefschetz fibration}

\begin{proposition}\label{thm:LF}
There is a Lefschetz fibration $\Xi: \T_{p,q,r} \to \C$ with smooth fibre a three-punctured elliptic curve $M$  and $p+q+r+3$ critical values. A description is given in Figure \ref{fig:Tpqrfibreandbase}. This shows critical values joined by matching paths, and the corresponding cycles in the fibre above the distinguished point $\star$. These matching cycles are, moreover, vanishing cycles for the original $T_{p,q,r}$ singularity. (Suitably ordered, they form a distinguished collection of vanishing cycles.)

\begin{figure}[htb]
\begin{center}
\includegraphics[scale=0.45]{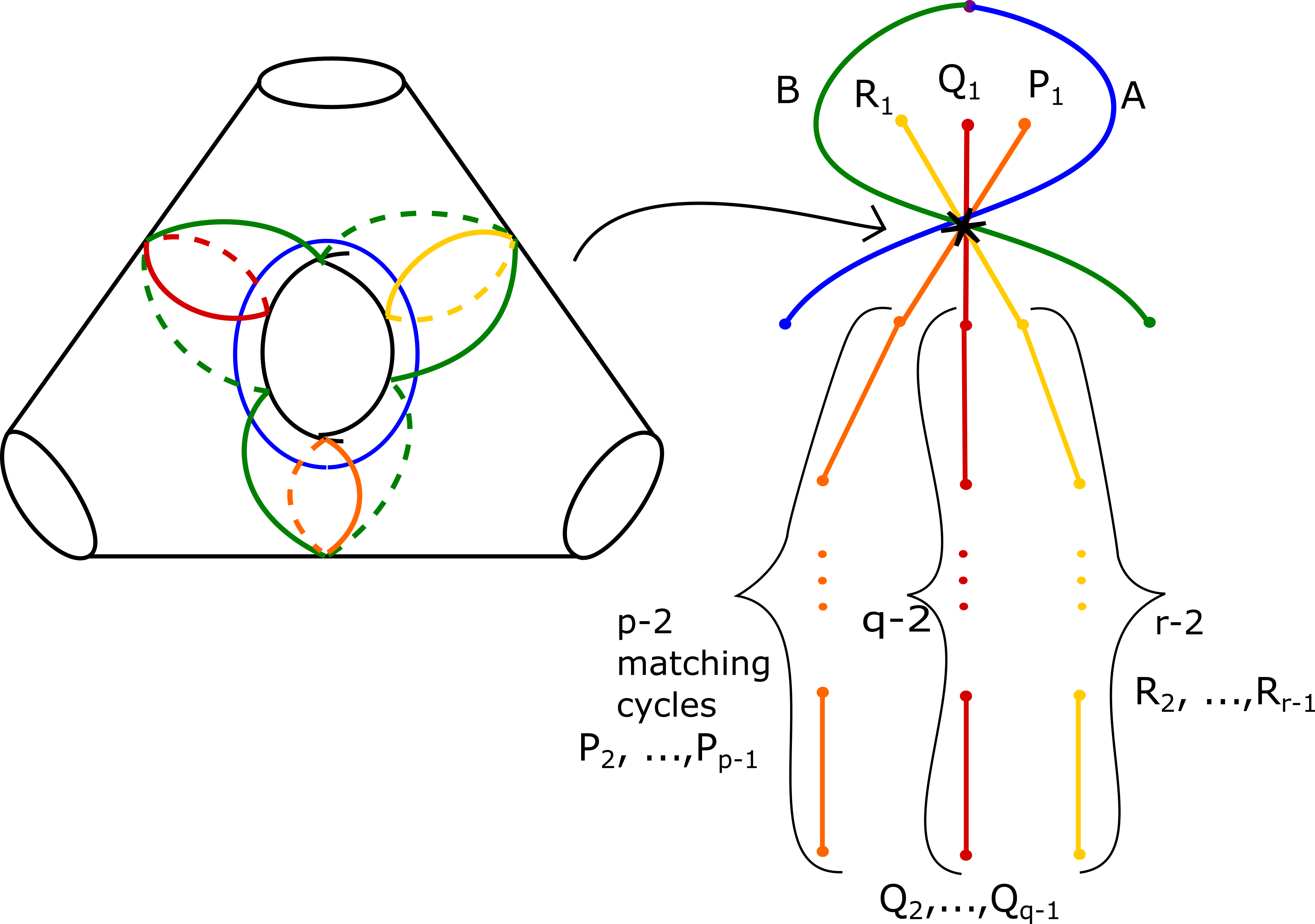}
% or [scale=0.85]
\caption{ Matching paths, and cycles in a fibre, giving a distinguished configuration of vanishing cycles for $T_{p,q,r}$
}
\label{fig:Tpqrfibreandbase}
\end{center}
\end{figure}

\end{proposition}

To prove Proposition \ref{thm:LF}, we will start with the Lefschetz fibration described therein, and perform moves to recover the one in \ref{thm:oldLF}. First, start with Figure \ref{fig:Tpqrfibreandbase}, and perform the stabilization shown in Figure \ref{fig:stabilizationcore}. After the mutations
\begin{eqnarray}
b & \to & \tau^{-1}_a b \\
a & \to & \tau_f \tau_e \tau_d \tau_c a
\end{eqnarray}
one precisely recovers the `core' configuration of Proposition  \ref{thm:oldLF}, given in Figure  \ref{fig:stabilizedcoremutated}.
\begin{figure}[htb]
\begin{center}
\includegraphics[scale=0.40]{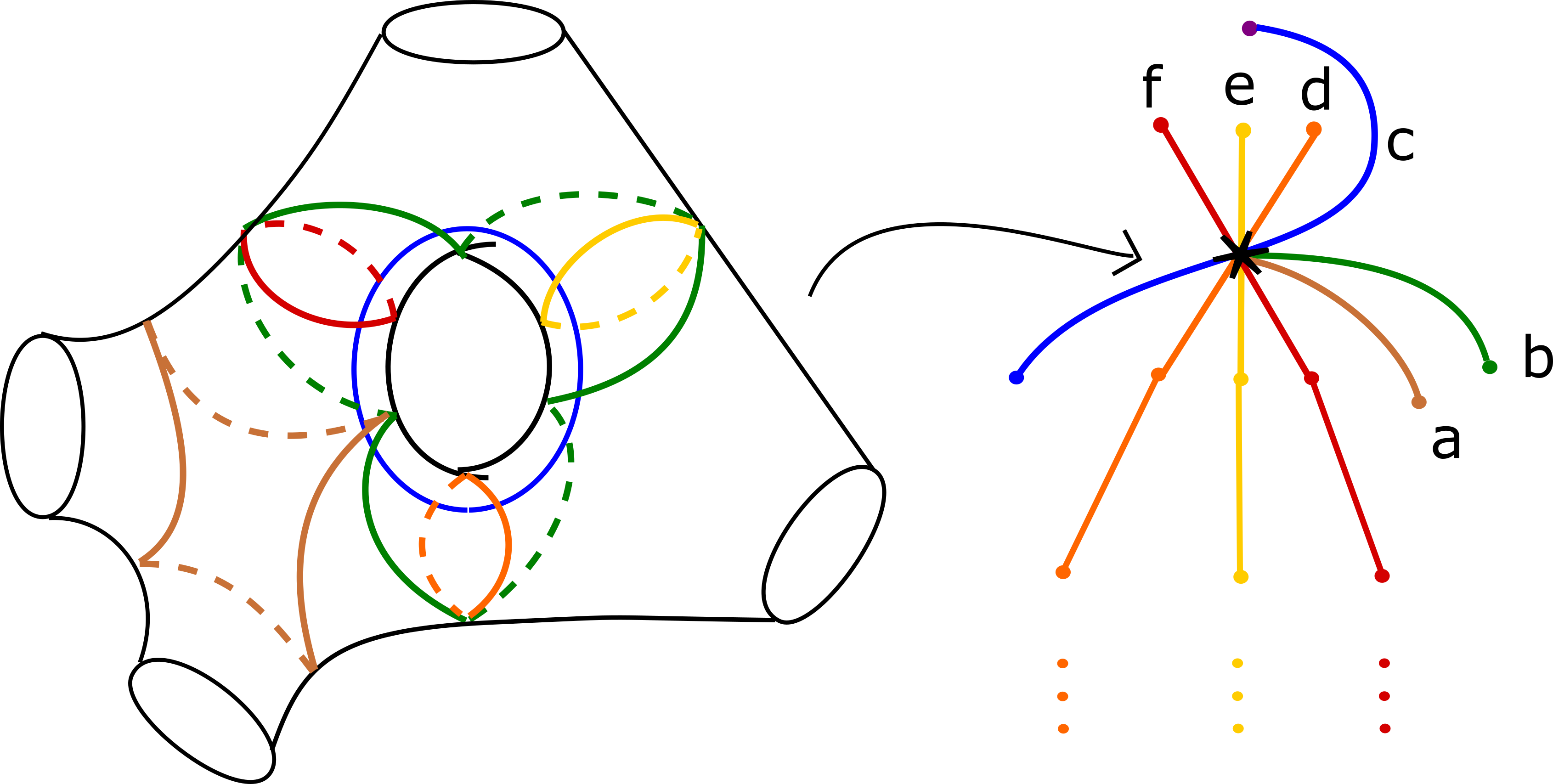}
% or [scale=0.85]
\caption{Core of Figure \ref{fig:Tpqrfibreandbase} after stabilizating along the brown curve $a$.}
\label{fig:stabilizationcore}
\end{center}
\end{figure}
\begin{figure}[htb]
\begin{center}
\includegraphics[scale=0.40]{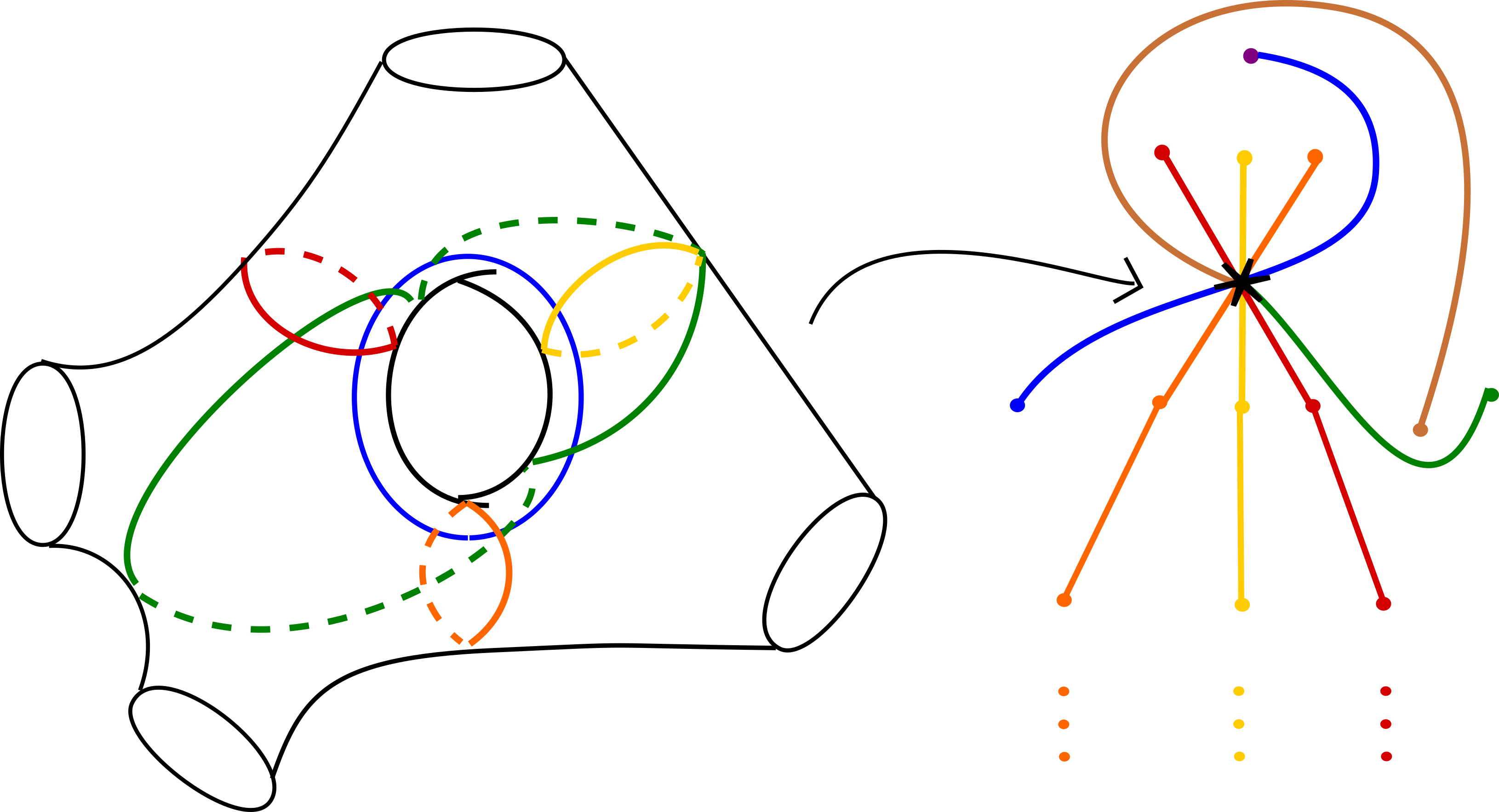}
% or [scale=0.85]
\caption{Stabilized core of Figure \ref{fig:Tpqrfibreandbase} after mutations. The brown and green vanishing cycles now agree.}
\label{fig:stabilizedcoremutated}
\end{center}
\end{figure}

Now notice that $P_1, \ldots, P_{p-1}$ and $Q_1, \ldots, Q_{q-1}$ are simply $A_n$ chains, for $n=p-1, q-1$. Locally, the Lefschetz fibrations correspond to the two `standard' Lefschetz fibrations on the $A_n$ Milnor fibre:  $$x: \{ x^2 + y^2 + p(z) =1 \} \to \C,$$ and $$z: \{ x^2 + y^2 + p(z) =1 \} \to \C$$
where $p$ is a Morse polynomial of degree $n+1$.
One can pass from the second one to the first one by a sequence of stabilizations, iterating the $A_2$ model, which is  given in Figure \ref{fig:A2stabilization}.

\begin{figure}[htb]
\begin{center}
\includegraphics[scale=0.25]{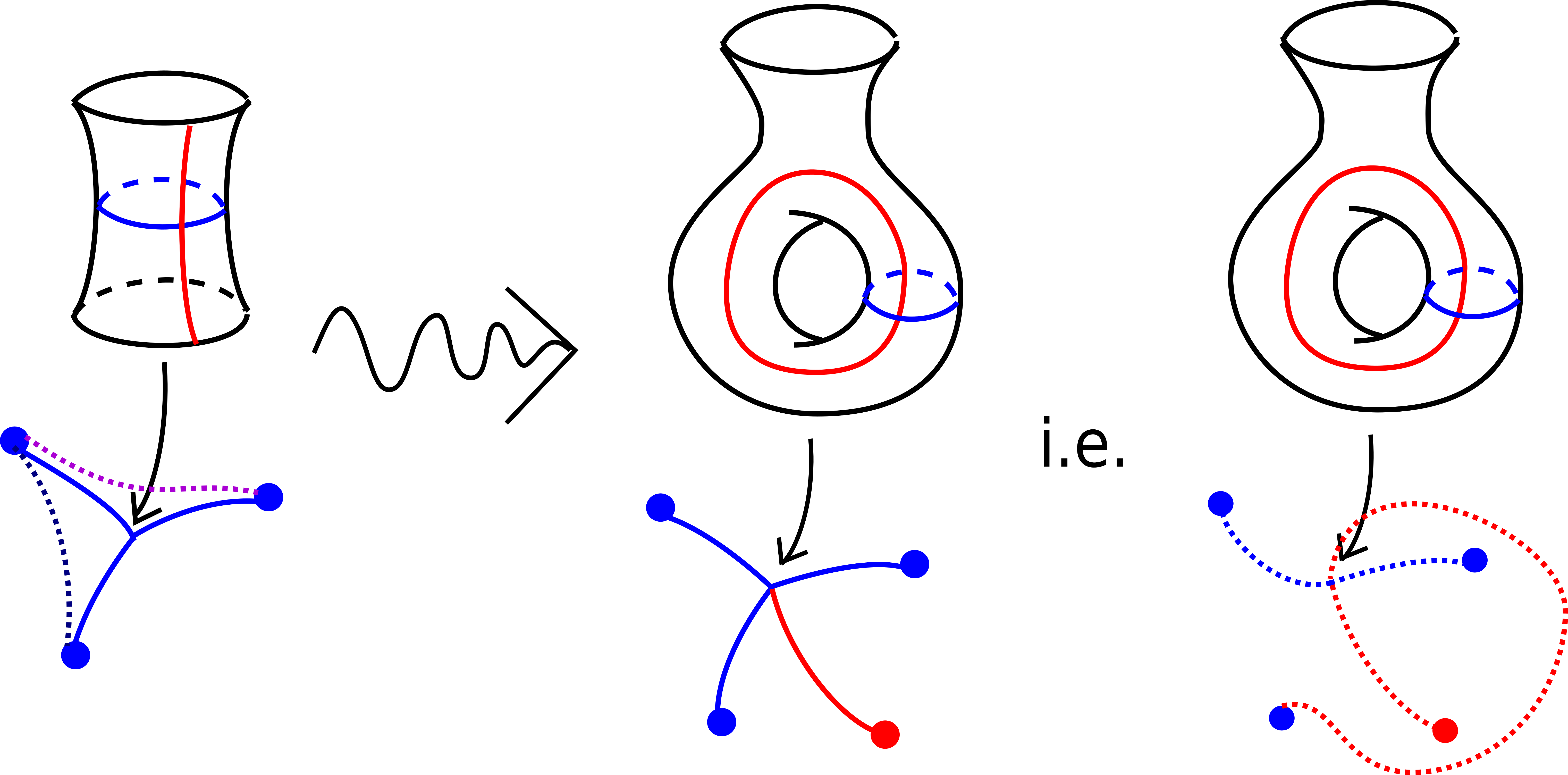}
% or [scale=0.85]
\caption{Stabilization of a Lefschetz fibration on the $A_2$ Milnor fibre, along the red curve. The dotted lines give matching paths for vanishing cycles in the total spaces.}
\label{fig:A2stabilization}
\end{center}
\end{figure}

Finally, consider the local model for a neighbourhood of $A$, $R_1$ and $R_2$ in Figure \ref{fig:T345fibreandbaseOLD} (``old'' Lefschetz fibration). This is the left-hand side of Figure \ref{fig:R2localmodel}. They form an $A_3$--type chain, and one can rearrange to get the matching paths and cycles on the right-hand side of Figure \ref{fig:R2localmodel}. 
This is precisely the configuration one sees as the local model for a neighbourhood of $A$, $R_1$, and $R_2$ in Figure \ref{fig:Tpqrfibreandbase} (``new'' Lefschetz fibration).  
This completes the proof of Proposition \ref{thm:LF}.

\begin{figure}[htb]
\begin{center}
\includegraphics[scale=0.25]{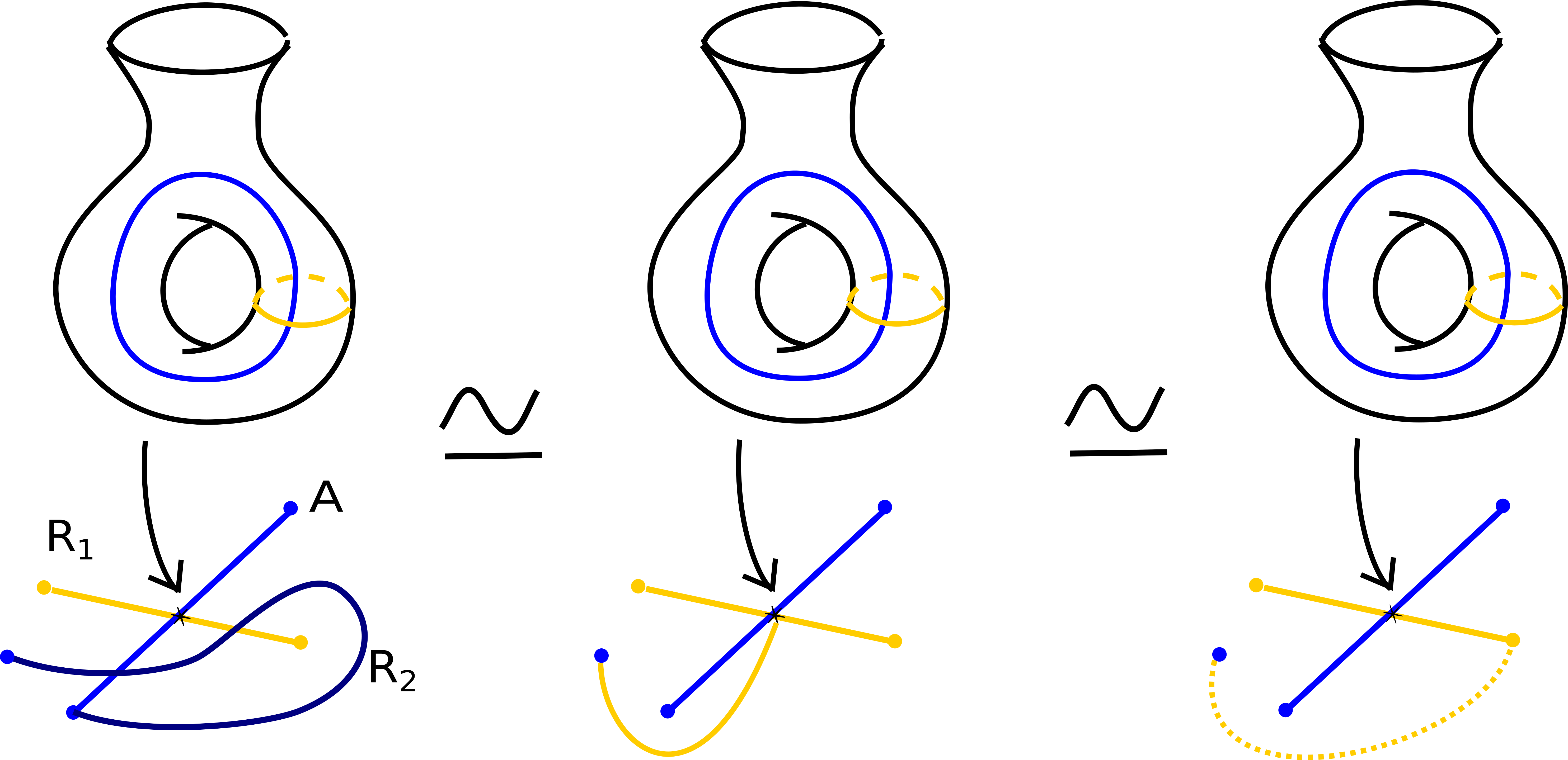}
% or [scale=0.85]
\caption{Local model for a neighbourhood of $A$, $R_1$ and $R_2$: three instance of exactly the same Lefschetz fibration, described using different matching and vanishing path data.}
\label{fig:R2localmodel}
\end{center}
\end{figure}

After deformation, we can arrange for the total space of the second Lefschetz fibration (with corners smoothed) to be a Liouville subdomain of the first one, given by taking convex (after deformation) open subsets of the fibre and of the base. 
(The cobordism given by the complement of the two is trivial.)
In particular, one can equip the new Lefschetz fibration $\Xi$ with a boundary-convex almost-complex structure $J$ such that $\Xi$ is $(J,i)$--holomorphic.

%%%%%%%%%%%%%%%%%%%%%%%%%%%%%%%%%%%%%%%
%%%%%%%%%%%%%%%%%%%%%%%%%%%%%%%%%%%%%%%
%%%%%%%%%%%%%%%%%%%%%%%%%%%%%%%%%%%%%%%

\section{$D^b \Fuk(\Xi)$ and $D^b Coh(Y_{p,q,r})$: comparison of semi-orthogonal decompositions}\label{sec:iso1}

We use Seidel's conventions (e.g.~\cite{Seidel_book}) for the ordering of indices  in $A_\infty$--operations throughout this document.

\subsection{$\A$: directed Fukaya category of $\Xi$}\label{sec:directed_Fukaya}

\subsubsection{Background: set-up and conventions}\label{sec:directed_Fukaya_background}
Given the data of a Lefschetz fibration $\Pi$, together with some auxiliary choices, one can build the directed Fukaya category of $\Pi$, $\Fuk^{\to} (\Pi)$. The principal auxiliary data is the choice of an ordered distinguished collection of vanishing paths for $\Pi$. The associated category $D^b \Fuk^{\to} (\Pi)$ is an invariant of the Lefschetz fibration. There are two different, equivalent ways of constructing these categories: one using the collection of Lagrangian thimbles associated to the vanishing paths (in which case directedness arises naturally), and one using the Lagrangian vanishing cycles in the smooth fibre of $\Pi$ (in which case, superficially, directedness presents itself as an algebraic imposition). 
We shall use the second approach, which better suits our later purposes. 
We give details in the case where the fibre of $\Pi$, say $(M, \omega = d \theta)$, has real dimension two, assuming moreover that $c_1(M) = 0$. For further background, see \cite[Section 18]{Seidel_book}. 

Fix a trivialization $TM \cong M \times \C$, and let $\alpha: \textrm{Gr}(TM) \to S^1$ be the associated \emph{squared} phase map. (See \cite[Section 11j]{Seidel_book}.) Suppose we are given an ordered, distinguished collection of vanishing paths for $\Pi$. Call the associated vanishing cycles $V_1$, \ldots, $V_m$. The objects of $\Fuk^{\to} (\Pi)$ are the Lagrangian branes
\bq
V_i^\# = (V_i, \alpha_i^\#, \mathfrak{s}_i ) 
\eq
where
\begin{enumerate}
\item $\alpha_i^\#$ is a choice of grading for $V_i$, i.e.~a smooth function $V_i \to \R$ such that $\mathrm{exp}(2\pi i \alpha^\#_i (x) ) = \alpha (T{V_i}_x)$. For a given $V_i$, these form a $\Z$-torsor. (Notice that $\alpha$ and $\alpha^\#_i$ determine an orientation of $V_i$ -- see \cite[Remark 11.18 and Section 12a]{Seidel_book}.)
\item $\mathfrak{s}_i$ is a spin structure on $V_i$. In the case at hand, $V_i$ is a circle, which admits two spin structures; we choose the \emph{non-trivial} one, i.e.~the one which corresponds to the connected double-cover of $S^1$.
\end{enumerate}

After generic Hamiltonian perturbations, we may assume that any pair $\{ V_i , V_j \}$ meet transversally, and that there are no triple intersection points. Fix a boundary-adapted complex structure $J$. 
The morphism groups are given by:
\bq
hom^\ast (V_i, V_j) = 
\begin{cases} CF^\ast (V_i, V_j ) & \textrm{if \,} i<j \\
\C\langle e_i \rangle & \textrm{if \,} i=j \\
0 & \textrm{otherwise}
\end{cases}
\eq
Here $e_i$ is an element of degree zero, and $CF^\ast (V_i, V_j)$ is the Floer chain group associated to $V_i$ and $V_j$, i.e.~the graded complex vector space generated by $V_i \pitchfork V_j$, where $x \in V_i \pitchfork V_j$ has grading its Maslov index $i(x)$, see \cite[Section 13c]{Seidel_book}. 

In order to keep track of signs in the $\Aoo$--morphisms $\mu^i$, we use the same framework as \cite{Lekili-Perutz-PNAS}. Equip each vanishing cycle $V_i$ with a marked point $\star$ which does not coincide with any of the intersection points of $V_i$ with the other $V_j$. This encodes the non-trivial spin structures: take the double-cover of $V_i$ that is trivial over $V_i \backslash \{ \star \}$, and swaps the two covering sheets at $\star$.

The $\Aoo$--structure is defined by requiring strict unitality, with units $e_j$, and by obtaining the other $\Aoo$--operations by counting holomorphic polygons as follows. Given points $x_{k+1} \in V_{i_k} \pitchfork V_{i_{k+1}}$, where $k=0, \ldots, d-1$, and $i_0 < i_1 < \ldots < i_{d}$, we have:
\bq
\mu^d(x_d, x_{d-1}, \ldots, x_1 ) 
= \sum_{x_{i_0} \in V_{i_0} \pitchfork V_{i_d} } \nu (x_{i_0}; x_{i_1}, \ldots, x_{i_d}) x_{i_0}.
\eq
The integer $\nu (x_{i_0}; x_{i_1}, \ldots, x_{i_d})$ is a signed count of immersed holomorphic polygons. More precisely, we count homotopy classes of maps:
\bq
u: \{ z \in \C \textrm{ \, with \,} |z| \leq 1 \textrm{\, and \,} z \neq e^{2\pi i /(d+1)} \} \to M
\eq
such that $u$ is orientation-preserving, its image has convex corners, and, for all $k = 0, \ldots, d$,  the arc $\left\{ e^{2 \pi i t / (d+1)} \, \middle| \, t \in \left(\frac{k}{d+1}, \frac{k+1}{d+1}\right)  \right\}$ gets mapped to $V_k$ in such a way that $$\textrm{lim}_{z \to e^{2 \pi i k /(d+1)} }u(z)= x_{i_k}.$$
Following \cite{Lekili-Perutz-PNAS}, the sign attached to an immersion is $(-1)^{a+b+c}$, where
\begin{itemize}
\item  $a$ is $i (x_{i_0}) + i(x_{i_d})$ if the image of the  (positively oriented) boundary of the polygon traverses $V_{i_d}$ in the negative direction, and zero otherwise;
\item $b$ is the sum of $i(x_{i_k})$ for all $k \in \{ 1, \ldots, d \}$ such that the boundary of the polygon traverses $V_{i_k}$ negatively;
\item $c$ is the number of stars in the boundary of the image of the polygon. 
\end{itemize}

Let $R$ be the semi-simple ring $\C e_1 \oplus \ldots \oplus \C e_m$, where the $e_i$ are degree zero generators such that $e_i e_j = \delta_{ij} e_i$. An $\Aoo$--category with $m$ ordered objects, such as $\Fuk^{\to}(\Pi)$,  is the same data as an $\Aoo$--algebra over $R$; we shall use the same notation for both.

\subsubsection{Distinguished basis of vanishing cycles}

We choose the distinguished configuration of vanishing paths given by Figure \ref{fig:Vcyclesbaseorder}; this will turn out to be the most natural one to make the comparison with coherent sheaves on $Y_{p,q,r}$. They are ordered clockwise, and we choose to start at the dashed black line.

\begin{figure}[htb]
\begin{center}
\includegraphics[scale=0.40]{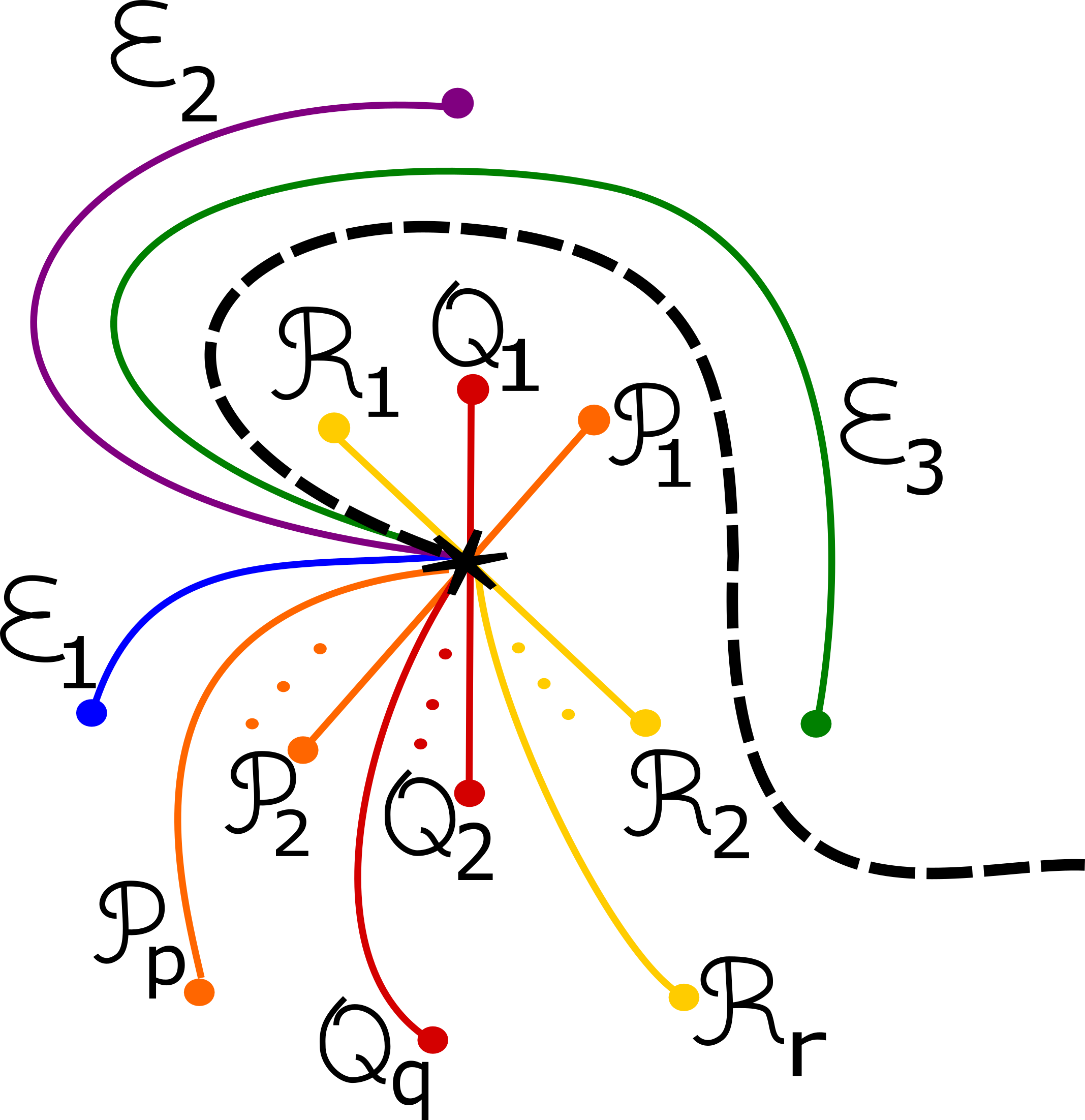}
% or [scale=0.85]
\caption{Distinguished basis of vanishing paths for $\pi$.}
\label{fig:Vcyclesbaseorder}
\end{center}
\end{figure}

 Let $M$ be the fibre above the distinguished smooth value $\star$. The vanishing paths induce vanishing cycles in $M$, which we label as 
\bq
\cR_1, \cQ_1, \cP_1, \cR_2, \ldots, \cR_r, \cQ_2, \ldots, \cQ_q, \cP_2, \ldots, \cP_p, \E_1, \E_2, \E_3
\eq 
as in Figure \ref{fig:Vcyclesbaseorder}. 
There are no intersection points between a $\cP_i$ and a $\cQ_j$, or $\cQ_i$ and $\cR_j$, or $\cP_i$ and $\cR_j$. Thus, after ``trivial'' mutations (where the two vanishing cycles involved are disjoint), we can use the following as our ordered distinguished collection of vanishing cycles:
\bq \label{eq:vanishing_cycles}
\cP_1, \ldots, \cP_p, \cQ_1, \ldots, \cQ_q, \cR_1, \ldots, \cR_r, \E_1, \E_2, \E_3.
\eq 
On the fibre $M$, these are given by the curves of Figure \ref{fig:3puncturedellipticcurve6cycles}. 
Note that the curves $\cP_j$, $j=1,\ldots, p$ are all Hamiltonian isotopic to each other; similarly for the $\cQ_j$ and for the $\cR_j$.

\begin{figure}[htb]
\begin{center}
\includegraphics[scale=0.35]{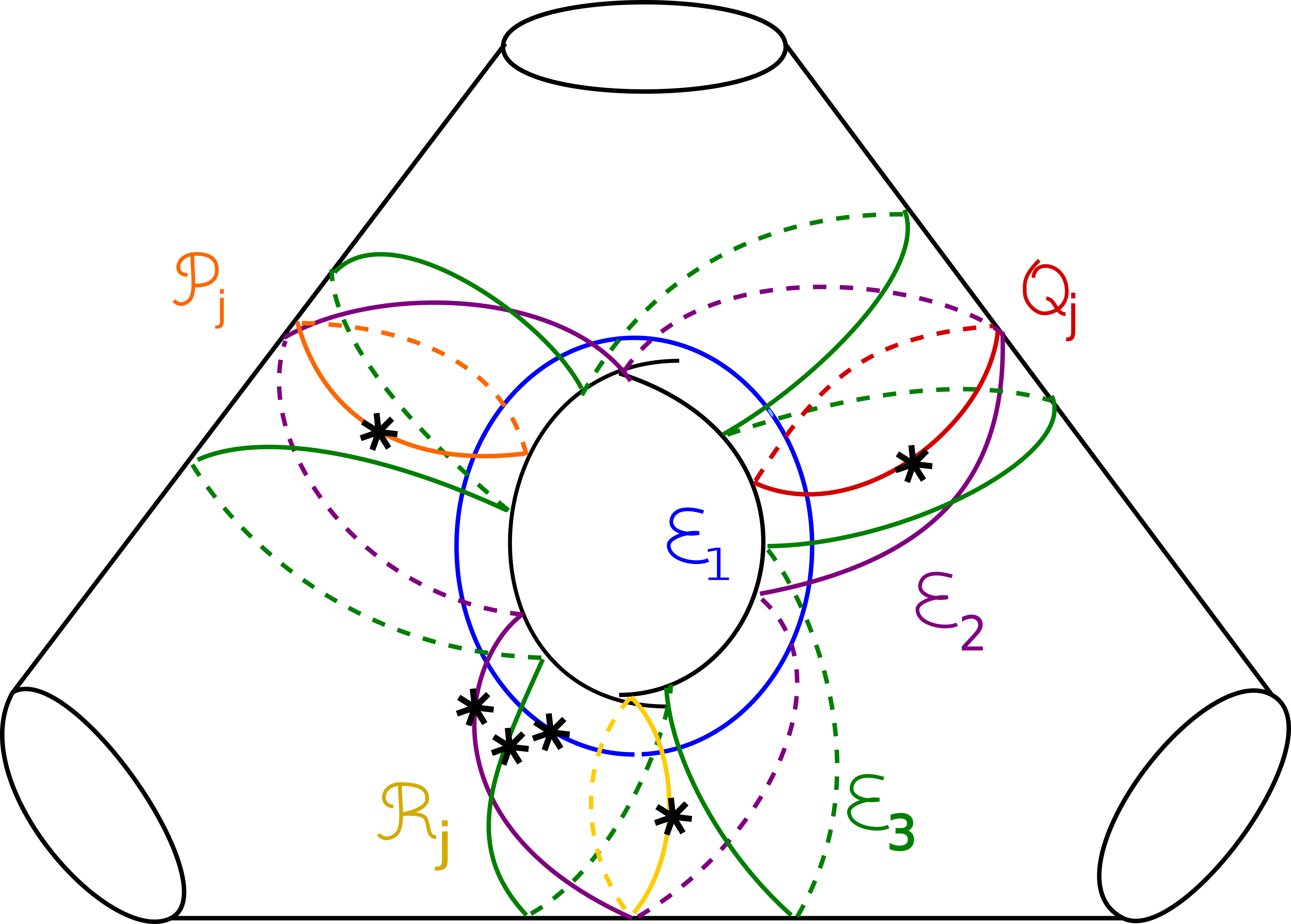}
% or [scale=0.85]
\caption{Basis of vanishing cycles for $\pi$, in the fibre $M$, and marked points for recording spin structures.}
\label{fig:3puncturedellipticcurve6cycles}
\end{center}
\end{figure}

\subsubsection{Lagrangian branes}\label{sec:fibration_branes}

We fix Hamiltonian deformations as follows: pick small deformations for the vanishing cycles $\cP_1, \ldots, \cP_p$ as in Figure \ref{fig:A_ntailintersections}, and ``nest'' these choices  into the elliptic curve $M$ following Figure \ref{fig:M5vcycles}. (This figure only shows $\cQ_1$, $\cQ_2$ and $\cQ_3$, but by choosing sufficiently small deformations, one can certainly fit them all.) We decorate each of the vanishing cycles with a marked point $\star$ recording the spin structure; see Figures \ref{fig:3puncturedellipticcurve6cycles}, \ref{fig:A_ntailintersections} and \ref{fig:M5vcycles}.
Where they will be relevant, we also record orientations.

The fibre $M$, as a Riemann surface, carries a natural orientation. Moreover, its tangent bundle is trivial.
Pick a trivialization $TM \cong M \times \C$ such that  the squared phase map $\alpha$ satisfies $\alpha (T\cP_1|_x) = 1$ for all $x \in \cP_1$, and similarly for $\cQ_1$ and $\cR_1$. 
We assign to $\cP_1$, $\cQ_1$ and $\cR_1$ the constant zero grading function. All of the $\cP_j$ are Hamiltonian isotopic to $\cP_1$; we equip them with the induced grading. In particular, for $i \neq j$, $\cP_i$ and $\cP_j$ intersect in two points, with gradings $0$ and $1$. Similarly for the $\cQ_j$ and the $\cR_j$. 
We assign to $\E_1$ a grading function such that the intersection point between $\cP_j$ and $\E_1$ has degree zero, for any $j=1,\ldots, p$. (Note we can pick our trivialization so that the squared phase map is identically $-1$ on $\E_1$.)
Then the intersection point between $\cQ_j$ and $\E_1$ and the one between $\cR_j$ and $\E_1$ both also have degree zero, again for all possible $j$. The cycles $\E_2$ and $\E_3$ are both results of Dehn twists of $\E_1$ in $\cP_1$, $\cQ_1$ and $\cR_1$ (one of each for $\E_2$ and two of each for $\E_3$). Assigning them the induced gradings ensures that the intersection points between any of the $\E_i$ and any of the $\cP_j$, $\cQ_j$ or $\cR_j$ have degree zero. Moreover, any of the intersection points between any two of the $\E_i$ also has degree zero. 

\begin{center}
\emph{In an abuse of notation, we denote our choices of Lagrangian branes by $\cP_1, \ldots, \E_3$ rather than $\cP_1^\#, \ldots, \E_3^\#$ in the hope that this will improve legibility.}
\end{center}

\begin{figure}[htb]
\begin{center}
\includegraphics[scale=0.35]{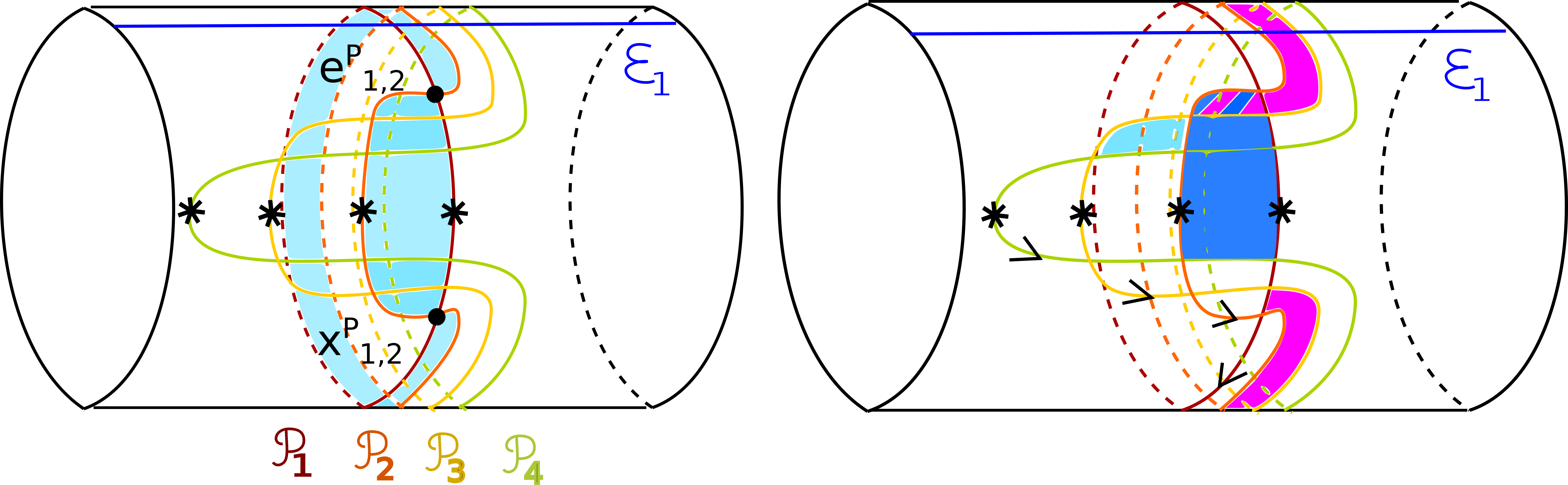}
% or [scale=0.85]
\caption{Choice of Hamiltonian deformations for the vanishing cycles $\cP_1, \ldots, \cP_p$, marked points on the cycles, orientations, and some holomorphic discs and polygons between them. They are shown on a subset of $M$, with the segments of $\E_2$ and $\E_3$ omitted.}
\label{fig:A_ntailintersections}
\end{center}
\end{figure}

\subsubsection{Holomorphic polygons and $\Aoo$--structure.}

Let $\A_\F = \Fuk^{\to} (\Xi)$ be the $\Aoo$--category induced by our choices of Hamiltonian deformations. It has the following feature. 

\begin{lemma}\label{lem:Ainfty_structure}
The category $\A_\F$ is minimal and formal: the only non-trivial $\Aoo$--structure map is $\mu^2$. This means that viewed as an $\Aoo$--algebra over $R$, $\A_\F$ is simply a graded $R$--algebra.
\end{lemma}

We will explicitly describe all of the holomorphic discs and triangles. In particular, we will describe the $\mu^2$ product structure; this is summarised in Lemma \ref{lem:full_quiver}. We will also show that there are no convex holomorphic $n$-gons, for $n >3$, contributing to the $\Aoo$--operations. Lemma \ref{lem:Ainfty_structure} will follow as an immediate consequence.

\begin{remark} We know  that $\E_2$ is the result of performing a single positive Dehn twist in each of $\cP_1$, $\cQ_1$ and $\cR_1$ on $\E_1$; iterating once gives $\E_3$. However, we choose to understand $\Aoo$--operations geometrically rather than by manipulating iterated cones. This has the double advantage of readily yielding a model for the category which is minimal, and of by-passing the complications resulting from the fact that the Fukaya category of $M$ is not formal (see \cite{Lekili-Perutz-PNAS}; there is a full embedding of the Fukaya category of the once punctured torus into $\Fuk(M)$). 
\end{remark}

\subsubsection*{Holomorphic discs}
For every $i, j$ with $1 \leq i < j \leq p$, we have
\bq
CF^\ast (\cP_i, \cP_j) \cong \C\langle e^{\cP}_{i,j} \rangle \oplus \C \langle x^{\cP}_{i,j} \rangle
\eq
where $e^{\cP}_{i,j}$ is a generator in degree 0, and $x^{\cP}_{i,j}$ is a generator in degree one. (See  Figure \ref{fig:A_ntailintersections}.) There are two holomorphic discs between $e^{\cP}_{i,j} $ and $x^{\cP}_{i,j}$, with cancelling signs. These are shaded in blue on the left-hand side of Figure \ref{fig:A_ntailintersections}. Thus $CF^\ast (\cP_i, \cP_j)$ has $\mu^1 = 0$. Similarly for pairs of $\cQ_j$ and of $\cR_j$, with intersection points labelled as $e^{\cQ}_{i,j} $, $x^{\cQ}_{i,j}$, $e^{\cR}_{i,j} $ and $x^{\cR}_{i,j}$ respectively. With our choices of vanishing cycles, by inspection, there are no other holomorphic bigons. Thus the $\Aoo$--structure on $\Fuk^{\to} (\Xi)$ is minimal.

\subsubsection*{Holomorphic triangles}

We label the remaining intersection points as follows: $y_{i,k}^{\cP} \in CF^\ast (\cP_i, \E_k)$ denotes the unique intersection point of $\cP_i$ and $\E_k$, and similarly for $y_{i,k}^{\cQ} \in CF^\ast (\cQ_i, \E_k)$ and $y_{i,k}^{\cR} \in CF^\ast (\cR_i, \E_k)$. The intersection points between $\E_1$ and $\E_2$ are labeled as $a_1$, $a_2$ and $a_3$; those between $\E_2$ and $\E_3$, as $c_1$, $c_2$ and $c_3$; and those between $\E_1$ and $\E_3$, as $b_1$, $b_{1,2}$, $b_2$, $b_{2,3}$, $b_3$ and $b_{3,1}$. See Figure \ref{fig:Mholotriangles} for details.
As all intersection points have index zero, we don't need to keep track of orientations of curves to get the signs of the $\Aoo$--products appearing in that figure, so we do not record orientations. (Similarly with subsequent pictures.)

\begin{figure}[htb]
\begin{center}
\includegraphics[scale=0.35]{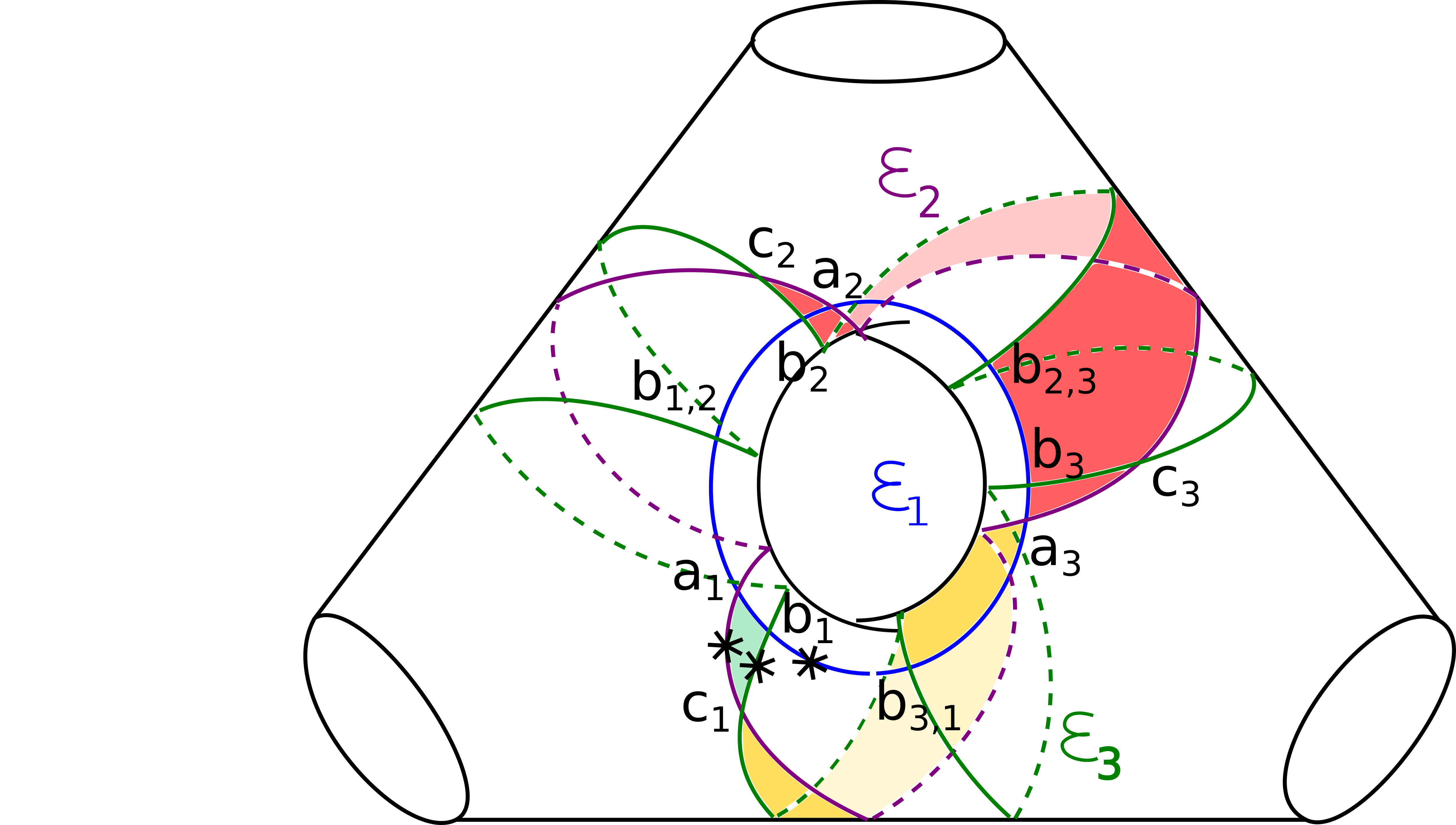}
% or [scale=0.85]
\caption{Examples of holomorphic triangles involving $\E_1, \E_2$ and $\E_3$, and labels for the intersection points between them.}
\label{fig:Mholotriangles}
\end{center}
\end{figure}

Let us start with $\mu^2$ products between the $\E_i$. Because of the ordering, the only possibility is the product $CF^\ast (\E_2, \E_3) \otimes CF^\ast (\E_1, \E_2 ) \to CF^\ast (\E_1, \E_3 )$. We read off the following:
\bq
\mu^2 (c_i, a_i) = b_i \qquad \mu^2 (c_i, a_{i-1}) = b_{i-1,i} \qquad \mu^2(c_{i-1}, a_i) = b_{i-1,1}
\eq
where indices $i$ are considered mod $3$. Figure \ref{fig:Mholotriangles} gives an example of a holomorphic triangle for each of the three cases. Equivalently, the full subcategory on $\{ \E_1, \E_2, \E_3 \}$ is given by the quiver with relations:
\bq \label{eq:P^2quiver}
\xymatrix{
\bullet  \ar[r]^{a_2} \ar@/^1.5pc/[r]^{a_1}  \ar@/_1.5pc/[r]^{a_3}  & \bullet   \ar[r]^{c_2} \ar@/^1.5pc/[r]^{c_1}  \ar@/_1.5pc/[r]^{c_3} & \bullet
} 
\text{\, \, where \,}
\begin{cases} a_i c_{i+1} = a_{i+1} c_i \,\, (= b_{i, i+1} )\\ (a_i c_i = b_i ) 
\end{cases}
\eq

Next, let us consider the products involving one $\cP_i$ (or $\cQ_i$, or $\cR_i$), and two of the $\E_j$. 
Figure \ref{fig:Mholotriangles2} shows holomorphic triangles involved in the following:
\begin{align}
\mu^2: \,\, & CF( \E_1, \cP_j) \otimes CF (\E_3, \E_1) \to CF (\E_3, \cP_j) \\
\mu^2: \,\, & CF( \E_1, \cQ_j) \otimes CF (\E_2, \E_1) \to CF (\E_2, \cQ_j) \\
\mu^2: \,\, & CF( \E_2, \cR_j) \otimes CF (\E_3, \E_2) \to CF (\E_3, \cR_j) 
\end{align}
In the case of $(\E_1, \E_2)$ and $(\E_2, \E_3)$, there are two triangles contributing to the product; in the case of $(\E_1, \E_3)$, there are three. These give the $\mu^2$ operations:
\begin{align}
& \mu^2 (y^{\cP}_{j,1} , b_{1} ) = \mu^2 (y^{\cP}_{j,1} , b_{1,2} ) =  \mu^2 (y^{\cP}_{j,1} , b_{2} )  = y^{\cP}_{j,3} \\
& \mu^2 (y^{\cQ}_{j,1} , a_2 ) = \mu^2 (y^{\cQ}_{j,1} , a_3 )  =  y^{\cQ}_{2,j} \\
& \mu^2 (y_{j,2}^{\cR} , c_3) = \mu^2 (y_{j,2}^{\cR} , c_1) = y^{\cR}_3 
\end{align}
where the signs follow from the formula at the end of Section \ref{sec:directed_Fukaya_background}. 
The six remaining products are analogous:
\begin{align}
\mu^2 (y^{\cP}_{j,1} , a_1 )  =  \mu^2 (y^{\cP}_{j,1} , a_2 )  =   y^{\cP}_{j,2}  \qquad 
\mu^2 (y^{\cP}_{j,2} , c_1 ) =  \mu^2 (y^{\cP}_{j,2} , c_2 ) =  y^{\cP}_{j,3} \\
\mu^2 (y^{\cQ}_{j,1} , b_{2} ) =\mu^2 (y^{\cQ}_{j,1} , b_{2,3} )  = \mu^2 (y^{\cQ}_{j,1} , b_{3} )  = y^{\cQ}_{j,3}  \qquad
\mu^2 (y^{\cQ}_{j,2} , c_2 )  = \mu^2 (y^{\cQ}_{j,2} , c_3)  =  y^{\cQ}_{j,3} \\
\mu^2 (y^{\cR}_{j,1} , b_{3} ) = \mu^2 (y^{\cR}_{j,1} , b_{3,1} ) = \mu^2 (y^{\cR}_{j,1} , b_1 )  = y^{\cR}_{j,3} \qquad \mu^2 (y^{\cR}_{j,1} , a_3 )  = \mu^2 (y^{\cR}_{j,1} , a_1 ) =  y^{\cR}_{j,2} 
\end{align}
In words, for any $i$, the intersection point of $\cP_i$ with $\E_1$, respectively $\E_2$, has a non-zero product with the generators $a_j, b_j, b_{j,k}$, respectively $c_j$, for $j , k \in \{1,2\}$, and said product is the distinguished generator in the target chain complex. For the $\cQ_i$, it is for $j,k \in \{2,3\} $, and for the $\cR_i$, $j,k \in \{ 1,3 \}$. (The reader may also wish to look ahead to Lemma \ref{lem:full_quiver} for a more synthetic description.)
 
\begin{figure}[htb]
\begin{center}
\includegraphics[scale=0.35]{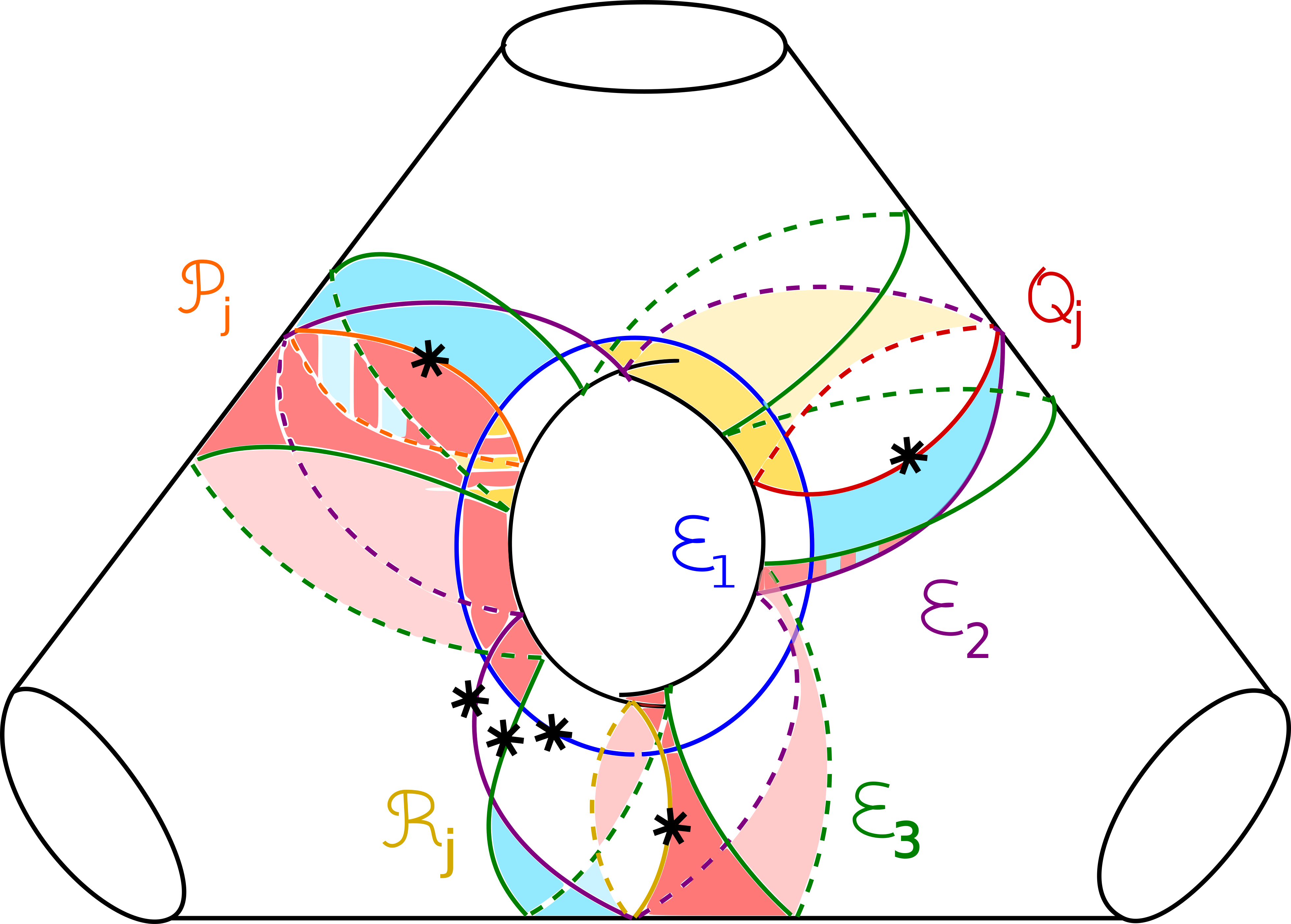}
% or [scale=0.85]
\caption{Examples of holomorphic triangles counting products between $\cP_j$, $\cQ_j$ or $\cR_j$ and the $\E_i$.
}
\label{fig:Mholotriangles2}
\end{center}
\end{figure}

Now consider products between $\cP_{i_1}$, $\cP_{i_2}$ and $\E_k$; for ordering reasons, the only possibility is 
$CF^\ast (\cP_{i_2} , \E_k) \otimes CF^\ast (\cP_{i_1}, \cP_{i_2}) \to CF^\ast (\cP_{i_1}, \E_k)$, 
for $i_1 < i_2$. By inspection and considering gradings, we see that the only such product is
\bq
\mu^2 (y_{i_2, k}^{\cP}, e_{i_1, i_2}^{\cP} ) = y_{i_1, k}^{\cP}.
\eq
See Figure \ref{fig:M5vcycles} for an example of a relevant holomorphic triangle.
\begin{figure}[htb]
\begin{center}
\includegraphics[scale=0.35]{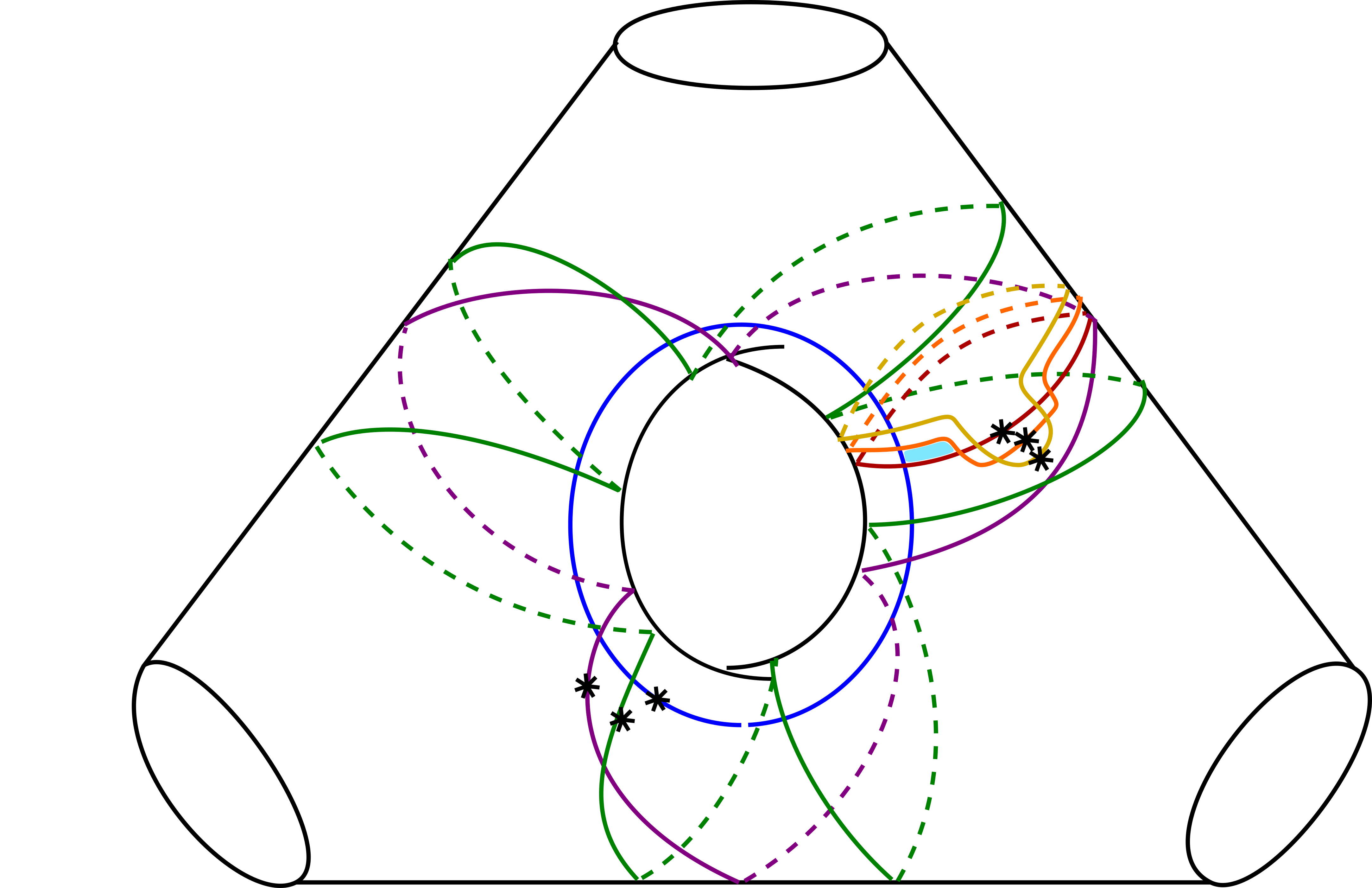}
% or [scale=0.85]
\caption{Nesting each $A_n$ tail on the 3-punctured ellipctic curve $M$ (only $\cQ_1$, $\cQ_2$ and $\cQ_3$ shown), and an example of a holomorphic triangle.}
\label{fig:M5vcycles}
\end{center}
\end{figure}

Consider products between $\cP_{i_1}, \cP_{i_2}$ and $\cP_{i_3}$,
 i.e.~maps $\mu^2: \, CF^\ast (\cP_{i_2}, \cP_{i_3}) \otimes CF^\ast (\cP_{i_1}, \cP_{i_2}) \to 
CF^\ast (\cP_{i_1}, \cP_{i_3}) $, for $i_1 < i_2 < i_3$. Again, the only possibility comes from ``multiplying by the unit'':
\bq
\mu^2 (e_{i_2, i_3}^{\cP} , e_{i_1, i_2}^{\cP} ) = e_{i_1, i_3}^{\cP}
\qquad
\mu^2 (x_{i_2, i_3}^{\cP} , e_{i_1, i_2}^{\cP} ) = x_{i_1, i_3}^{\cP}
\qquad
\mu^2 (e_{i_2, i_3}^{\cP} , x_{i_1, i_2}^{\cP} ) = - x_{i_1, i_3}^{\cP}.
\eq
Instances of relevant holomorphic triangles are given on the right-hand side of Figure \ref{fig:A_ntailintersections}. Similarly for the $\cQ_i$ and $\cR_i$.

As there are no intersection points between any $\cP_i$ and $\cQ_j$ (and, similarly, $\cQ_i$ and $\cR_j$, or $\cP_i$ and $\cR_j$), there are no other possibilities for products.

\begin{remark}\label{rmk:P^2model}
 After a mutation, the ordered collection $\{ \E_1, \E_2, \E_3 \} $ becomes $\{ L_1 = \E_1, L_2 = \tau_{\E_2} \E_3, L_3 = \E_2 \}$, as in Figure \ref{fig:MVCMcycles}. These are precisely the vanishing cycles for the ``standard'' Lefschetz fibration $(\C^\ast)^2 \to \C$, $(x,y ) \mapsto x+y+\frac{1}{xy}$, as given e.g.~in \cite[Figure 2]{Seidel_mvcm}. These choices lead to the quiver description of \cite[Proposition 3.2]{Seidel_mvcm}. 
  (See also \cite[Section 6]{Seidel_subalgebras}.)
\begin{figure}[htb]
\begin{center}
\includegraphics[scale=0.25]{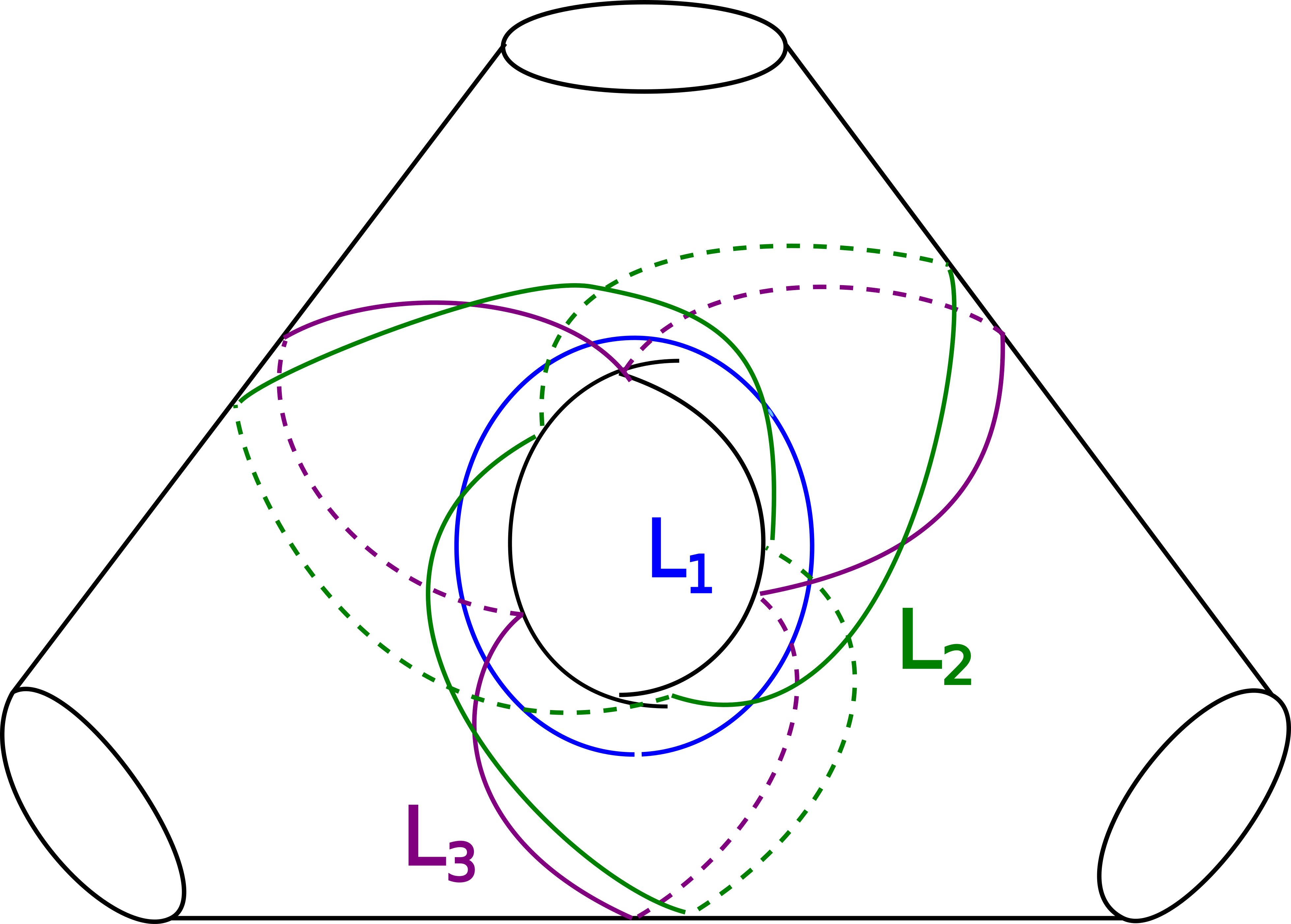}
% or [scale=0.85]
\caption{The vanishing cycles $L_1, L_2$ and $L_3$.}
\label{fig:MVCMcycles}
\end{center}
\end{figure}
\end{remark}

\subsubsection*{Holomorphic polygons with more than three edges}

We claim that there are no convex holomorphic $n$-gons  contributing to $\mu^{n-1}$ for $n >3$. First, notice that no such $n$-gon can have boundary on both a $\cP_i$ and a $\cQ_j$: as these don't intersect, one would have to use two  of the $\E_k$ to get between them (or the same one twice), and the order of the Lagrangians on the boundary of the $n$-gon would violate the order of the distinguished vanishing cycles. (Similarly for $n$-gons involving, say, both a $\cQ_i$ and a $\cR_j$.)

By inspection of Figure \ref{fig:A_ntailintersections}, there are no convex holomorphic $n$-gons with boundary only on the $\cP_i$ and for which the distinguished order of the vanishing cycles is respected. 

It remains to exclude homolorphic $n$-gons with order boundary segments on $\cP_{i_1}, \cP_{i_2}, \ldots, \cP_{i_l}$ (some $i_1 < i_2 < \ldots < i_n$), followed by at least one of the $\E_i$. (Similarly for sequences of $\cQ_i$ and $\cR_i$.) By grading considerations, all of the elements in $hom(\cP_{i_k}, \cP_{i_{k+1}})$ must be $e^{\cP}_{i_k, i_{k+1}}$. 
Consider first the case of a sequence ending with a unique $\E_i$; see e.g.~Figure \ref{fig:A_ntailintersections} for the case of $\E_1$, the others having an identical local model.  By inspection, the only holomorphic polygons whose boundary segments respect the order of the vanishing cycles have at least one concave corner, at an $e^{\cP}_{i_k, i_{k+1}}$, which excludes them from consideration.
Now, look at sequences ending with two or three of the $\E_i$; see Figure \ref{fig:Mholoquadrilateral} for an example. 
If there are two of more $\cP_j$ at the start of the sequence, the previous considerations on concave corners still apply -- -- see e.g.~the yellow region in Figure \ref{fig:Mholoquadrilateral}. 
This leaves the possibility of a $\mu^3$ product between a $\cP_i$ and all three of the $\E_i$. 
These can be excluded by inspection, by combining the fact that each of the corners of the polygons must be convex with the fact that its boundary must be a closed contractible loop -- see e.g.~the blue polygon (with a concave corner) in  Figure \ref{fig:Mholoquadrilateral}. 

\begin{figure}[htb]
\begin{center}
\includegraphics[scale=0.35]{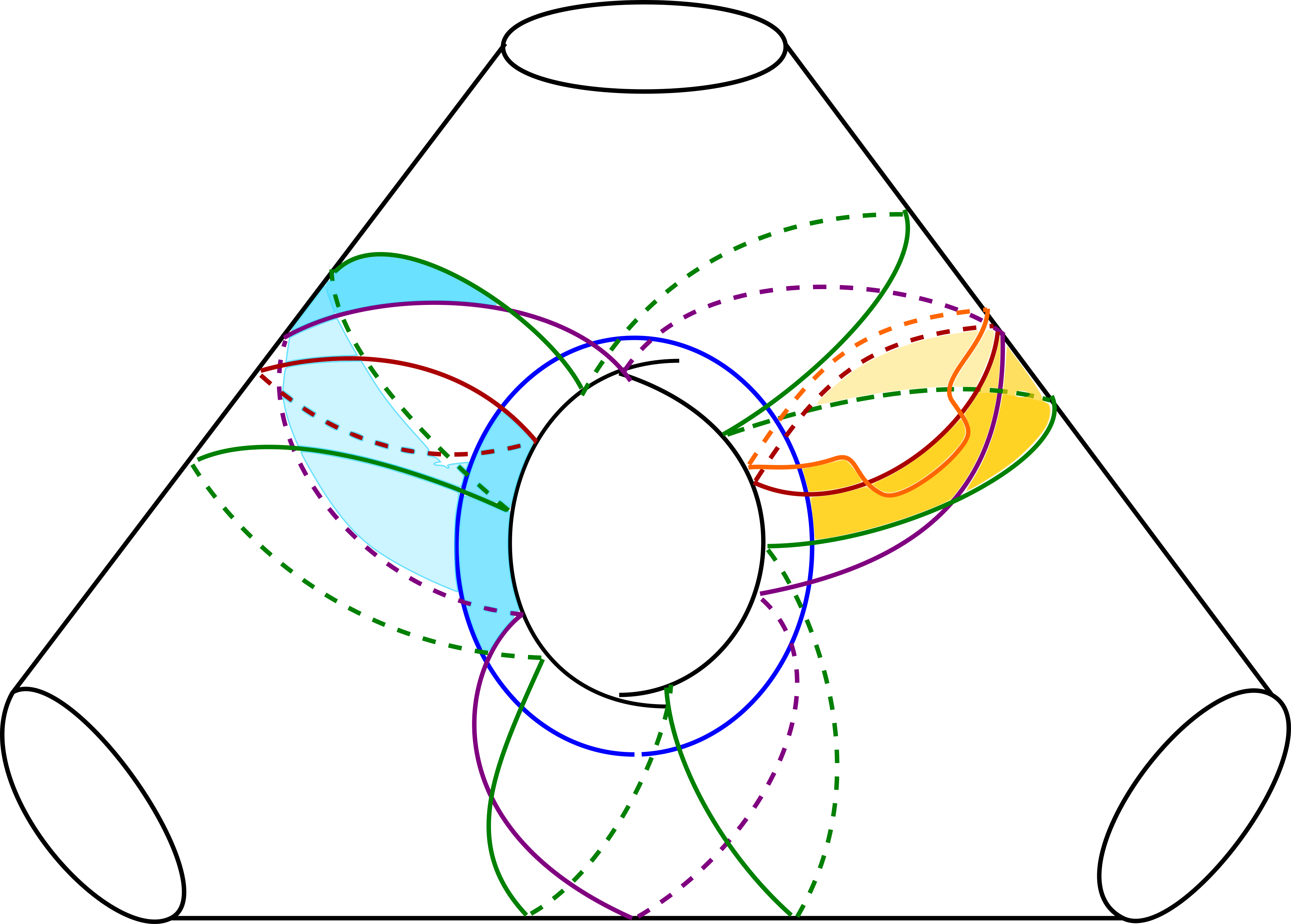}
% or [scale=0.85]
\caption{Holomorphic quadrilateral between $\cP_j, \cP_{j+k}, \E_1, \E_3$ with a convex corner (yellow), and between $\cQ_i$, $\E_1$, $\E_2$ and $\E_3$ (blue). 
}
\label{fig:Mholoquadrilateral}
\end{center}
\end{figure}

This concludes the proof of Lemma \ref{lem:Ainfty_structure}. 

\begin{remark} The Fukaya category of the two-variable $A_1$ Milnor fibre (topologically, a cylinder) is intrinsically formal. One could use this, together with inclusions of the type considered later in Section \ref{sec:Fuk_restriction}, to deduce that for some quasi-isomorphic copy of $\mathcal{A}_\F$, there are no $\mu^k$ products between the $\cP_i$, for any $k >2$ (ditto $\cQ$ and $\cR$). However, as the Fukaya category of $M$ itself is not formal, this would not readily give formality of $\A_\F$. 
\end{remark}

Finally, we record the structure of $\A_\F$ as a graded $R$--algebra.

\begin{lemma}\label{lem:full_quiver}
The product structure of $\A_\F$ is given by the following quiver algebra, where we are ignoring the unit $e_i \in hom(V_i, V_i)$, for all Lagrangian branes $V_i$ (and products involving it):
\begin{equation*}
\xymatrix@C=3.8em@R=4em{
\cP_1 \ar@/^1pc/[r]^-{x^{\cP}_{1,2}} \ar@/_1pc/[r]^-{e^{\cP}_{1,2}} & \cP_2 \ar@/^1pc/[r]^-{x^{\cP}_{2,3}} \ar@/_1pc/[r]^-{e^{\cP}_{2,3}} & \cP_3 \ar@{.>}@/^1pc/[r] \ar@{.>}@/_1pc/[r] &  \cP_{p-1} \ar@/^1pc/[r]^-{x^{\cP}_{p-1,p}} \ar@/_1pc/[r]^-{e^{\cP}_{p-1,p}} & \cP_p  \ar@/_0pc/[dr]^-{y^{\cP}_{p,1}} & &  &
 \\
\cQ_1 \ar@/^1pc/[r]^-{x^{\cQ}_{1,2}} \ar@/_1pc/[r]^-{e^{\cQ}_{1,2}} & \cQ_2 \ar@/^1pc/[r]^-{x^{\cQ}_{2,3}} \ar@/_1pc/[r]^-{e^{\cQ}_{2,3}} & \cQ_3 \ar@{.>}@/^1pc/[r] \ar@{.>}@/_1pc/[r] &  \cQ_{q-1} \ar@/^1pc/[r]^-{x^{\cQ}_{p-1,p}} \ar@/_1pc/[r]^-{e^{\cQ}_{q-1,q}} & \cQ_q  \ar@/_0pc/[r]^-{y^{\cQ}_{q,1}}
 & \E_1  \ar[r]^{a_2} \ar@/^1.5pc/[r]^{a_1}  \ar@/_1.5pc/[r]^{a_3}  & \E_2   \ar[r]^{c_2} \ar@/^1.5pc/[r]^{c_1}  \ar@/_1.5pc/[r]^{c_3} & \E_3
\\
\cR_1 \ar@/^1pc/[r]^-{x^{\cR}_{1,2}} \ar@/_1pc/[r]^-{e^{\cR}_{1,2}} & \cR_2 \ar@/^1pc/[r]^-{x^{\cR}_{2,3}} \ar@/_1pc/[r]^-{e^{\cR}_{2,3}} & \cR_3 \ar@{.>}@/^1pc/[r] \ar@{.>}@/_1pc/[r] &  \cR_{r-1} \ar@/^1pc/[r]^-{x^{\cR}_{p-1,p}} \ar@/_1pc/[r]^-{e^{\cR}_{r-1,r}} & \cR_r  \ar@/_0pc/[ur]_-{y^{\cR}_{r,1}} & &  &
} 
\end{equation*}
subject to the relations
\begin{eqnarray}
x^{\cP}_{i,i+1} \cdot e^{\cP}_{i+1, i+2} &= & -e^{\cP}_{i,i+1} \cdot x_{i+1,i+2}^{\cP} \\
x^{\cP}_{p-1,p} \cdot y^{\cP}_{p,1} &=& 0 \\
 a_i \cdot c_{i+1} & = & a_{i+1} \cdot c_i \\
 y^{\cP}_{p,1} \cdot a_3 = 0 \quad  y^{\cQ}_{q,1} \cdot a_1 & = & 0  \quad y^{\cR}_{r,1} \cdot a_2 = 0 
\end{eqnarray}
$$ \text{Products with an \,}x^{\cP}_{i,i+1} \text{\,are zero except when all other elements are $e^{\cP}_{l,l+1}$.} $$
Similarly for $\cQ$ and $\cR$, for those relations above only involving $\cP$. 
All the other morphisms are the products of ones in the quiver. For instance, $x^{\cP}_{1,2} \cdot e^{\cP}_{2,3} = x^{\cP}_{1,3}$, and $e^{\cP}_{i, i+1} \cdot \ldots \cdot e^{\cP}_{p-1,p} \cdot y^{\cP}_{p,1} = y^{\cP}_{i, 1}$. Also, for instance $y^{\cP}_{p,1} \cdot a_1 = y^{\cP}_{p,1} \cdot a_2 = y^{\cP}_{p,2}$. Finally, as before 
$ a_i \cdot c_{i+1} =  b_{i, i+1} $ and $a_i \cdot c_i = b_i$. 
\end{lemma}

\subsection{$D^b \Coh (Y_{p,q,r})$: semi-orthogonal decomposition and product structure}

\subsubsection{Preliminaries}

Our convention is that unless otherwise stated, given any map $f: X \to Y$ of algebraic varieties, $f^\ast$ will denote the \emph{left-derived} pull-back map, and $f_\ast $ will denote the \emph{right-derived} push-forward map. We start by recalling the definitions we shall use, and refer the reader to the textbook \cite{Huybrechts} for further background.

\begin{definition}\cite[Definition 1.59]{Huybrechts}
A semi-orthogonal decomposition of a $\C$--linear triangulated category  $\mathcal{D}$ is a sequence of subcategories $\mathcal{D}_1, \ldots, \mathcal{D}_n$ such that:
\begin{itemize}
\item 
 Let $\mathcal{D}_j^{\perp}$ be the right-orthogonal to  $\mathcal{D}_j$, i.e.~the full subcategory with objects $C$ such that for all objects $D$ in $\mathcal{D}_j$, one has $Hom(D,C)=0$. Then for all $0 \leq i < j \leq n$, we have that $\mathcal{D}_i \subset \mathcal{D}_j^{\perp}$.
\item The sequence of subcategories $\mathcal{D}_i$ generates $\mathcal{D}$. This means that the smallest full subcategory of $\D$ containing all of the $\D_i$ is equivalent to $\D$ itself.
\end{itemize}
\end{definition}

\begin{definition}\cite[Definition 1.57]{Huybrechts}
An object $E$ in a $\C$--linear triangulated category $\D$ is exceptional if 
\bq
\text{Hom} (E, E[l]) = 
\begin{cases}
\C \text{\, if \,} l=0 \\
0 \text{\, otherwise}
\end{cases}
\eq
An exceptional sequence is a sequence $E_1, \ldots, E_n$ of exceptional objects such that $Hom(E_i, E_j[l]) = 0$ for all $i >j$ and all $l$. It is called full if $\D$ is generated by the $E_i$, i.e.~the smallest full subcategory of $\D$ containing all of the $E_i$ is $\D$ itself. 
\end{definition}

We will use the following particular case of a theorem of Bondal and Orlov.
\begin{theorem} \cite{Orlov}, \cite[Theorem 4.2]{Bondal-Orlov} \label{thm:BO}
Let $X$ be a smooth algebraic surface, and $x$ a point of $X$. Let $\widetilde{X}$ denote the blow-up of $X$ at $x$, and let $E$ be the associated exceptional divisor. We have the following semi-orthogonal decomposition of the derived category of coherent sheaves on $\widetilde{X}$:
\bq
D^b \Coh (\widetilde{X}) = \langle i_{\ast} \mathcal{O}_E (-1) , \pi^{\ast} D^b \Coh(X)           \rangle
\eq
where $i: E \hookrightarrow \widetilde{X}$ is the inclusion map, $\pi: \widetilde{X} \to X$ the blow-down map, and $ i_{\ast}$ and $\pi^{\ast}$ the resulting right, resp.~left, derived push-forward and pull-back maps. Moreover, $\pi^\ast$ gives a fully faithful embedding of $D^b \Coh(X)$ into $D^b \Coh(\widetilde{X})$. 
\end{theorem}

To distinguish between uses of the full statement above versus those of its final sentence (which is comparatively elementary), we record the latter as a stand-alone lemma.
\begin{lemma}\label{lem:ext_pullbacks}
 Let $X$ be a smooth projective variety, $\widetilde{X}$ the blow-up of $X$ at a point, and $\pi: \widetilde{X} \to X$ the blow-down map. Fix $E, F \in D^b \Coh (X)$. Then
\bq
\Ext_{\widetilde{X}}^i (\pi^\ast E, \pi^\ast F ) \cong \Ext_X^i (E, F).
\eq
\end{lemma}

Note that $i_\ast \O_E(-1) $ is an exceptional object of $D^b \Coh(\widetilde{X})$. One can either show this from first principles using the standard short exact sequence 
\bq 0 \to \O_{\widetilde{X}}(-E) \to \O_{\widetilde{X}} \to i_\ast \O_E \to 0 \eq
and the push-pull lemma, or by applying \cite[Proposition 11.8]{Huybrechts}.

\subsubsection{Exceptional sequence of objects}

Let $Y_{p,q,r}$ be the projective variety obtained as the result of   $p+q+r$  (iterated) blow-ups of $\P^2$, described in the introduction. 
Let us recall the construction. 
Start with three collinear points,  each on a different component of the toric divisor on $\P^2$ -- for instance, $[1:-1:0], [0:1:-1]$ and $ [-1:0:1] $. 
Let $H$ be the complex line that they all belong to -- e.g. $u+v+w=0$, where $u$, $v$ and $w$ are the homogeneous coordinates on $\P^2$. Perform the following blow-ups:
\begin{itemize}
\item Blow up the point $[1:-1:0]$. The resulting exceptional divisor, say $E_{\cP, p}$, intersects the strict transform of the toric divisor in one point. Blow up that point. The new exceptional divisor, say $E_{\cP,p-1}$, in turn intersects the strict transform of the toric divisor in one point. Iterate to perform a total of $p$ blow ups. The strict transforms of exceptional divisors form a chain of $-2$ rational curves: $\widetilde{E}_{\cP, p}, \ldots, \widetilde{E}_{\cP, 2}$, with a single $-1$ rational curve at the end of it (intersecting $\widetilde{E}_{\cP, p-1}$): $E_{\cP, 1}$. (The reason for the decreasing labelling will become clear: it is for it to later agree with the order of elements of an exceptional sequence.)
\item Perform a completely analogous procedure with $[0:1:-1]$, with $q$ blow-ups. Call the resulting chain of strict transforms of exceptional divisors $\widetilde{E}_{\cQ, q}, \ldots, \widetilde{E}_{\cQ, 2}$; again, there is a single rational $-1$ curve at the end of the chain: $E_{\cQ, 1}$.
\item Similarly again for $[-1:0:1]$, with $r$ blow-ups. The chain of strict transforms of exceptional divisors will be denoted  $\widetilde{E}_{\cR, r}, \ldots, \widetilde{E}_{\cR, 2}$; there is a single $-1$ curve at the end of the chain, intersecting $\widetilde{E}_{\cR, r-1}$: $E_{\cR, 1}$.
\end{itemize}

\begin{figure}[htb]
\begin{center}
\includegraphics[scale=0.55]{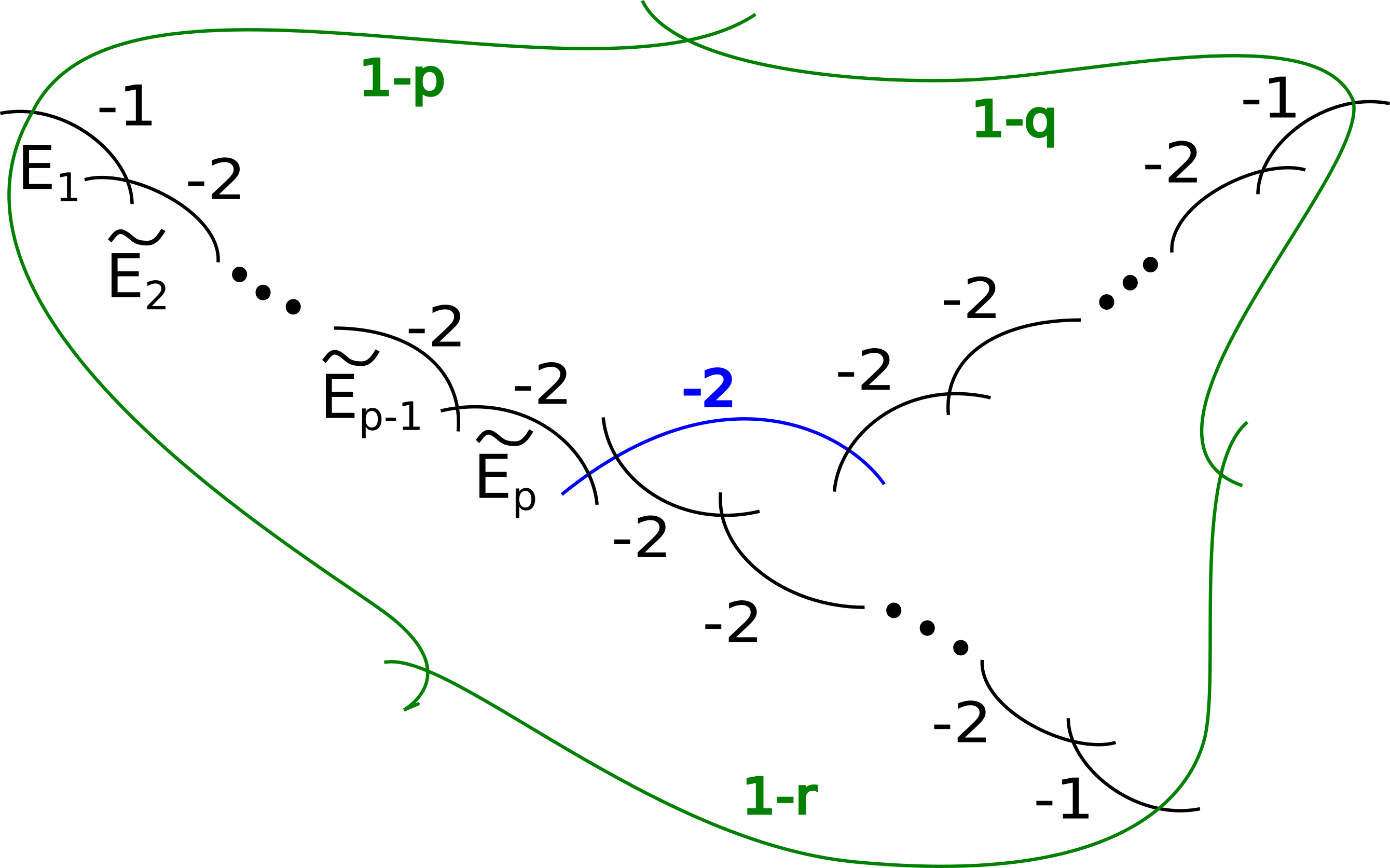}
% or [scale=0.85]
\caption{Curves in $Y_{p,q,r}$, with their self-intersections: $\widetilde{H}$ in blue, $D$ in green, and the chains of strict transforms of exceptional divisors in black. For instance, $\widetilde{E}_i$ denotes $\widetilde{E}_{\cP, i}$.
}
\label{fig:-2curve}
\end{center}
\end{figure}

Denote by  $\widetilde{H}$  the strict transform of $H$, which is also a $-2$ curve, and by $D$ the strict transform of the toric anticanonical divisor $\{ uvw=0 \}$ on $\P^2$. 
The configuration of curves in $Y_{p,q,r}$ is given in Figure \ref{fig:-2curve}. 

\begin{definition}\label{def:exceptional_sheaves}
We introduce the following notation:
\bq
\D_{\cP, j} =
\begin{cases}
\pi^\ast i_\ast   \O_{E_{\mathcal{P},j}}(-1) & \text{if  } 2\leq j \leq p \\
 i_\ast   \O_{E_{\mathcal{P},j}}(-1)  & \text{if  } j=1
\end{cases}
\eq
Here $\pi$ is used to denote several different blow-downs (or sequences thereof): from $Y_{p,q,r}$ to the surface, say $\widetilde{X}$, that is the intermediate surface in the sequence of blow ups described above for which the new exceptional divisor is $E_{\cP, j}$.  $\O_E$ denotes the structure sheaf of $E$ in $D^b \Coh(E)$, and $i$  the inclusion $i: E_{\cP, j} \to \widetilde{X}$. 
We define $\D_{\cQ, j}$, for $1 \leq j \leq q$ and $\D_{\cR, j}$, for $1 \leq j \leq r$, analogously.
\end{definition} 

Theorem \ref{thm:BO} readily yields the following semi-orthogonal decomposition for $Y_{p,q,r}$:
\bq
\langle
\D_{\cP,1}, \D_{\cP,2} \ldots, \D_{\cP, p}, \D_{\cQ, 1}, \ldots, \D_{\cQ, q}, \D_{\cR, 1}, \ldots,  \D_{\cR, r}, \pi^\ast D^b \Coh(\P^2)
\rangle
\eq
An exceptional sequence for $D^b \Coh (\P^2)$ is given by $\O_{\P^2}, \O_{\P^2}(1), \O_{\P^2}(2)$ (Beilinson, see \cite[Corollary 8.29]{Huybrechts}). This yields the following.

\begin{lemma}\label{lem:exceptional_sequence}
A full exceptional sequence for $D^b \Coh(Y_{p,q,r})$ is given by:
\begin{multline}
\qquad \D_{\cP,1}, \D_{\cP,2} \ldots, \D_{\cP, p}, \D_{\cQ, 1}, \ldots, \D_{\cQ, q}, \D_{\cR, 1}, \ldots, \D_{\cR, r}, \\
\pi^{\ast} \O_{\P^2} ,
\pi^{\ast} \O_{\P^2}(1) ,
\pi^{\ast} \O_{\P^2}(2). \qquad
\end{multline}
\end{lemma}

\begin{definition}
Let $A_\text{C} \subset D^b \Coh (Y_{p,q,r})$ be the full subcategory on the $p+q+r+3$  objects of the full exceptional sequence. It can be viewed as an algebra over $R$, the same semi-simple ring as in Section \ref{sec:directed_Fukaya_background}.
\end{definition}

\begin{definition}
Let $A_\F $ be the cohomology category of $\A_\F$. As $\A_\F$ is minimal and formal, it is just given by changing signs for products in $\A_\F$:
\bq
b \cdot a = (-1)^{|a|} \mu^2_{\A_\F} (b,a).
\eq
\end{definition}

\begin{proposition}\label{prop:isomorphic_sequences} 
$A_\F$ and $A_\tC$ are isomorphic as $R$--algebras. 
 More precisely, consider the following map $\phi_\A$ from the objects of $\A_\F$ to the objects of $A_\tC$, respecting their ordering as exceptional sequences:
\begin{eqnarray}
\phi_\A ( \cP_i ) =  \D_{\cP, i}&\textrm{ \, for \, } i=1,2,\ldots, p \\
\phi_\A ( \cQ_i  ) = \D_{\cQ, i}&\textrm{ \, for \, } i=1,2,\ldots, q \\
\phi_\A ( \cR_i ) = \D_{\cR, i} &\textrm{ \, for \, } i=1,2,\ldots, r \\
\phi_\A (\E_i  )  = \pi^\ast \O_{\P^2} (-1+i) &\textrm{ \, for \, }i=1,2, \textrm{and \,} 3 
\end{eqnarray} 
Then for any objects $X$, $Y$ in $\A_\F$, there are isomorphisms
\bq
\phi_\A^{X,Y} : \, hom^\ast ( X , Y)  \to  \Ext^\ast (\phi_\A (X), \phi_A(Y)) 
\eq
which are compatible with the product structure on both sides:
\bq
 \phi_\A^{X,Z} \left( b \cdot a  \right) =  \phi_\A^{Y,Z} (b) \cdot \phi_\A^{X,Y} (a)
\eq
where the product on the left-hand side is in $A_\F$, and the right-hand side product is the Yoneda product:
\bq
\Ext^\ast (V_2,V_3) \otimes \Ext^\ast (V_1,V_2 ) \to \Ext^\ast (V_1,V_3).
\eq
\end{proposition}
The $\phi_\A^{X,Y}$ will be described in Definition \ref{def:morphism_identifications}. 
To prove Proposition \ref{prop:isomorphic_sequences} we proceed in several steps. We first compute the $\Ext$ groups between different elements of the exceptional sequence for $D^b \Coh (Y_{p,q,r})$, and fix identifications with our model for the Floer cochain groups. Notice that for any $j,k$ with $1 \leq j \leq p$, $1 \leq k \leq q$, we have
\bq
\Ext^\ast (\D_{\cP, j}, \D_{\cQ,k}) = \Ext^\ast ( \D_{\cQ,k}, \D_{\cP, j})  = 0
\eq
as the two coherent sheaves have disjoint closed supports, and similarly for pairs $(\D_{\cP, j}, \D_{\cR, k})$ and $(\D_{\cQ, j}, \D_{\cR,k})$. The non-trivial $\Ext$ groups between different elements of the exceptional sequence are calculated in Lemmas \ref{lem:ext_DD}, \ref{lem:ext_DE} and \ref{lem:ext_EE} below.  We then compare the two product structures (Lemma \ref{lem:products_agree}).

\begin{lemma} 
\label{lem:ext_DD}
Suppose that $1 \leq j < k \leq p$. Then
\bq
\Ext_{Y_{p,q,r}}^\ast (\D_{\cP, j}, \D_{\cP, k} ) = 
\begin{cases}
\C & \text{if  } \ast = 0 \text{  or  } 1 \\
0 & \text{otherwise}
\end{cases}
\eq
Similarly for $\cQ$ and $\cR$. 
\end{lemma}

\begin{proof}We'll use the notation $E_s := E_{\cP, s}$. Fix a divisor $F$ which intersects $\widetilde{E}_i$ transversally in one point, and does not intersect any of $\widetilde{E}_{i-1}, \ldots, \widetilde{E}_{2}, E_1$: if $i<p$, pick $F=\widetilde{E}_{i+1}$, and if $i=p$, $\widetilde{H}$. We use the resolutions:
\begin{eqnarray}\label{eq:resolution}
\D_{\cP, i} & \cong & \{ \O(- F -\widetilde{E}_i  - \widetilde{E}_{i-1} 
- \ldots - \widetilde{E}_{2} - E_1 ) \to \O(-F)     \}.
\end{eqnarray} 

In particular, dually, the push-pull yields:
\begin{align}
(\D_{\cP, 1} )^{\mathbf{L}\vee} & \cong \{   \O (\widetilde{E}_{2} ) \to  \O (E_1 + \widetilde{E}_{2} ) \}
\\
& \cong \O(E_1 + \widetilde{E}_{2}) \otimes \{\O(-E_1) \to \O \} \\
& \cong (i_{E_1})_\ast (i_{E_1})^\ast \O(E_1 + \widetilde{E}_{2})  \\
& \cong (i_{E_1})_\ast \O_{E_1}
\end{align}

To prove the claim, we may  assume without loss of generality that $j=1$, by  Lemma \ref{lem:ext_pullbacks}. 
Thus
\begin{align}
& \Ext^\ast (\D_{\cP, 1}, \D_{\cP, k}) 
 \\
& \cong 
 \H^\ast \left(Y_{p,q,r},  
(i_{E_1})_\ast \O_{E_1} \otimes
 \{ \O(- F -\widetilde{E}_k  - \widetilde{E}_{k-1} 
- \ldots - \widetilde{E}_{2} - E_1 ) \to \O(-F)     \} \right)
 \\
& \cong \H^\ast \left( Y_{p,q,r}, (i_{E_1})_\ast \O_{E_1}
\oplus
 (i_{E_1})_\ast \O_{E_1}[-1]
\right)  \\
& \cong H^\ast (E_1, \O_{E_1})\oplus H^\ast (E_1, \O_{E_1})[-1]
\end{align}
where $\H$ denotes hypercohomolgy, $[k]$ denotes a shift of degrees \emph{down} by $k \in \Z$, and we use push-pull to get from the second to the third line. This completes the proof.
\end{proof}
We label the generator corresponding to the constant section $1 \in H^0 (E_1, \O_{E_1} ) $ by $\e^{\cP}_{j,k}$, and the generator corresponding to $1[-1] \in H^0 (E_1, \O_{E_1})[-1] $ by $\x^{\cP}_{j,k}$. Similarly for $\cQ$ and $\cR$.

\begin{lemma}
\label{lem:ext_DE}
 Let $s=0,1$ or 2. For any $j$ with $1 \leq j \leq p$, we have
\bq
\Ext_{Y_{p,q,r}}^\ast (\D_{\cP, j} , \pi^\ast \O_{\P^2}(s)) = 
\begin{cases}
\C & \text{if  } \ast = 0\\
0 & \text{otherwise}
\end{cases}
\eq
Similarly for the exceptional sheaves associated to $E_{\cQ,j}$ and $E_{\cR,j}$.
\end{lemma}

\begin{proof} Again, without loss of generality $j=1$. 
Pick a line $L$ in $\P^2$ that is disjoint from the three points $[1:-1:0]$, $[0:1:-1]$, $[-1:0:1]$.
We have that
\bq
(\D_{\cP, 1} )^{\mathbf{L}\vee} \otimes \big( \pi^\ast \O_{\P^2} (sL) \big) \cong
\big( (i_{E_1})_\ast \O_{E_1} \big) \otimes \big( \pi^\ast \O_{\P^2} (sL) \big) \cong (i_{E_p})_\ast \O_{E_p}.
\eq 
from which one concludes the following:
\bq 
\Ext_{Y_{p,q,r}}^\ast (\D_{\cP, j} , \pi^\ast \O_{\P^2}(s)) \cong H^\ast (E_p, \O_{E_p}).
\eq
\end{proof}
We label the generator corresponding to $1 \in H^0 (E_p, \O_{E_p})$ by $\y_{j,s}^{\cP}$. Similarly for $\cQ$ and $\cR$.
Finally, using Lemma \ref{lem:ext_pullbacks} to reduce  the ambient space to $\P^2$, it is immediate to calculate the $\Ext$ groups between the $\pi^\ast \O_{\P^2} (s)$. Let us fix the notation.
\begin{lemma}\label{lem:ext_EE}
The $\Ext$ groups between the $\pi^\ast \O_{\P^2} (s)$ are as follows.
\begin{align}
\Ext^\ast  (\pi^\ast \O_{\P^2} , \pi^\ast \O_{\P^2} (1) ) 
 &\cong \C \langle \aa_1, \aa_2, \aa_3 \rangle \\
\Ext^\ast  (\pi^\ast \O_{\P^2} (1), \pi^\ast \O_{\P^2} (2) ) 
&\cong  \C \langle \cc_1, \cc_2, \cc_3 \rangle  \\
\Ext^\ast  (\pi^\ast \O_{\P^2}, \pi^\ast \O_{\P^2} (2) )  
&\cong  \C \langle \bb_1, \bb_{1,2}, \bb_2, \bb_{2,3}, \bb_3, \bb_{3,1} \rangle
\end{align}
with all generators are in degree 0. Given standard homogeneous coordinates $u,v$ and $w$ on $\P^2$, our convention is that
\begin{enumerate}
\item $\aa_1, \aa_2, \aa_3$ correspond, respectively, to $u, v, w$;
\item  $\bb_1, \bb_{1,2}, \bb_2, \bb_{2,3}, \bb_3, \bb_{3,1}$ correspond, respectively, to $u^2, uv, v^2, vw, w^2, wu$;
\item and $\cc_1, \cc_2, \cc_3$ correspond, respectively, to $u, v, w$.
\end{enumerate}
\end{lemma}

\begin{definition}\label{def:morphism_identifications}
We define the identifications of morphism spaces of Proposition \ref{prop:isomorphic_sequences} as follows. For better legibility, we suppress the superscripts on the $\phi_\A$. 
\begin{eqnarray*}
\phi_\A (e^{\cP}_{i,j} ) = \e^{\cP}_{i,j}  \quad \phi_\A (x^{\cP}_{i,j} ) = \x^{\cP}_{i,j}  &\textrm{for all \,} \, 1 \leq i < j \leq p\\
\phi_\A (y^{\cP}_{i,k} ) = \y^{\cP}_{i,k-1}   & \textrm{for all \,} 1 \leq i \leq p \text{\, and \,} k \in \{ 1,2,3 \}\\
\phi_\A (a_k) = \aa_k; \, \, \, \phi_\A (b_k ) = \bb_k;  \, \, \,   \phi_\A (b_{k,l} ) = \bb_{k,l}; \, \, \,  \phi_\A (c_k) = \cc_k &
\textrm{for all \,} \{k , l \} \in \{ 1,2,3 \}.
\end{eqnarray*}
Similarly for $\cQ$ and $\cR$. 
Note that the index shift on the second line is simply reflecting our choice of $\O_{\P^2} $, $\O_{\P^2} (1)$,  $\O_{\P^2}(2)$ as the full exceptional sequence for $D^b \Coh(\P^2)$. 
\end{definition}

\begin{lemma}\label{lem:products_agree}
Under the identifications $\phi_\A$, the products between  objects in $A_C$ precisely agree with the products in $A_\F$ between the distinguished collection of vanishing cycles of equation \eqref{eq:vanishing_cycles}.
\end{lemma}

\begin{proof} For those products involving the `unit'--type element of $Ext^0 (\D_{\cP, i}, \D_{\cP, j} )$, corresponding to $1 \in H^0 (E, \O_E)$ for some exceptional divisor $E$, this readily follows from the isomorphisms of Lemmas \ref{lem:ext_DD} and \ref{lem:ext_DE}. Similarly for the $\cQ$ and $\cR$. The multiplication
\bq
\Ext^\ast (\pi^\ast \O_{\P^2} (1) , \pi^\ast \O_{\P^2} (2)) \otimes \Ext^\ast (\pi^\ast \O_{\P^2}  , \pi^\ast \O_{\P^2} (1) ) \to
\Ext^\ast (\pi^\ast \O_{\P^2}, \pi^\ast \O_{\P^2} (2)) 
\eq
is given, with respect to our choices of bases of sections, by the standard multiplication of the coordinates $u, v$ and $w$. Transcribing back to the notation $\aa_i, \bb_i,$ etc., precisely gives the same products as those between the Floer complexes of the $\E_i$. See the quiver with relations in \eqref{eq:P^2quiver}. (To continue the parallel with \cite[Section 6]{Seidel_subalgebras} started in Remark \ref{rmk:P^2model}, one could consider instead the full exceptional sequence $(\Omega^2_{\P^2} (2), \Omega^1_{\P^2} (1), \O_{\P^2})$. )

The principal cases to check are the products
\bq
\Ext^\ast (\pi^\ast \O_{\P^2} (j) , \pi^\ast \O_{\P^2} (k) )  \otimes \Ext^\ast (\D_{\cP, i} ,  \pi^\ast \O_{\P^2} (j) )  \to \Ext^\ast (\D_{\cP, i} ,  \pi^\ast \O_{\P^2} (k) ) 
\eq
for $( j, k ) \in \{ (0,1), (1,2), (0,2) \}$. 
Recall  (Definition \ref{def:exceptional_sheaves}) that the $\D_{\cP,i}$ are supported on the proper transforms of divisors $E_{\cP, i}$, obtained from repeatedly blowing up the point $[1:-1:0]$. This lies on the $w=0$ component of the toric divisor on $\P^2$.
 In particular, the product of $\y^{\cP}_{i,j}$, the sole generator of $\Ext^\ast(\D_{\cP, i} , \pi^\ast \O_{\P^2}(j)) $, with an element $\mathbf{z} \in \Ext^\ast (\pi^\ast \O_{\P^2} (j) , \pi^\ast \O_{\P^2} (k) )$, will be zero if $\mathbf{z}$ corresponds to a section which is a multiple of $w$ (i.e.~$\aa_3, \cc_3, \bb_{1,3}, \bb_{2,3}, \bb_3$ -- see  Lemma \ref{lem:ext_EE} for the notation).
 When that section is not a multiple of $w$, the product is determined by the fact that $\y^{\cP}_{i,j}$ was naturally identified with $1 \in H^0 (E, \O_E)$, for some exceptional divisor $E$. 
Transcribing back to the notation $\aa_i, \bb_i,$ etc., this again precisely gives the same products as those between the Floer groups. 

Finally, further products can be ruled out either by noting that the relevant sheaves would have disjoint support, or by considering degrees.
\end{proof}

This completes the proof of Proposition \ref{prop:isomorphic_sequences}.

\subsection{A dg enhancement of  $D^b \Coh(Y_{p,q,r}) $}We will later want to use a dg enhancement of $D^b \Coh(Y_{p,q,r})$. One possible representative is the category $\twvect(Y_{p,q,r})$, defined as follows.

Given a projective variety $X$, following \cite{Seidel_quartic, Lekili-Perutz}, let $\text{vect} (X)$ denote the dg category with  objects the  locally free coherent sheaves on $X$, and morphisms given by \v{C}ech cochain complexes with values in $\mathcal{H}om$ sheaves, for some fixed finite affine open cover $\mathfrak{U}$ of $X$:
\bq
hom (E,F) = ( \check{C}^\ast (\mathfrak{U}; \mathcal{H}om(E,F) ), \partial)
\eq
where $\partial$ is the \v{C}ech differential, and composition combines the composition of sheaf morphisms with the shuffle product. Let $\twvect(X)$ denote the (pre-triangulated) dg category of twisted complexes in $\text{vect}(X)$. There is an equivalence of categories $\epsilon: H^0 (\twvect(X)) \to \Perf(X)$. Different choices of affine covers give quasi-isomorphic dg categories. Moreover, by work of Lunts and Orlov, for any other dg enhancement $\mathcal{C}$ of $\Perf(X)$, with $\epsilon': H^0(\mathcal{C}) \to \Perf(X)$, there is an $\Aoo$--quasi-isomorphism $\phi: \mathcal{C} \to \twvect(X)$ such that $\epsilon \circ [\phi] = \epsilon'$ \cite[Theorem 2.14]{Lunts-Orlov}. Of course when $X$ is also smooth, the is a natural equivalence $\Perf(X) = D^b \Coh(X)$.

\begin{definition}
Let $\A_\tC$ be the full subcategory of $\twvect(Y_{p,q,r})$ with the $p+q+r+3$ objects given by (resolutions of) the elements of the full exceptional sequence for $Y_{p,q,r}$. We use  the twisted complexes of \eqref{eq:resolution} to represent the $\D_{\cP, i}$,and similarly for $\cQ$ and $\cR$.
\end{definition}

%%%%%%%%%%%%%%%%%%%%%%%%%%%%%%%%%%%%%%%
%%%%%%%%%%%%%%%%%%%%%%%%%%%%%%%%%%%%%%%

\section{Auxiliary isomorphisms: fibre of $\Xi$ and anticanonical divisior on $Y_{p,q,r}$}\label{sec:iso2}

\subsection{The Fukaya category of the fibre of $\Xi$}

We are interested in the Fukaya category $\Fuk(M)$, defined as in \cite[Section 12]{Seidel_book}. Its objects are Lagrangian branes given by closed, exact Lagrangian submanifolds decorated with choices of grading and a spin structure.

Let $\cP$, $\cQ$, $\cR$ and $\E$ be the exact Lagrangian $S^1$s in $M$ given in Figure \ref{fig:D_4}. 
(These agree with the vanishing cycles $\cP_1$, $\cQ_1$, $\cR_1$ and $\E_1$, respectively, for the Lefschetz fibration $\Xi$.) 
We equip these with the same choices of gradings and non-trivial spin structures as for $\cP_1$, $\cQ_1$, $\cR_1$ and $\E_1$ in Section \ref{sec:fibration_branes}. 
Continuing our abuse of notation, we also denote the resulting Lagrangian branes by  $\cP$, etc., rather than $\cP^\#$, etc.

\begin{lemma} We have that
 \bq
\text{tw\,} \Fuk(M) \cong \langle   \cP, \cQ, \cR , \E    \rangle.
\eq
This means that $\text{tw\,} \Fuk(M)$ is the smallest full $\Aoo$--subcategory of itself which contains these objects and is closed under quasi-isomorphisms, shifts, cones and passing to idempotents.
\end{lemma}

\begin{proof}First, observe that as a Liouville domain, $M$ can be regarded as the Milnor fibre of $x^3 + y^3$, the two-variable $D_4$ singularity. The Lagrangians $\cP$, $\cQ$, $\cR$ and $\E$ form an ordered distinguished collection of vanishing cycles for $D_4$; this classical configuration can for instance be recovered using A'Campo's techniques \cite{ACampo} (see \cite[Section II]{Keating14} for a symplectic account). The two-variable $D_4$ singularity is weighted homogeneous with weights $(3,3)$. As  $1/3 + 1/3 \neq 1$, results of Seidel, specifically \cite[Proposition 18.17]{Seidel_book} combined with \cite[Lemma 4.16]{Seidel_gradedlagrangians}, then imply the claim.
\end{proof}

\subsection{Perfect complexes on the anticanonical divisor of $Y_{p,q,r}$} \label{sec:PerfD}

The variety $Y_{p,q,r}$ is obtained through an iterated sequence of blow-ups, starting with $\P^2$; $D$ denotes the strict transform of the toric anticanonical divisor on $\P^2$, which is itself an anticanonical cycle of three $\P^1$s. Call $D_1$ the strict transform of the hyperplane $w =0$ in $\P^2$; $D_2$, the strict transform of $u=0$; and $D_3$, the strict transform of $v=0$. Let $s_i$ be any smooth point of $D$ on $D_i$, $i=1,2,3$.

\begin{lemma}Following \cite[Lemma 3.1]{Lekili-Perutz}, we have
\bq \label{eq:Perf_generators}
\twvect(D) = \langle   \O, \{ \O(-s_1) \to \O \} , \{ \O(-s_2) \to \O \}, \{ \O(-s_3) \to \O \} \rangle
\eq
where $\langle \cdot \rangle$ means the same as before, and $\O = \O_D$.
\end{lemma}

\begin{proof}
Notice that $\{ \O(-s_i) \to \O \}$ is a resolution of the skyscraper sheaf at $s_i$.
The argument is analogous to the proof of \cite[Lemma 3.1]{Lekili-Perutz}, which shows that  for the nodal cubic $E$, $\twvect(E) =  \langle \O_E,  \{ \O_E (-s) \to \O_\E \} \rangle$, for a smooth point $s \in E$.  
Alternatively,
there exists a classification of non-decomposable vector bundles on cycles of $\P^1$'s (see \cite[Theorem 2.12]{Drozd-Greuel}, and the account in \cite[Theorem 2.1]{Burban-Drozd}). 
One could proceed directly from this and show that each of those is a summand of a twisted complex in the elements on the right-hand side of \eqref{eq:Perf_generators}.
\end{proof}

Let us add a remark concerning the second strategy. To deal with line bundles, first recall that a line bundle on a cycle of three $\P^1$s is determined by its multidegree $(k,l,m) \in \Z^3$ and a parameter in $\C^\ast$.
The multidegree of a bundle is given by the degrees of its restrictions to each of the components of the cycle. The extra parameter comes from the fact that we are considering a cycle (rather than a chain) of $\P^1$s.
 One can build such line bundles as cones on the aforementioned objects by hand, proceeding iteratively. For instance, the line bundle of multi-degree $(1,1,1)$ and parameter $1 \in \C^\ast$ is given by 
\bq
\xymatrix{
\{  \{ \O(-s_1) \to \O \} \oplus \{ \O(-s_2) \to \O \} \oplus \{ \O(-s_3) \to \O \}  \ar[rr]^-{1+1+1} &&  \O  \}
}
\eq
where $1$ simply denotes the section with constant value $1$ of $\mathcal{H}om(\O,\O)$.

\subsection{Quasi-isomorphism between $D^b \Fuk(M)$ and $\Perf(D)$}

We will show that $D^\pi \Fuk(M)$ is quasi-isomorphic to $\Perf(D)$. This again follows from work of Lekili and Perutz \cite[Theorem A]{Lekili-Perutz}, together with a covering argument.

Let $T_0$ be the Milnor fibre of the two-variable $A_2$ singularity, i.e.~the once punctured torus with its ``standard'' Liouville form, and let $\mathcal{G}$ and $\mathcal{H}$ be two exact embedded Lagrangian $S^1$s with $\mathcal{G} \cdot \mathcal{H} =1$. Again, we will also denote by $\mathcal{G}$ and $\mathcal{H}$ the two associated Lagrangian branes given by equipping the Lagrangians with the non-trivial spin structures, and any gradings such that the intersection point $y \in \mathcal{G}\pitchfork \mathcal{H}$ has degree zero.

\begin{theorem}\cite[Theorem A]{Lekili-Perutz} \label{thm:Lekili-Perutz}
There is an $\Aoo$--functor $\bar{\phi}: \Fuk(T_0) \to \twvect(E)$ which induces an equivalence of derived categories
\bq
[\bar{\phi}]: D^{\pi} \Fuk(T_0) \cong \Perf (E).
\eq
Here $\Fuk(T_0)$ is the Fukaya category of $T_0$ (defined again as in \cite[Section 12]{Seidel_book}) and $E$ is a nodal elliptic curve. The functor $\bar{\phi}$ sends $\mathcal{G}$ to $\O_E$, and $\mathcal{H}$ to $i_\ast \O_{s}$, where as before $s$ is a smooth point.
(Formally, we are using the complex $\{ \O(-s) \to \O \}$ for $i_\ast \O_s$.)
Moreover, 
$\bar{\phi}$ is an  $\Aoo$--quasi-isomorphism from the full subcategory on $\{ \mathcal{G},  \mathcal{H} \}$, the the full subcategory $\{ \O_E, i_\ast \O_s \}$, which split-generate $\text{tw\,}\Fuk(T_0)$, respectively $\twvect (E)$. 
\end{theorem}

\begin{proof}
The bulk of the claim is stated explicitly in  points (i) and (iii) of \cite[Theorem A]{Lekili-Perutz}. While the results in the introduction of \cite{Lekili-Perutz} are stated over $Spec(\Z)$, notice that their proofs are set up for coefficients in a class of rings which includes the complex numbers $\C$.
For the second sentence, see Lemma 3.1 and Section 6.1 of the same article, together e.g.~with the ``Outline of method'' summary of page 8. 
\end{proof}

There is a 3:1 unbranched cover $\sigma: M \to T_0$. We pick $\mathcal{G}$ and $\mathcal{H}$ such that $\sigma$ maps $\cP$, $\cQ$ and $\cR$ to $\mathcal{H}$, and $\E$ to (a triple copy of) $\mathcal{G}$. See Figure \ref{fig:D_4}. 
Moreover, because there was a $\Z/3$ symmetry in our original choices for $\cP_1$, $\cQ_1$ and $\cR_1$, we can assume that the gradings and marked points (recording spin structures) are compatible with the covering.

\begin{figure}[htb]
\begin{center}
\includegraphics[scale=0.26]{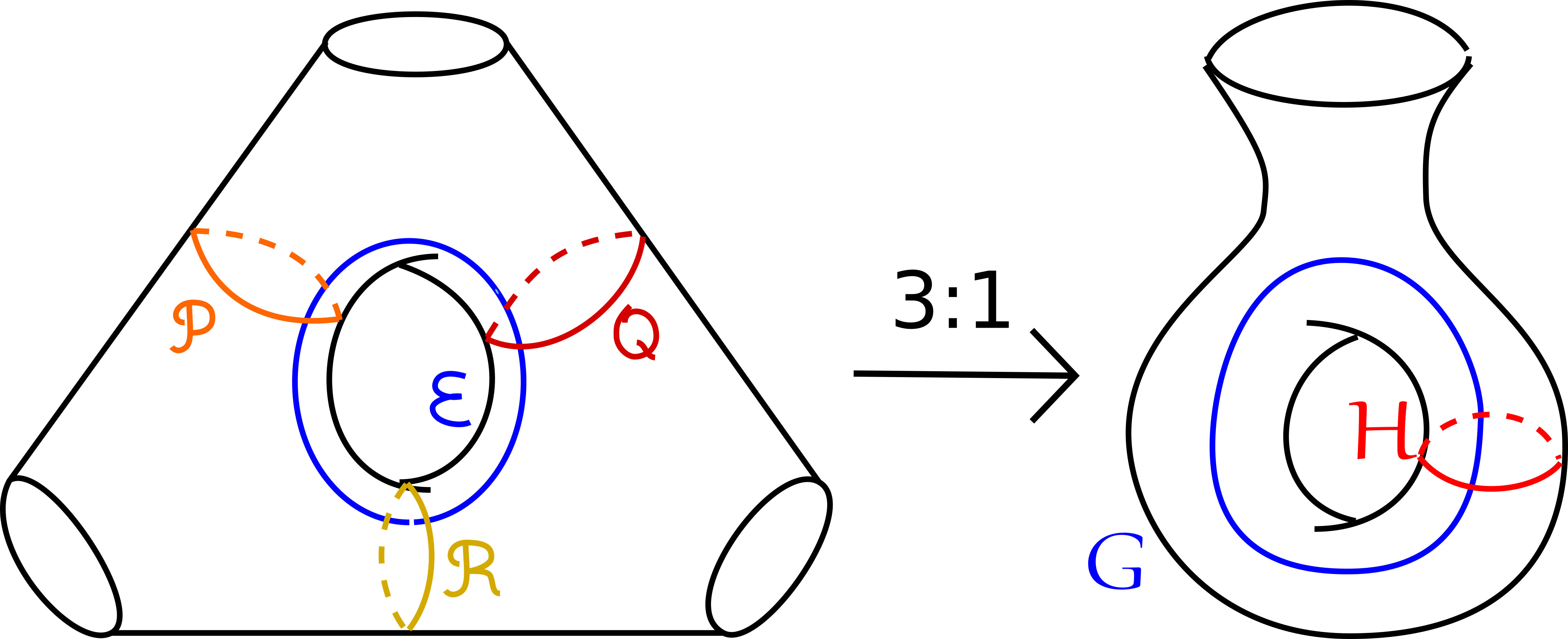}
% or [scale=0.85]
\caption{Unbranched 3:1 cover from the $D_4$ Milnor fibre, $M$, to the $A_2$ Milnor fibre.}
\label{fig:D_4}
\end{center}
\end{figure}

On the other hand, there is 3:1 unbranched cover $\rho: D \to E$, and we can pick the smooth points $s_i \in D_i$ such that all three map to $s$. Moreover, a choice of finite affine open cover for $E$ lifts to give one for $D$.

\begin{proposition}\label{pro:iso2}
There is an $\Aoo$--functor $\phi_\B: \Fuk(M) \to \twvect(D)$ such that 
\begin{eqnarray}
\phi_{\B} (\cP) = \{ \O(-s_1) \to \O \} &
 \phi_{\B} (\cQ) = \{ \O(-s_2) \to \O \} &
 \phi_{\B} (\cR)  = \{ \O(-s_3) \to \O \} \\
& \phi_{\B} (\E) = \O_D &
\end{eqnarray}
and it induces an equivalence of derived categories:
\bq
\phi_{\B}: D^\pi \Fuk(M) \to \Perf(D).
\eq
\end{proposition}

\begin{proof}
The $\Z/3$ action on $M$ induces a $\Z/3$ action on the objects of $\Fuk(M)$, which needn't respect morphisms or the $\Aoo$-- structure. 
Let $\bar{B}_\F$ be the cohomology category associated to the full subcategory of $\Fuk(T_0)$ on  $\{ \mathcal{G}, \mathcal{H} \}$, and let $B_\F$ be the cohomology category associated to the full subcategory of $\Fuk(M)$ on $\{ \cP, \cQ, \cR, \E \}$.
Equip $\bar{B}_\F$ with any $\Aoo$--structure which makes it a strictly unital quasi-isomorphic representative of the full subcategory of $\Fuk(T_0)$ on $\{ \mathcal{G}, \mathcal{H} \}$, say $\bar{\mathcal{B}}_\F$. (This is automatically minimal.)
This lifts to an $\Aoo$--structure on $B_\F$; the result, say $\mathcal{B}_\F$, is itself a quasi-isomorphic representative of the full subcategory on $\{ \cP, \cQ, \cR, \E \}$ of the Fukaya category of $M$: consistent universal choices of perturbation data for $\{ \mathcal{G}, \mathcal{H} \}$ lift to consistent universal choices for $\{ \cP, \cQ, \cR, \E \}$, as do $\Aoo$--functors giving quasi-isomorphisms between two different choices. By construction, the resulting choice of $\Aoo$--representative for $\Fuk(M) \subset \text{tw}^\pi (\B_\F)$ is $\Z/3$ equivariant. 

The $\Z/3$ action on $D$ induces one on $\twvect(D)$, with quotient $\twvect(E)$. Now pick a functor $\bar{\phi}: \text{tw}^\pi \bar{\mathcal{B}}_\F \to \twvect(D)$, with $ \bar{\phi}_{\B}(\mathcal{H}) = i_\ast (\O_s)$, and $ \bar{\phi}_{\B} (\mathcal{G}) = \O_E$, as in Theorem \ref{thm:Lekili-Perutz}. This restricts to an $\Aoo$--functor 
$ \bar{\phi}_{\B}: \bar{\mathcal{B}}_\F \to \twvect(D)$, which in turn lifts to an $\Aoo$--functor $\phi_{\B}: \mathcal{B}_\F \to \twvect(E)$, with the choice of lift determined by requiring $\phi_{\B} (\cP) =  \{ \O(-s_1) \to \O \}$, $\phi_{\B} (\cQ) =  \{ \O(-s_2) \to \O \}$ and $\phi_{\B} (\cR) =  \{ \O(-s_3) \to \O \}$. Taking split-closure and passing to cohomology gives an $\Aoo$--functor $D^\pi \B_\F \to \Perf(D)$, which, by construction, is an equivalence.
\end{proof}

\begin{remark}[Version 2 of the article]
An elegant proof of a more general theorem, for cycles of $\P^1$s of arbitrary length, is now available in a recent preprint 
Lekili and Polishchuk \cite[Theorem A]{Lekili-Polishchuk} . (This only appeared after the first arXiv version of the present article.)
\end{remark}

We record the following.
\begin{definition}
$\B_\F$ denotes a strictly unital, minimal model for the full subcategory of $\Fuk(M)$ on the objects $\{ \cP, \cQ, \cR, \E \}$, as above.  Moreover, we define $\B_\tC$ to be the full subcategory of $\twvect(D)$ on $ \{ \O(-s_1) \to \O \},  \{ \O(-s_2) \to \O \},  \{ \O(-s_3) \to \O \}$ and $\O_D$. 
\end{definition}

By Proposition \ref{pro:iso2}, $\phi_\B$ restricts to a quasi-isomorphism $\B_\F \to \B_\tC$. 
Let $\psi_{\B}$ denote a quasi-isomorphism $\B_\tC \to \B_\F$ which is a homotopy inverse to $\phi_\B$. By construction, on the level of object and morphisms, this is given by passing to cohomology; it transfers the dg structure on $\B_\tC$ to a minimal $\Aoo$ structure on its cohomology.

\begin{remark}
We expect there to be an equivalence of triangulated categories
\bq
D^b \mathcal{W} (M) \to D^b \Coh (D)
\eq
where $\mathcal{W}(M)$ is the wrapped Fukaya category of the three-punctured elliptic curve $M$ (note that this is already idempotent-closed), following \cite[Theorem A(iv)]{Lekili-Perutz}. 
\end{remark}

%%%%%%%%%%%%%%%%%%%%%%%%%%%%%%
%%%%%%%%%%%%%%%%%%%%%%%%%%%%%%
\section{Restriction functors and localization}\label{sec:restrictions}

\subsection{The restriction $c_{\Fuk}: \A_\F \to \B_\F$}\label{sec:Fuk_restriction}

A Lagrangian brane for $\A_\F$ is also a Lagrangian brane for $\B_\F$. This induces an $\A_\infty$--functor  $c_{\Fuk}: \A_\F \to \B_\F $. 
Explicitly, on objects, we have:
\begin{itemize}
\item $c_{\Fuk} (\cP_i) = \cP$, $c_{\Fuk} (\cQ_j) = \cQ$, and $c_{\Fuk} (\cR_k) = \cR$, for all possible $i,j,k$;
\item $c_{\Fuk} (\E_1) = \E$, $c_{\Fuk} (\E_2) = \E'$, and $c_{\Fuk} (\E_3) = \E''$. 
\end{itemize}
where $\E'$ is the object in $\Fuk(M)$ which is the result of performing a positive Dehn twist in $\cP$, in $\cQ$ and in $\cR$ to $\E$ (with their Lagangian brane data), and  $\E''$ is the Lagrangian brane which is a result of performing a further positive Dehn twist in $\cP$, in $\cQ$ and in $\cR$ to $\E'$.

The action on morphisms is given by viewing the intersection points between Lagrangian branes in  $\textrm{Ob}(\A_\F)$ as intersection points between Lagrangian branes in $\textrm{Ob}(\B_\F)$. For any brane $V$, the unit element in $hom_\A(V,V)$, ``artificially'' introduced in Section \ref{sec:directed_Fukaya}, is mapped to the unit in $hom_\B (V,V)$, say $e^V$. Moreover, for any $1 \leq i < j \leq p$, the degree zero element $e^{\cP}_{i,j} \in hom(\cP_i, \cP_j)$ is mapped to $e^{\cP}$, and the degree one element $x^{\cP}_{i,j}$ is mapped to the standard generator of $hom^1_\B (\cP, \cP) \cong \C$, say $x^{\cP}$. Similarly for $\cQ$ and $\cR$. Altogether, this data naturally defines a  $\Aoo$--functor
\bq
c_{\Fuk}: \A_\F \to \B_\F
\eq
with maps $c^1_{\Fuk}: hom_{\A} (X_0, X_1) \to hom_{\B}(c_{\Fuk} X_0, c_{\Fuk} X_1)$, and where all of the higher-order maps on tensors of morphism spaces ($c^i_{\Fuk}$ for $i > 1$) vanish.

`Artificially' add $p-1$, $q-1$ and $r-1$ copies of, respectively, $\cP$, $\cQ$ and $\cR$ to $\B_{\F}$ to get a category with $p+q+r+3$ objects, say $\B^+_{\F}$, which is quasi-isomorphic to $\B_{\F}$. We can think of $\B^+_{\F}$ as a $\Aoo$--algebra over the semi-simple ring $R$, and of $\A_{\F}$ as a subalgebra of $\B_{\F}$. Now $c_{\Fuk}: \A_{\F} \to \B^+_{\F}$ is simply an inclusion of categories. $\B^+_{\F}$ inherits the structure of $\Aoo$--bimodule over $\A_{\F}$ by restricting the diagonal bimodule on $\B^+_{\F}$.

Of course, $c_{\Fuk}$ is not a full inclusion. However, whenever a morphism groups in $\A_\F$, say $hom_{\A_\F}(X_i, X_{i+1})$, is non-trivial, then $c_{\Fuk}:hom_{\A_\F}(X_i, X_{i+1}) \to hom_{\B^+_\F}(c_{\Fuk} (X_i), c_{\Fuk} (X_{i+1}))$ is an isomorphism of vector spaces, and the $\Aoo$--operations between such morphism spaces agree.

\subsection{The pull-back $\twvect\left(Y_{pq,r} \right) \to \twvect(D)$ }

\subsubsection{First order}

Pick a finite affine cover of $Y_{p,q,r}$; this induces one on $D$ by intersection. Given these choices, 
 the inclusion $\iota: D \to Y_{p,q,r}$ induces a dg functor
\bq
c_{\vect}: \twvect(Y_{p,q,r}) \to \twvect(D)
\eq
which one might think of as a dg enhancement of the usual pull-back map $$\iota^\ast: D^b \Coh(Y_{p,q,r}) \to \Perf(D).$$ 
From  \eqref{eq:resolution}, one readily reads off that
\begin{eqnarray}
c_\ve (\D_{\cP, i})  = \{ \O_D(-s_1) \to \O_D \}  \\
c_\ve (\D_{\cQ, j})  = \{ \O_D(-s_2) \to \O_D \}  \\
c_\ve (\D_{\cR, k})  = \{ \O_D(-s_3) \to \O_D \}
\end{eqnarray}
for all possible $i,j$ and $k$. 
We have that
\begin{multline} \pi^\ast (\O_{\P^2}(1)) = \pi^\ast (\O_{\P^2}(H)) =  \\
\O_{Y_{p,q,r}} \left( E^{\cP}_1 + \widetilde{E}^{\cP}_2+ \ldots + \widetilde{E}^{\cP}_p +
E^{\cQ}_1 + \widetilde{E}^{\cQ}_2+ \ldots + \widetilde{E}^{\cQ}_q +
E^{\cR}_1 + \widetilde{E}^{\cR}_2+ \ldots + \widetilde{E}^{\cR}_r +
\widetilde{H} \right).
\end{multline}
This implies that the coherent sheaf $c_\ve \pi^\ast ( \O_{\P^2}(1) )$ is a line bundle of multidegree $(1,1,1)$; by symmetry considerations, it correspond to the parameter $1 \in \C$.
  As discussed at the end of Section \ref{sec:PerfD}, there is an exact triangle
\bq
\xymatrix{ 
 i_\ast \O_{s_1} \oplus i_\ast \O_{s_2} \oplus i_\ast \O_{s_3} \ar[rr]^-{(1\oplus 1 \oplus 1)[1]} & & \O_D  \ar[dl]  \\  & \iota^\ast \left( \pi^\ast ( \O_{\P^2}(1) ) \right) \ar[ul] &
}
\eq
where we are using the standard identification 
$\Ext^\ast (i_\ast \O_{s_j} , \O_D ) = \Ext^0 (i_\ast \O_{s_j}, \O_D ) = \C$, and write $i_\ast \O_{s_1}$ to represent the corresponding resolution, as above. 
In particular, we recognize that $\iota^\ast \pi^\ast ( \O_{\P^2}(1) )$ is the positive twist of $\O_D$ in 
$i_\ast \O_{s_1}$, $i_\ast \O_{s_2}$ and $i_\ast \O_{s_3}$. 
Similarly, $\iota^\ast \pi^\ast ( \O_{\P^2}(2) )$ is the result of three further positive twists on $\iota^\ast \pi^\ast ( \O_{\P^2}(1) )$, one in each of the $i_\ast \O_{s_j}$. 
This yields the following.

\begin{lemma}\label{thm:coho_agrees}
The maps $c_{\Fuk}$ and $c_\ve$ are compatible on the level of cohomology: the diagram
\bq
\xymatrix{
A_\F  \ar[r]^{H(c_{\Fuk})} \ar[d]_{\phi_\A} & B_\F \ar[d]^{H(\phi_\B)}  \\
A_\tC \ar[r]^{H(c_{\ve})}  & B_\tC  
}
\eq
commutes.
\end{lemma}

\subsubsection{A better dg enhancement of $\Perf(D)$} We will later want to understand, loosely, the structure of $\twvect(D)$ as a $\Aoo$--bimodule over $\twvect(Y_{p,q,r})$. This will be made easier by switching to the following dg enhancement for $\Perf(D)$, following the strategy used in \cite[p.~105]{Seidel_subalgebras} for the affine case.

Let $s$ be a section of $\O_{Y_{p,q,r}}(D)$, with $s^{-1}(0)= D$. Fix an affine open cover $\mathfrak{U}$ of $Y_{p,q,r}$. The key observation  in \cite[p.~105]{Seidel_subalgebras} is that given any $U \in \mathfrak{U}$, we may replace $\C[U\cap D]= \Gamma(U, \O_D)$ with 
\bq \left(\Gamma(U, \O_{Y_{p,q,r}})[\varepsilon], \partial \right) = \left( \C[U \cap Y_{p,q,r}] [\varepsilon], \partial \right), \quad |\varepsilon| = -1 \text{\, and \,} \partial \varepsilon = s.
\eq
Fix $E, F \in  \text{Ob}\left( \text{vect}(Y_{p,q,r}) \right)$. 
Now define
\bq
hom_s (E,F) = \left( \check{C}^{\ast}(\mathfrak{U}; \mathcal{H}om(E,F)[s] ), \partial    \right)
\eq
where  $\Gamma(U, \mathcal{H}om(E,F)[s] ) = \Gamma(U, \mathcal{H}om(E,F))[s]$, and the differential $\partial$ is obtained by combining the \v{C}ech differential with $\partial \varepsilon = s$. Call the resulting dg category $\text{vect}_s (Y_{p,q,r})$, and the associated category of twisted complexes $\twvect_s (Y_{p,q,r})$. Let $\B^+_\tC$ be the full subcategory on $p+q+r+3$ objects given by the resolutions of the objects in the full exceptional sequence for $Y_{p,q,r}$. By construction, there is a quasi-isomorphism of dg categories $\B^+_\tC \to \B_\tC$. Moreover, the `restriction' map $\A_\tC \to \B^+_\tC$, which we will still denote by $c_\ve$, is now simply an inclusion of categories. This gives $\B^+_\tC$ the structure of an $\Aoo$--bimodule over $\A_\tC$.

As before with $(\A_\F, \B_\F^+)$, the inclusion $c_\ve: \A_\tC \to \B^+_\tC$ is of course not a full inclusion. However, again, for those morphism groups in $\A_\tC$ which are non trivial, the inclusion $$c_\ve: hom_{\A_\tC}(X_i, X_{i+1}) \subset hom_{\B^+_{\tC}}(c_\ve(X_i), c_\ve(X_{i+1}))$$ is essentially surjective, and the dg operations $\mu^1$ and $\mu^2$ agree.

%%%%%%

\subsection{Formality of $\A_\text{C}$ and an equivalence of derived categories}
There is a quasi-equivalence of $\Aoo$--categories $\B^+_\tC \to \B_\F^+$, which as before we denote by $\psi_\B$. This is given by passing to the cohomology category and transferring the $\Aoo$--structure. By Lemma \ref{thm:coho_agrees}, on the level of objects and morphisms, the image of $\A_\tC$ is $\A_\F$. Moreover, $\psi_B$ actually restricts to a functor from $\A_\tC$ to $\A_\F$: the image of the restriction of the maps
\bq
\psi^d_{\B}: hom_{\B^+_\tC} (X_{d-1}, X_d) \otimes \ldots  hom_{\B^+_\tC} (X_{0}, X_1)
\to
 hom_{\B^+_\F} (\psi_\B(X_{0}), \psi_\B(X_d) )
\eq
to morphisms in $hom_{\A_\tC}(X_{i-1}, X_i)$ necessarily lies in $hom_{\A_\F} (\psi_\B(X_{0}), \psi_\B(X_d) )$, by the observation at the end of Section \ref{sec:Fuk_restriction}. Call this functor $\psi_\A$. It induces an isomorphism of cohomology categories, and so it is a quasi-isomorphism of $\Aoo$--categories.

\begin{corollary}\label{cor:iso1} \label{thm:iso1}
$\phi_\A$ induces an equivalence of categories
\bq
D^b \Fuk^{\to} (\Xi) \cong D^b \Coh(Y_{p,q,r} ).
\eq
\end{corollary}

\begin{proof}
Both are quasi-isomorphic to $H^0 \left(\perf(\A_\F) \right) = H^0 \left(\perf(\A_\tC)\right)$, where $\perf(\A_\F)$ is the dg category of perfect $\Aoo$--modules over $\A_\F$. (Note that here $\perf(\A_\F)$ is quasi-isomorphic to $\text{tw}( \A_\F)$, the enlargement of $\A_\F$ to the dg category of its twisted complexes, as discussed in \cite[Example 7.11]{Seidel_lectures}.)
\end{proof}

While we shall not make further use of the following corollary in the present work, we feel it is worth recording as a potential ingredient for extensions. 
\begin{corollary}
Up to quasi-isomorphism, the $\Aoo$--category $\A_\tC$ is formal.
\end{corollary}

The pairs $(\A_\F, \B^+_\F)$, $(\A_\tC, \B^+_\tC)$ give equivalent $\Aoo$--bimodule structures; we shall exploit this to prove the third equivalence announced in the introduction. 

%%%%%%%%%%%%%%%%%%%%%%%%%%%%%%%%%%%
%%%%%%%%%%%%%%%%%%%%%%%%%%%%%%%%%%%%

\section{Localization and the wrapped Fukaya category} \label{sec:localization}

\subsection{Localization: generalities}

We are going to use a variation on the localization construction which is the main object of \cite{Seidel_subalgebras}. 
Consider a pair $(\A, \B)$, where $\B$ as an $\Aoo$--algebra over $R$, and $\A$ as a subalgebra of $\B$. While the constructions in the main body of \cite{Seidel_subalgebras} are for $\Aoo$--algebras over some ground field $\mathbb{K}$, rather than the ring $R$,  the main results and their proofs also hold over $R$, as explained at the start of Section 6 therein.

We briefly summarize the construction, and refer the reader to \cite[Sections 3 and 4]{Seidel_subalgebras} for details. Start with the following naive short exact sequence of $\Aoo$--bimodules over $\A$:
\bq
0 \to \A \to \B \to \B / \A \to 0. 
\eq
Let $\delta: \B / \A  \to \A$ be the map of $\Aoo$--bimodules over $\A$ which is boundary homomorphism of the exact sequence; we will mostly consider its shift $\delta[-1]: (\B / \A) [-1] \to \A$, which has degree zero.

In the case at hand, with $(\A, \B) = (\A_\F, \B_\F)$, there is a quasi-isomorphism $\B / \A \to \A^\vee [-1]$ of $\Aoo$--bimodules over $\A$, where $\A^\vee$ denotes the dual diagonal bimodule, and the naive short exact sequence is
\bq
0 \to \A \to \B \to \A^\vee[-1] \to 0.
\eq
This is proven in \cite{Seidel_FLI}: see Equation 2.19 therein, together with Propositions 2.1 and 3.1. (The main object of \cite{Seidel_FLI} is to study the boundary homomorphism $\delta: \A^\vee[-1] \to \A$.)

Let $\mathcal{V} = \perf(\A)$ be the dg category of perfect modules over $\A$, and let $V = H^0 (\mathcal{V})$. (Here we differ slightly from the set-up in \cite[Section 4]{Seidel_subalgebras}, which works with the larger category $\text{mod}(\A)$.) Consider the convolution functor
\bq
\Phi_{\A^\vee[-2]}: \text{mod} (\A) \to \text{mod}(\A).
\eq
In our case, as $\A$ is a  proper $\Aoo$--algebra, this restricts to a dg  functor
\bq
\Phi_{\A^\vee[-2]}: \text{perf} (\A) \to \text{prop}(\A)
\eq
where $\text{prop}(\A)$, the dg category of proper $\Aoo$--modules over $\A$, is quasi-isomorphic to $\perf(\A)$. Ignoring the $-2$ shift, this is a cochain-level implementation of the Serre functor for perfect modules. (For background exposition, the reader may wish to start with \cite[Lecture 7]{Seidel_lectures}, notably Examples 7.8 and 7.11.)  Following \cite{Seidel_subalgebras}, we set 
\bq
F = H^0 \left(\Phi_{\A^\vee [-2]} \right): V \to V.
\eq
Convolution with the diagonal bimodule $\A$ gives a dg functor $ \Phi_{A}: \text{perf} (\A) \to \text{perf}(\A)$
which is quasi-isomorphic to the identity: $H^0 (\Phi_\A) \cong \text{Id}$. Now $\delta[-1]: \A^\vee[-2] \to \A$ induces a natural transformation 
\bq
\Phi_{\delta[-1]}: \Phi_{\A^\vee[-2]} \to \Phi_{\A}.
\eq
On the cohomology level, this gives a natural transformation 
\bq
T = \left[  \Phi_{\delta[-1]} \right]: F \to \text{Id}.
\eq
The heart of \cite{Seidel_subalgebras} is the construction of a dg category $W_{(\A, \B)}$, which is shown to be the localization of $V$ along $T$. 

\subsection{Localization for $\Fuk^{\to}(\Xi)$ and $\Fuk(M)$}

To identify $W_{(\A_\F, \B^+_\F)}$ geometrically, we appeal to a theorem in work-in-progress of Abouzaid and Seidel.

\begin{theorem}\cite{Abouzaid-Seidel} \label{thm:Abouzaid-Seidel}
Let $\pi: N \to \C$ be a Lefschetz fibration with total space  a Liouville domain $N$ (after smoothing corners). Pick an ordered distinguished collection of vanishing cycles for $\pi$; say there are $k$ of them. Let $\A$ be the associated directed Fukaya category, and $\B$ the full subcategory of a smooth fibre on the same objects as $\A$. (The notation is chosen to agree both with the present text and earlier work of Seidel relating to the same structures, such as \cite{Seidel_FLI}, which considers the structure of $\B$ as an $\Aoo$--bimodule over $\A$.) There is an equivalence
\bq W_{(\A, \B)}\cong D^b\, \mathcal{W} (N) \eq
where $\mathcal{W}(N)$ is the wrapped Fukaya category of $N$.  
\end{theorem}

\subsubsection*{Discussion}

The reader may wish to compare this statement to the conjecture in \cite[p.~110]{Seidel_subalgebras}. This predicts that the full subcategory of $\mathcal{W}(N)$ on the distinguished collection of thimbles corresponding to the vanishing cycles, say $\Delta_1, \ldots, \Delta_k$, is quasi-isomorphic to the full subcategory of $\mathscr{W}_{(\A, \B)}$ on the pull-backs of the modules $\A e_1, \ldots, \A e_k$, where $\mathscr{W}_{(\A, \B)}$ is a dg enhancement of $W_{(\A, \B)}$ (see \cite{Seidel_subalgebras} for a definition). The $e_i$ are idempotents in a semi-simple ring $R$ defined as before. Notice that the collections $\Delta_1, \ldots, \Delta_k$ and $\A e_1, \ldots, \A e_k$ split-generate $\mathcal{W}(N)$ and $\perf(\A)$, respectively. As such,  while \cite{Abouzaid-Seidel} has yet to appear publicly, we do not expect Theorem \ref{thm:Abouzaid-Seidel} to surprise experts. (Note that the thimbles don't readily give objects of $\mathcal{W}(N)$: they are not conical with respect to the Liouville form on $M$, but rather with respect to some 'purely horizontal' form. Some technical care needs to be taken to address this.)
We also refer the reader to the following existing literature: first, the quotient construction of Fukaya categories described in Lecture 10 of \cite{Seidel_lectures}. Second, a partial account of the results of \cite{Abouzaid-Seidel} can be found in the Appendix of the recent preprint of Abouzaid and Smith \cite{Abouzaid-Smith} -- see in particular Section A4 therein. 
To get a feel for a special case of this theorem, the reader might also wish to consider the discussion of wrapped Fukaya categories in  \cite{AAEKO}.

While we shall not use this in our proofs, the following, due to Kontsevich \cite[p.~30-31]{Kontsevich_ENS}, might help intuition. Recall that convolution with $\B / \A \cong \A^{\vee}[-1]$ induces a functor $F: H^0(\perf(\A)) \to H^0(\perf(\A))$, which, ignoring shifts, is the Serre functor. Geometrically, we have $H^0(\perf(\A)) \cong D^{b} \Fuk^{\to}(\pi)$, and $F$ corresponds to the `total monodromy' functor $\mu$, which replaces a thimble $\Delta$ by its image after a $2\pi$ twist in the base of the Lefschetz fibration. See Figure \ref{fig:Serre}.

\begin{figure}[htb]
\begin{center}
\includegraphics[scale=0.4]{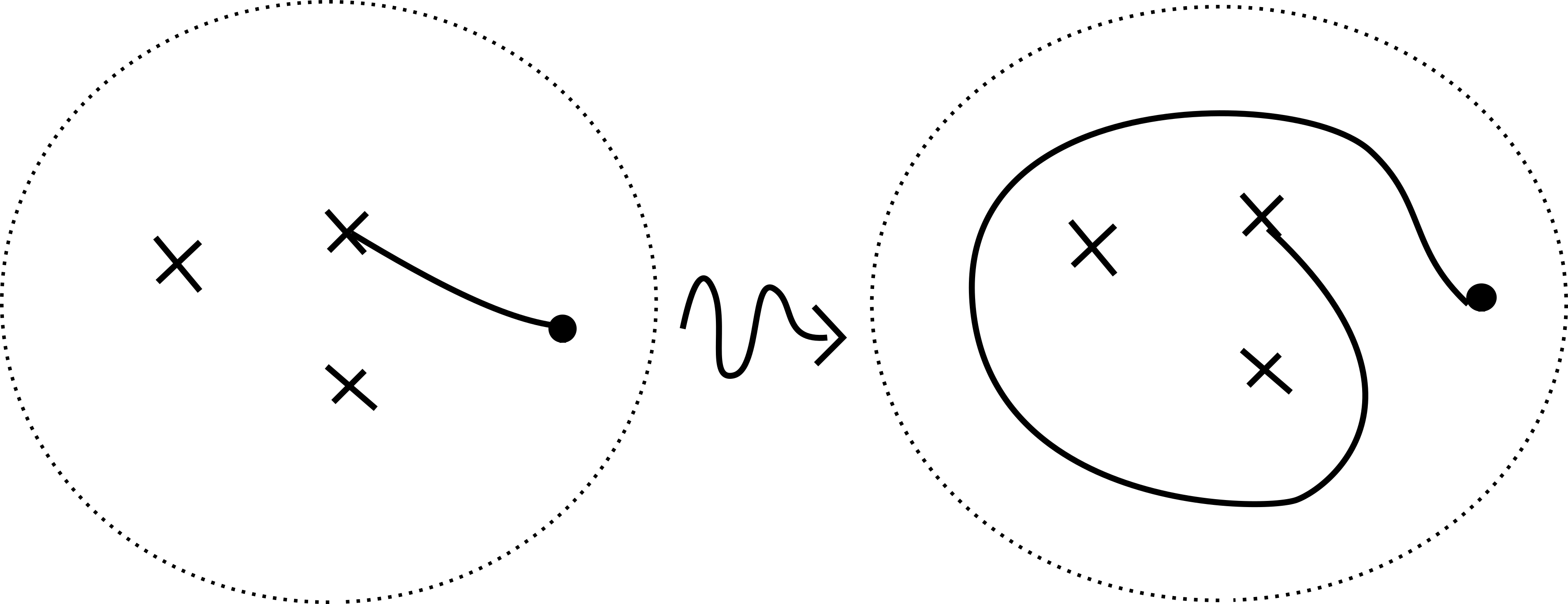}
% or [scale=0.85]
\caption{Serre functor: vanishing paths corresponding to $\Delta$ (left) and $\mu(\Delta)$ (right), in the basis of some Lefschetz fibration.}
\label{fig:Serre}
\end{center}
\end{figure}

\subsection{Localization using $\twvect(Y_{p,q,r})$ and $\twvect(D)$} 

\begin{lemma} \label{lem:sheaves_on_complement}
The dg category
  $W_{(\A_\tC, \B^+_\tC)}$ is quasi-isomorphic to $D^b \Coh (Y_{p,q,r} \backslash D)$. 
\end{lemma}

\begin{proof} 
Let $V = H^0 \left( \perf(\A_\tC) \right)$ as before. We have a quasi-isomorphism $D^b \Coh(Y_{p,q,r}) \cong V $. 
Consider the line bundle $L = \O_{Y_{p,q,r}}(D)$, with canonical section $s$, so that $D = s^{-1} (0)$. Let $F': D^b \Coh(Y_{p,q,r}) \to D^b \Coh(Y_{p,q,r})$ be the functor given by tensoring with $L^{-1}$, and let $T': F' \to \text{Id}$ be the natural transformation given by multiplying with $s$. 
As explained in  \cite[p.~87]{Seidel_subalgebras}, $D^b \Coh(Y_{p,q,r} \backslash D)$ is equivalent to the category obtained by localizing $V$ along $T'$. 

On the other hand, the functor $F = H^0 \left( \Phi_{\A^\vee_\tC[-2]  }  \right): V \to V$ is the $-2$ shift of the Serre functor; thus, up to a $-2$ shift, $F$ corresponds to tensoring with the canonical bundle $\mathcal{K}_{Y_{p,q,r}}$. 
This implies that the natural transformation $T: F \to \text{Id}$ is given by multiplying with some section of $\mathcal{K}^{-1}_{Y_{p,q,r}}$, say $t$. 
By the first paragraph, $W$ is equivalent to $D^b \Coh \left(Y_{p,q,r} \backslash t^{-1}(0) \right)$. 
On the other hand, recall that we have set things up so that the inclusion $\A_\tC \to \B^+_\tC$ is induced, on the level of morphisms, by inclusions of the form
\bq
\Gamma(U; \mathcal{H}om(E,F)) \subset \Gamma(U; \mathcal{H}om(E,F))[\varepsilon]
\eq
with $\partial \varepsilon = s$. Thus, by construction, $t = cs$ for some constant $c \in \C^\ast$. 
\end{proof}

\subsection{Conclusion} Putting Theorem \ref{thm:Abouzaid-Seidel} and Lemma \ref{lem:sheaves_on_complement} together yields the following.

\begin{theorem}\label{thm:iso3}
There is an equivalence of categories
\bq
D^b \mathcal{W} \left(\mathcal{T}_{p,q,r} \right) \cong D^b \left(Y_{p,q,r} \backslash D \right).
\eq
\end{theorem}

\section{Extensions and speculations}

\subsection{Images of some distinguished Lagrangians}
We record the images of certain distinguished compact Lagrangians under the mirror equivalence.

$\T_{p,q,r}$ contains several exact Lagrangian tori; a particularly noteworthy one is given by  restricting $\Xi$ to a contractible open subset of the base only containing the critical values for $\E_1, \E_2$ and $\E_3$, with total space $T^\ast T^2$. There are $(\C^\ast)^2$ choices of flat complex line bundles on the zero-section $T^2$, which induce a family of $(\C^\ast)^2$ objects in $\mathcal{W}(T^\ast T^2) \subset \mathcal{W} (\T_{p,q,r})$. Under homological mirror symmetry for $T^\ast T^2$, these objects correspond to skyscraper sheaves of the points of $(\C^\ast)^2 = \P^2 \backslash D_{\P^2}$, where $D_{\P^2}$ is the toric divisor on $\P^2$ \cite{Abouzaid_cotangent}. Thus, under the isomorphism of Theorem \ref{thm:iso3}, these objects correspond to skyscraper sheaves of the points of the strict transform of this $(\C^\ast)^2$ patch.

Let us also consider the vanishing cycles for $T_{p,q,r}$.
This discussion turns out to be more natural if we  re-order our distinguished collection of vanishing cycles: perform trivial mutations to get the ordered collection:
\bq
\cP_1' = \cP_p, \cP_1' = \cP_{p-1},\ldots, \cQ_1'= \cQ_q, \ldots,  \cQ_r' = \cQ_1, \cR'_1 =\cR_r,  \cR_r'=\cR_1, \E_1, \E_2, \E_3
\eq
There is a natural identification of the resulting directed Fukaya category, say $(\Fuk^{ \to})' (\Xi)$, with the previously studied $\Fuk^{\to}(\Xi)$ (with order-reversing shufflings among the critical points of types $\cP$, $\cQ$ and $\cR$), and one could have established the various equivalences of categories using it instead. 
 Each of the vanishing cycles of Figure \ref{fig:Dynkin_cycles} gives an object of $\mathcal{W}(\T_{p,q,r})$ (to be precise, we choose gradings so that all of the intersection points between them have grading zero). 
From their descriptions as matching cycles in $\Xi$, one can calculate that under Theorem \ref{thm:Abouzaid-Seidel}, each of the vanishing cycles for $T_{p,q,r}$ is the localization of a cone on two of the generators for $(\Fuk^{\to})'(\Xi)$; these correspond to two thimbles which glue together to form the matching cycles -- c.f.~\cite[Section 18]{Seidel_book}.
For instance, for $P_1$, take the thimbles $\cP'_p$ and $\cP'_{p-1}$. In $D^b \mathcal{W}(\T_{p,q,r})$, $P_1$ can be shown to be quasi-isomorphic to the image under localization of
\bq
\xymatrix{
\{  \cP'_{p-1} \ar[r]_{e^{\cP}_{p-1,p}} & \cP'_{p}
  \}.
}
\eq
A calculation shows that this gets mapped to $i_\ast \O_{\widetilde{E}_{\cP,p}}$ under Theorem \ref{thm:iso3}; more generally, $P_j$ corresponds to $i_\ast \O_{\widetilde{E}_{\cP, p-j+1}}$, $j=2, \ldots, p-1$, and similarly for the $Q_k$ and $R_l$. Moreover, $A$  corresponds to $i_\ast \O_{\widetilde{H}}$, and $B$ to $i_\ast (\O_{\widetilde{H}}(-1) )$.

\subsection{Restricting to compact Lagrangians}
Let $\Fuk(\mathcal{T}_{p,q,r})$ denote the Fukaya category of $\mathcal{T}_{p,q,r}$, defined as in \cite[Section 12]{Seidel_book}. Its objects are twisted complexes of Lagrangian branes, where we restrict ourselves to closed exact Lagrangian submanifolds. There is a natural full inclusion of categories $D^b \Fuk(\T_{p,q,r}) \subset D^b \mathcal{W}(\T_{p,q,r})$.  Let $D^b \Coh_{\text{cpt}}(Y_{p,q,r} \backslash D)$ denote the full subcategory of $D^b \Coh (Y_{p,q,r} \backslash D)$ with objects complexes of vector bundles which have (as complexes) compact support. We expect there to be a full inclusion $D^b \Fuk(\T_{p,q,r}) \subseteqq D^b \Coh_{\text{cpt}} (Y_{p,q,r} \backslash D)$, such that the following diagram commutes:
\begin{eqnarray}
 D^b \mathcal{W}(\T_{p,q,r}) & \cong & D^b (Y_{p,q,r} \backslash D) \\
\rotatebox{90}{$\subset$}\qquad  & & \qquad \rotatebox{90}{$\subset$} \\
D^b \Fuk(\T_{p,q,r}) & \subseteqq  & D^b \Coh_{\text{cpt}} (Y_{p,q,r} \backslash D)
\end{eqnarray}

\subsection{Restricting to cores}
Let $\Fuk_o(\mathcal{T}_{p,q,r})$ denote the full subcategory of $\Fuk(\mathcal{T}_{p,q,r})$ split-generated by vanishing cycles, and $D^b \Coh_o(Y_{p,q,r} \backslash D)$ denote the derived category of complexes of coherent sheaves with support (as complexes) on the triangular configuration of $-2$ curves in $Y_{p,q,r} \backslash D$. We expect the previous inclusion to specialize to
\bq
D^b \Fuk_o(\T_{p,q,r}) \subseteqq D^b \Coh_o (Y_{p,q,r} \backslash D) 
\eq
with the same compatibilities as before. As motivation for the notation, note that in the case of the three simple elliptic singularities, as discussed in the introduction, the variety $Y_{p,q,r} \backslash D$ is a rational elliptic surface with an $I_d$ fibre removed, where $d=3,2,1$, for, respectively, $(p,q,r)= (3,3,3), (4,4,2)$ and $(6,3,2)$. Moreover, there is the proper elliptic fibration
\bq
q: Y_{p,q,r} \backslash D \to \C
\eq
such that $q^{-1}(0)$ is precisely the triangular configuration of $-2$ curves.

\subsection{Singular affine structures}
Starting with $\Xi$, we can get a description of $\T_{p,q,r}$ as the result of attaching $p+q+r$ Weinstein handles to $D^\ast T^2$, the cotangent disc bundle over $T^2$: one handle for each of the critical points for the $\cP_i$, $\cQ_j$ and $\cR_k$. More precisely, we are attaching $p$, $q$ and $r$ handles to Legendrian lifts to $\partial D^\ast T^2$ of three curves on $T^2$. By e.g.~drawing $\cP_1$, $\cQ_1$ and $\cR_1$ onto Figure \ref{fig:MVCMcycles} and performing exact isotopies (first to arrange for each $L_i$ to agree almost everywhere with one of the other two, and second to displace $\cP_1$, $\cQ_1$ and $\cR_1$ onto the resulting skeleton), we see that each of these three curves in $T^2$ must intersect the other two once. By symmetry considerations, they can be chosen to be, for instance, curves of slopes $0, 1$ and $\infty$. 

This fits with expectations from the Gross--Siebert program, which suggests that $\mathcal{T}_{p,q,r}$ should admit a singular affine structure, with Lagrangian torus fibres and $p+q+r$ singular fibres: $p$ aligned vertically, where, with respect to the central fibre, a longitude has been pinched; $q$ aligned horizontally, where a meridian has been pinched; and $r$ aligned diagonally, where a curve of slope one has been pinched. 
(Dually, this should correspond to a singular affine structure on $Y_{p,q,r} \backslash D$ given by starting with the standard one on $(\C^\ast)^2$ and iteratively modifying it by Symington cuts \cite{Symington}.)
On a related note, using this presentation the spaces $\mathcal{T}_{p,q,r}$ fit into the framework of the work-in-progress \cite{STW}; we expect the cluster structure constructed therein to agree with the one of \cite{GHK_birational}.

\bibliography{bib}{}

\def\cprime{$'$}
\begin{thebibliography}{10}

\bibitem{Abouzaid-Seidel}
M.~Abouzaid and P.~Seidel.
\newblock {L}efschetz fibration methods in wrapped {F}loer cohomology.
\newblock in preparation.

\bibitem{Abouzaid-Smith}
M.~Abouzaid and I.~Smith.
\newblock {K}hovanov homology from {F}loer cohomology.
\newblock arXiv:1504.01230.

\bibitem{Abouzaid_cotangent}
Mohammed Abouzaid.
\newblock A cotangent fibre generates the {F}ukaya category.
\newblock {\em Adv. Math.}, 228(2):894--939, 2011.

\bibitem{AAEKO}
Mohammed Abouzaid, Denis Auroux, Alexander~I. Efimov, Ludmil Katzarkov, and
  Dmitri Orlov.
\newblock Homological mirror symmetry for punctured spheres.
\newblock {\em J. Amer. Math. Soc.}, 26(4):1051--1083, 2013.

\bibitem{Abouzaid-Auroux-Katzarkov}
Mohammed Abouzaid, Denis Auroux, and Ludmil Katzarkov.
\newblock Lagrangian fibrations on blowups of toric varieties and mirror
  symmetry for hypersurfaces.
\newblock arXiv:1205.0053.

\bibitem{Abouzaid-Seidel_wrapped}
Mohammed Abouzaid and Paul Seidel.
\newblock An open string analogue of {V}iterbo functoriality.
\newblock {\em Geom. Topol.}, 14(2):627--718, 2010.

\bibitem{ACampo}
Norbert A'Campo.
\newblock Real deformations and complex topology of plane curve singularities.
\newblock {\em Ann. Fac. Sci. Toulouse Math. (6)}, 8(1):5--23, 1999.

\bibitem{Akbulut-Arikan}
Selman Akbulut and M.~Firat Arikan.
\newblock {S}tabilizations via {L}efschetz {F}ibrations and {E}xact {O}pen
  {B}ooks.
\newblock arXiv:1112.0519.

\bibitem{Arnold_VI}
V.~I. Arnold, V.~V. Goryunov, O.~V. Lyashko, and V.~A. Vasil{\cprime}ev.
\newblock {\em Singularity theory. {I}}.
\newblock Springer-Verlag, Berlin, 1998.
\newblock Translated from the 1988 Russian original by A. Iacob, Reprint of the
  original English edition from the series Encyclopaedia of Mathematical
  Sciences [{{\i}t Dynamical systems. VI}, Encyclopaedia Math. Sci., 6,
  Springer, Berlin, 1993; MR1230637 (94b:58018)].

\bibitem{Bondal-Orlov}
A.~Bondal and D.~Orlov.
\newblock Derived categories of coherent sheaves.
\newblock In {\em Proceedings of the {I}nternational {C}ongress of
  {M}athematicians, {V}ol. {II} ({B}eijing, 2002)}, pages 47--56. Higher Ed.
  Press, Beijing, 2002.

\bibitem{Burban-Drozd}
Igor Burban and Yurij Drozd.
\newblock Coherent sheaves on rational curves with simple double points and
  transversal intersections.
\newblock {\em Duke Math. J.}, 121(2):189--229, 2004.

\bibitem{CPS}
M.~Carl, M.~Pumperla, and S.~Siebert.
\newblock A tropical view on {L}andau-{G}inzburg models.
\newblock preprint available at:
  http://www.math.uni-hamburg.de/home/siebert/research/research.html.

\bibitem{Chan-Ueda}
Kwokwai Chan and Kazushi Ueda.
\newblock Dual torus fibrations and homological mirror symmetry for
  {$A_n$}-singlarities.
\newblock {\em Commun. Number Theory Phys.}, 7(2):361--396, 2013.

\bibitem{Drozd-Greuel}
Yuri~A. Drozd and Gert-Martin Greuel.
\newblock Tame and wild projective curves and classification of vector bundles.
\newblock {\em J. Algebra}, 246(1):1--54, 2001.

\bibitem{Etgu-Lekili}
T.~Etg{\"u} and Y~Lekili.
\newblock {K}oszul duality patterns in {F}loer theory.
\newblock arXiv:1502.07922.

\bibitem{Giroux-Pardon}
Emmanuel Giroux and John Pardon.
\newblock {E}xistence of {L}efschetz fibrations on {S}tein and {W}einstein
  domains.
\newblock arXiv:1411.6176.

\bibitem{GHK_logCY}
M.~Gross, P.~Hacking, and S.~Keel.
\newblock Mirror symmetry for log {C}alabi-{Y}au surfaces {I}.
\newblock arXiv:1106.4977.

\bibitem{GHKv1}
M.~Gross, P.~Hacking, and S.~Keel.
\newblock Mirror symmetry for log {C}alabi-{Y}au surfaces {I}.
\newblock arXiv:1106.4977v1.

\bibitem{GHK_birational}
Mark Gross, Paul Hacking, and Sean Keel.
\newblock Birational geometry of cluster algebras.
\newblock {\em Algebr. Geom.}, 2(2):137--175, 2015.

\bibitem{GS1}
Mark Gross and Bernd Siebert.
\newblock Mirror symmetry via logarithmic degeneration data. {I}.
\newblock {\em J. Differential Geom.}, 72(2):169--338, 2006.

\bibitem{GS2}
Mark Gross and Bernd Siebert.
\newblock Mirror symmetry via logarithmic degeneration data, {II}.
\newblock {\em J. Algebraic Geom.}, 19(4):679--780, 2010.

\bibitem{GS3}
Mark Gross and Bernd Siebert.
\newblock From real affine geometry to complex geometry.
\newblock {\em Ann. of Math. (2)}, 174(3):1301--1428, 2011.

\bibitem{Huybrechts}
D.~Huybrechts.
\newblock {\em Fourier-{M}ukai transforms in algebraic geometry}.
\newblock Oxford Mathematical Monographs. The Clarendon Press, Oxford
  University Press, Oxford, 2006.

\bibitem{Keating14}
A.~Keating.
\newblock Lagrangian tori in four-dimensional {M}ilnor fibres.
\newblock arXiv:1405.0744.

\bibitem{Kontsevich_ENS}
M.~Kontsevich.
\newblock Lectures at {ENS} {P}aris, {S}pring 1998.
\newblock Set of notes taken by J.~Bellaiche, J.-F.~Dat, I.~Marin, G.~Racinet
  and H.~Randriambololona.

\bibitem{Kontsevich_ICM}
Maxim Kontsevich.
\newblock Homological algebra of mirror symmetry.
\newblock In {\em Proceedings of the {I}nternational {C}ongress of
  {M}athematicians, {V}ol.\ 1, 2 ({Z}\"urich, 1994)}, pages 120--139.
  Birkh\"auser, Basel, 1995.

\bibitem{KS1}
Maxim Kontsevich and Yan Soibelman.
\newblock Homological mirror symmetry and torus fibrations.
\newblock In {\em Symplectic geometry and mirror symmetry ({S}eoul, 2000)},
  pages 203--263. World Sci. Publ., River Edge, NJ, 2001.

\bibitem{KS2}
Maxim Kontsevich and Yan Soibelman.
\newblock Affine structures and non-{A}rchimedean analytic spaces.
\newblock In {\em The unity of mathematics}, volume 244 of {\em Progr. Math.},
  pages 321--385. Birkh\"auser Boston, Boston, MA, 2006.

\bibitem{Lekili-Perutz}
Y.~Lekili and T.~Perutz.
\newblock Arithmetic mirror symmetry for the 2-torus.
\newblock arXiv:1211.4632.

\bibitem{Lekili-Polishchuk}
Y.~Lekili and A.~Polishchuk.
\newblock Arithmetic mirror symmetry for genus 1 curves with $n$ marked points.
\newblock arXiv:1601.06141.

\bibitem{Lekili-Perutz-PNAS}
Yank{\i} Lekili and Timothy Perutz.
\newblock Fukaya categories of the torus and {D}ehn surgery.
\newblock {\em Proc. Natl. Acad. Sci. USA}, 108(20):8106--8113, 2011.

\bibitem{Looijenga}
Eduard Looijenga.
\newblock Rational surfaces with an anticanonical cycle.
\newblock {\em Ann. of Math. (2)}, 114(2):267--322, 1981.

\bibitem{Lunts-Orlov}
Valery~A. Lunts and Dmitri~O. Orlov.
\newblock Uniqueness of enhancement for triangulated categories.
\newblock {\em J. Amer. Math. Soc.}, 23(3):853--908, 2010.

\bibitem{Orlov}
D.~O. Orlov.
\newblock Projective bundles, monoidal transformations, and derived categories
  of coherent sheaves.
\newblock {\em Izv. Ross. Akad. Nauk Ser. Mat.}, 56(4):852--862, 1992.

\bibitem{Saito}
Kyoji Saito.
\newblock Einfach-elliptische {S}ingularit\"aten.
\newblock {\em Invent. Math.}, 23:289--325, 1974.

\bibitem{Seidel_quartic}
P~Seidel.
\newblock Homological mirror symmetry for the quartic surface.
\newblock {\em Mem. Amer. Math. Soc.}, 236(1116):vii+129, 2015.

\bibitem{Seidel_FL2}
Paul Seidel.
\newblock {F}ukaya {$A_\infty$} structures associated to {L}efschetz
  fibrations. {II}.
\newblock arXiv:1404.1352.

\bibitem{Seidel_lectures}
Paul Seidel.
\newblock Lecture notes on {C}ategorical {D}ynamics and {S}ymplectic
  {T}opology.
\newblock Available at math.mit.edu/~seidel.

\bibitem{Seidel_gradedlagrangians}
Paul Seidel.
\newblock Graded {L}agrangian submanifolds.
\newblock {\em Bull. Soc. Math. France}, 128(1):103--149, 2000.

\bibitem{Seidel_mvcm}
Paul Seidel.
\newblock More about vanishing cycles and mutation.
\newblock In {\em Symplectic geometry and mirror symmetry ({S}eoul, 2000)},
  pages 429--465. World Sci. Publ., River Edge, NJ, 2001.

\bibitem{Seidel_vcm}
Paul Seidel.
\newblock Vanishing cycles and mutation.
\newblock In {\em European {C}ongress of {M}athematics, {V}ol. {II}
  ({B}arcelona, 2000)}, volume 202 of {\em Progr. Math.}, pages 65--85.
  Birkh\"auser, Basel, 2001.

\bibitem{Seidel_subalgebras}
Paul Seidel.
\newblock {$A_\infty$}-subalgebras and natural transformations.
\newblock {\em Homology, Homotopy Appl.}, 10(2):83--114, 2008.

\bibitem{Seidel_book}
Paul Seidel.
\newblock {\em Fukaya categories and {P}icard-{L}efschetz theory}.
\newblock Zurich Lectures in Advanced Mathematics. European Mathematical
  Society (EMS), Z\"urich, 2008.

\bibitem{Seidel_FLI}
Paul Seidel.
\newblock Fukaya {$A_\infty$}-structures associated to {L}efschetz fibrations.
  {I}.
\newblock {\em J. Symplectic Geom.}, 10(3):325--388, 2012.

\bibitem{STW}
Vivek Shende, David Treumann, and Harold Williams.
\newblock Symplectic geometry of cluster algebras.
\newblock in preparation.

\bibitem{Symington}
Margaret Symington.
\newblock Four dimensions from two in symplectic topology.
\newblock In {\em Topology and geometry of manifolds ({A}thens, {GA}, 2001)},
  volume~71 of {\em Proc. Sympos. Pure Math.}, pages 153--208. Amer. Math.
  Soc., Providence, RI, 2003.

\end{thebibliography}
\bibliographystyle{plain}
\end{document}